\documentclass[11pt,reqno,a4paper]{amsart}

\usepackage[T1]{fontenc}
\usepackage[utf8]{inputenc}
\usepackage{lmodern}
\usepackage{microtype}
\usepackage{geometry}
\usepackage{amsmath,amssymb,amsthm,mathtools,mathrsfs}
\usepackage{enumitem}
\usepackage[numbers,sort&compress]{natbib}
\usepackage[colorlinks=true,linkcolor=blue,citecolor=blue,urlcolor=blue]{hyperref}
\usepackage{aliascnt}
\usepackage[nameinlink,noabbrev]{cleveref}

\allowdisplaybreaks
\setlist[itemize]{topsep=4pt,itemsep=2pt,parsep=1pt}
\setlist[enumerate]{topsep=4pt,itemsep=2pt,parsep=1pt}
\numberwithin{equation}{section}
\newtheorem{theorem}{Theorem}[section]
\newaliascnt{proposition}{theorem}
\newtheorem{proposition}[proposition]{Proposition}
\aliascntresetthe{proposition}
\newaliascnt{lemma}{theorem}
\newtheorem{lemma}[lemma]{Lemma}
\aliascntresetthe{lemma}
\newaliascnt{corollary}{theorem}
\newtheorem{corollary}[corollary]{Corollary}
\aliascntresetthe{corollary}
\theoremstyle{definition}
\newaliascnt{definition}{theorem}
\newtheorem{definition}[definition]{Definition}
\aliascntresetthe{definition}
\newaliascnt{assumption}{theorem}
\newtheorem{assumption}[assumption]{Assumption}
\aliascntresetthe{assumption}
\theoremstyle{remark}
\newaliascnt{remark}{theorem}
\newtheorem{remark}[remark]{Remark}
\aliascntresetthe{remark}

\Crefname{theorem}{theorem}{theorems}
\Crefname{proposition}{proposition}{propositions}
\Crefname{lemma}{lemma}{lemmas}
\Crefname{corollary}{corollary}{corollaries}
\Crefname{definition}{definition}{definitions}
\Crefname{assumption}{assumption}{assumptions}
\Crefname{remark}{remark}{remarks}
\Crefname{theorem}{Theorem}{Theorems}
\Crefname{proposition}{Proposition}{Propositions}
\Crefname{lemma}{Lemma}{Lemmas}
\Crefname{corollary}{Corollary}{Corollaries}
\Crefname{definition}{Definition}{Definitions}
\Crefname{assumption}{Assumption}{Assumptions}
\Crefname{remark}{Remark}{Remarks}

\newcommand{\R}{\mathbb R}
\newcommand{\T}{\mathbb T}
\newcommand{\E}{\mathbb E}
\newcommand{\PP}{\mathbb P}

\newcommand{\Ran}{\operatorname{Ran}}

\newcommand{\Ent}{\operatorname{Ent}}

\newcommand{\dd}{\,\mathrm d}

\newcommand{\ip}[2]{\left\langle #1,#2\right\rangle}

\title[Spectral gap for stochastic PE]{Spectral gap for the stochastic primitive equations under degenerate Brownian forcing}

\author{Quyuan Lin}
\address{	School of Mathematical and Statistical Sciences, 
Clemson University, 
Clemson, SC 29634, USA} \email{quyuanl@clemson.edu}

\author{Rongchang Liu}
\address{School of Mathematics, Sichuan University, Chengdu 610064, China}
\email{rcliu@scu.edu.cn}

\author{Kening Lu}
\address{School of Mathematics, Sichuan University, Chengdu 610064, China}
\email{keninglu@scu.edu.cn}

\date{}
\thanks{This work was supported by the Fundamental Research Funds for the Central Universities, the National Natural Science Foundation of China (Grant Nos.~12090010 and 12090013) and the Simons Foundation (SFI-MPS-TSM-00013384).}

\begin{document}

\begin{abstract}
We prove a weighted Wasserstein spectral gap 
for the three-dimensional primitive equations on $\mathscr E=\{u\in H^1\cap L^8: \partial_zu\in L^4\}$ 
under finite-rank additive Brownian
forcing that is phase-complete finite Fourier and satisfies a
quantitative quadratic saturation condition. In particular, four
real forcing directions suffice and this number is optimal within the
phase-complete class. Consequently, the system is uniquely ergodic and exponentially mixing in the $H^1$ topology. 

The main difficulties stem from the lack of sufficiently strong moment bounds and the highly degenerate nature of the noise, placing the
problem outside the direct reach of standard methods.  We introduce a new quantitative stable--compact asymptotic coupling
criterion that captures the delicate balance between the available
control of the growth of initial perturbations and the accuracy--cost
of compensating them by perturbing the noise. For the primitive equations, this leads to a quantitative 
infinite-dimensional H\"ormander analysis requiring tracking the degree and coefficient cost of  
Lie-bracket polynomials with the Fourier frequency.
\end{abstract}

\maketitle
\enlargethispage{2pt}
\tableofcontents
\section{Introduction}
In this paper, we investigate the Markov dynamics of the three-dimensional stochastic
primitive equations. On $\T^3=(\R/2\pi\mathbb Z)^3$, the equation reads 
\begin{align}\label{eq:intro-pe}
\begin{split}
 \dd v
 +\bigl[(v\cdot\nabla_h)v+w\partial_zv
        -\nu\Delta v+\nabla_h p\bigr]\dd t
 &=\sum_{k=1}^m g_k\,\dd W_t^k,\\
 \partial_zp&=0,\\
 \nabla_h\cdot v+\partial_zw&=0.
\end{split}
\end{align}
Here $v=(v_1,v_2)$ is the horizontal velocity, $w$ is the vertical
velocity, $p$ is the hydrostatic pressure,
$\nabla_h=(\partial_{x_1},\partial_{x_2})$, and
$\Delta=\Delta_h+\partial_z^2$.  We impose the standard periodic
hydrostatic symmetry: $v$ is even and $w$ is odd in the vertical
variable. Thus $w$ can be recovered from $v$ through the incompressibility
condition, and \eqref{eq:intro-pe} can be reformulated as an evolution
equation for $v$; see \eqref{eq:pe}.
The viscosity $\nu>0$ is arbitrary and fixed,
$W^1,\ldots,W^m$ are independent real Brownian motions, and $g_k$ are
the noise coefficients. The primitive equations arise as the hydrostatic small-aspect-ratio
limit of the Navier--Stokes equations
\cite{azerad2001mathematical,li2019primitive} and are a fundamental model
for large-scale oceanic and atmospheric dynamics. Unlike the
three-dimensional Navier--Stokes equations, the three-dimensional viscous primitive equations are globally well-posed at the strong level \cite{cao-titi,GuoHuang2009,GHKVZ,DGTZ,HK},
although uniqueness of Leray--Hopf weak
solutions remains open
\cite{BinzIida2025}.

We are interested in statistical dynamics for the Markov system generated by strong solutions of \eqref{eq:intro-pe} in $H^1$, particularly  when the forcing acts only through a few low Fourier modes. This setting naturally 
reflects the scale separation underlying Kolmogorov's K41 theory
\cite{Kolmogorov1941a,Kolmogorov1941b} on turbulence, where energy
is injected at low frequencies, well separated from the dissipative
scales, allowing nonlinear interactions to sustain a nonequilibrium
stationary flux.  From a complementary dynamical systems perspective initiated by Kolmogorov \cite{Sinai1989} and further developed by Sinai, Ruelle and others  through the theory of Sinai--Ruelle--Bowen measures \cite{EckmannRuelle1985,Ruelle1978}, distinguished stationary measures
are natural candidates for describing the nonequilibrium
stationary flux under suitable scaling limits; see also \cite{BCPW3D,BCPW2D}. Establishing unique ergodicity and quantitative mixing for \eqref{eq:intro-pe} may therefore be viewed as a first step toward such a statistical description of hydrostatic turbulence.

The existence and regularity of invariant measures were established in
\cite{GHKVZ}. For bounded noise, uniqueness and exponential mixing were
proved in the nondegenerate case in
\cite{Chueshov2014,Boulvard2022}, and in a degenerate finite-dimensional
regime in \cite{BGN}. For nondegenerate white in time noise,
\cite{Agresti} established unique ergodicity, while the corresponding degenerate problem has remained open. Our main result resolves this problem for a natural class of finite-rank
Fourier forcing. Let $P_t$ denote the Markov semigroup generated by
\eqref{eq:intro-pe} on $H^1$.

\vskip0.05in
\noindent {\bf Main Theorem}. \label{mainthm}
{\it Assume that the forcing range is phase-complete finite Fourier and
satisfies the quantitative quadratic saturation condition of
\Cref{def:phase-complete,def:fourier-certificate}. Then the Markov
semigroup $P_t$ admits a spectral gap on the regular phase space $\mathscr E=\{u\in H^1\cap L^8: \partial_zu\in L^4\}$ in an explicit weighted
logarithmic Wasserstein distance built over $1\wedge\|\cdot\|_{H^1}$. Consequently, the primitive system \eqref{eq:intro-pe} has a unique invariant measure and is exponentially mixing in the bounded Wasserstein distance induced
by $1\wedge\|\cdot\|_{H^1}$.}

\vskip0.05in

See \Cref{thm:main} and \Cref{thm:exponential-mixing} for the precise statement,
including the explicit transportation cost and its Lyapunov weight. The coupling is constructed using the $L^2$ comparison distance and parabolic smoothing then transfers the resulting weighted
spectral gap to $H^1$. 
In particular, four real forcing directions suffice, and this number is
optimal within the phase-complete class; see
\Cref{rem:low-rank-example}. The main difficulties hinge on the lack of
sufficiently strong moment bounds and the highly degenerate nature of
the noise \cite{GHKVZ,BGN,Agresti}. For the primitive equations, we obtain an apparently new fractional logarithmic amplification moment: there is $\theta\in(0,1)$ such that for any two horizontal velocities $u_t,v_t$ solving \eqref{eq:intro-pe}, 
\begin{align}\label{eq:intro-frac-amplification}
 \sup_{t\in[0,T]}\|u_t-v_t\|_{L^2}
 \le e^{H_T}
 \|u_0-v_0\|_{L^2}, \text{ and } \E H_T^\theta\le C_TU(v_0)
\end{align}
where $U$ is the mixed Lyapunov function in \eqref{eq:mixing-lyapunov}.  This improves the logarithmic  moment bounds of
\cite{GHKVZ}, but in the degenerate noise setting it still falls
below the moment regime in which the standard spectral-gap mechanisms
can be applied directly.

From a general perspective, our method lies within the weak Harris theory developed in \cite{HM08,HMS}, rooted in the classical ideas of
Doeblin \cite{Doeblin1940} and Harris \cite{Harris1956}: local contraction controls
small scales, irreducibility makes recurrent Lyapunov sublevels
$d$-small, and the Lyapunov drift controls the large-weight regime. The distinction lies in the local input. Rather than assuming a
pre-existing contracting distance or deriving one from a gradient
estimate whose conversion into a uniform
Wasserstein contraction becomes challenging under the weak moment bounds
available for system \eqref{eq:intro-pe}, we introduce a quantitative stable--compact
asymptotic coupling criterion that captures the balance between the
available control of the growth of initial perturbations and the
accuracy--cost of compensating them by perturbing the noise. Realizing
this local mechanism for \eqref{eq:intro-pe} in turn leads to an
infinite-dimensional H\"ormander analysis whose Lie-bracket polynomials must be
tracked quantitatively with the Fourier resolution.

\subsection{Quantitative stable--compact compensation}
Let $J$
denote the derivative of the solution map with respect to the initial
state and let $A$ denote its Malliavin derivative. 
A useful common language for local contraction in several approaches to degenerate
stochastic dynamics is the decomposition
\begin{align}\label{eq:intro-JARS}
 J=AR+S,
\end{align}
where $R$ is the repair operator producing a perturbation of the noise
that compensates the initial perturbation up to a prescribed accuracy, and $S$ is the residual.  The
essential differences between the existing methods lie both in how
\eqref{eq:intro-JARS} is obtained and in how it is converted into a
finite-distance contraction.

In the hypoelliptic theory of 
\cite{HM06,HM}, a regularized inverse of the Malliavin covariance
compensates the finite unstable directions.  The remaining
directions are driven into the dissipative high frequencies, so that
the residual becomes small after propagation by the dynamics.
Approximate Malliavin integration by parts then converts this
linearized compensation into an asymptotic gradient estimate.  This
is the mechanism underlying asymptotic strong Feller regularization
and, together with the Wasserstein Harris theory developed in
\cite{HM08,HMS}, quantitative convergence to equilibrium.   The bounded noise approach of  \cite{KNS} realizes the same principle differently.
Qualitative controllability of the linearized noise map, together with
compactness on a large compact part of the bounded noise space, gives
a finite-dimensional approximate right inverse.  The target accuracy
is fixed sufficiently small, and the nonlinear remainder is then
absorbed by restricting the initial separation.  The resulting repair
has a fixed linear strength.

A related stable--compact viewpoint was developed in our recent works
\cite{LLLProjective,LLWave,LLSNLS}. For the projective Navier--Stokes
process, compactness of the state derivative and density of the
Malliavin range yield a finite-dimensional noise correction leaving a
residual with negative mean logarithmic growth. For the damped wave
equation, the same principle gives a finite-distance contraction and a
weighted Wasserstein spectral gap. In the stochastic NLS setting, only
weak logarithmic amplification moments are available and a fixed-strength
repair then leads to polynomial
mixing through a renewal--reset argument. 

The regime considered here requires a different quantitative
realization of \eqref{eq:intro-JARS}.  On a compact regular core we
prove, for every sufficiently small accuracy $\delta$, the quantitative Malliavin
observability inequality 
\begin{align}\label{eq:intro-observability}
 JJ^*\le C_\delta^2AA^*+\delta^2I,
 \qquad
 C_\delta\lesssim\delta^{-b}
\end{align}
for some $b>0$, which in turn yields  the stable--compact decomposition 
\begin{align}\label{eq:intro-stable-compact-decomposition}
 J=A R_\delta+S_\delta,
 \qquad
 \|S_\delta\|\lesssim\delta,
 \qquad
 \|R_\delta\|\lesssim\delta^{-b},
\end{align}
where $R_\delta$ takes values in a finite-dimensional subspace of the
Cameron--Martin space. 
Thus \eqref{eq:intro-observability} provides an explicit
accuracy--cost relation: an initial tangent perturbation can be
compensated to accuracy $\delta$ with Cameron--Martin cost of order $\delta^{-b}$.

The dependence of the accuracy on the incoming scale is essential. In stronger regimes \cite{HM06,HM,KNS,LLLProjective,LLWave}, where the amplification can be controlled in a
linear or power-type transportation gauge, a fixed repair
$r\mapsto cr$ for some $c\in(0,1)$ already yields a uniform local contraction.  Here, however, the
fractional moment bound \eqref{eq:intro-frac-amplification} forces a logarithmic gauge, in which
such a fixed repair becomes asymptotically invisible. Indeed, let $r=d(u,v)=1\wedge\|u-v\|_{L^2}$ be the incoming distance and suppose that a repair  only
gave local contraction $r\mapsto cr$. In the logarithmic gauge
\begin{align*}
 \ell(r)=[1+\log(1/r)]^{-q}
\end{align*}
with $q>0$, this implies 
\begin{align}\label{eq:intro-fixed-repair-failure}
 \ell(cr)/\ell(r)\longrightarrow1
 \qquad(r\downarrow0),
\end{align}
showing no uniform local contraction gain.  

Therefore repair strength must be measured relative to the gauge dictated by amplification tails. Here we take $\delta=r^a$ so the linear residual $(J-AR_{\delta})(u-v)$ in
\eqref{eq:intro-stable-compact-decomposition} is of size
$\delta r=r^{1+a}$, while the quadratic Taylor remainder is of size
$C_\delta^2r^2$.  For sufficiently small $a>0$ the latter is of the
same or smaller order.  The Gaussian change of law has a vanishing
error $O(r^\varkappa)$ for some $\varkappa>0$, and hence the exact coupling $(X',Y')$ satisfies, in the comparison metric $d$, 
\begin{align}\label{eq:intro-local-repair}
\begin{split}
 \PP\{d(X',Y')\le r^{1+a}\}&\ge p,\\
 \PP\{d(X',Y')>e^{C+H_T}r\}&\le Cr^\varkappa,
\end{split}
\end{align}
for constants $p,C>0$. 
The superlinear map $r\mapsto r^{1+a}$ multiplies
$\log(1/r)$ by $1+a$, and consequently we obtain the local contraction 
\begin{align*}
 \ell(r^{1+a})/\ell(r)
 \longrightarrow(1+a)^{-q}<1,
\end{align*}
which is precisely the gain that is missing in
\eqref{eq:intro-fixed-repair-failure}. Note that the exceptional event
$H_T\gtrsim\log(1/r)$  has probability of order at most 
$[\log(1/r)]^{-\theta}$ by \eqref{eq:intro-frac-amplification}.  Choosing $q<\theta$ makes this error smaller
than the logarithmic transportation scale.  This balance between
\eqref{eq:intro-local-repair} and
\eqref{eq:intro-frac-amplification} is one of the cores of the abstract
criterion in \Cref{sec:coupling}. Once the local logarithmic
contraction has been obtained, the globalization by Lyapunov
recurrence and a Wasserstein small-set argument follows the weak
Harris architecture of \cite{HM08,HMS}. For \eqref{eq:intro-pe}, the required Lyapunov structure follows from the same triangular system of estimates used in deriving \eqref{eq:intro-frac-amplification}, while the uniform irreducibility follows from the standard approximate controllability argument \cite{AgrachevSarychev2005,AgrachevSarychev2006,GHHM}.

There are also two important distinctions from the Malliavin-gradient
route of \cite{HM06,HM08,HM}.  First, \eqref{eq:intro-stable-compact-decomposition} is a
fixed-time pathwise operator estimate on a compact positive probability
core; the residual is made small by choosing the accuracy, rather than
by iterating the dynamics until the residual decays.  Second, rather than using the compensation only at the infinitesimal
level to derive a Malliavin integration-by-parts estimate, we realize
$R_\delta$ at finite displacement as an actual finite-dimensional
Cameron--Martin transformation of the Wiener path.  A Gaussian
change-of-variables estimate, followed by maximal coupling, turns this
transformation into an exact coupling of the two endpoint laws.  In
this sense the Malliavin-gradient mechanisms move derivatives between
the state and the noise, whereas the present argument moves the
trajectories themselves. 

We also mention that, although the logarithmic amplification for the
stochastic damped NLS in \cite{LLSNLS} admits polynomial moment bounds
of arbitrary order, a quantitative Malliavin observability estimate of
the form \eqref{eq:intro-observability} is not available by the present
argument, since the damped NLS lacks the parabolic high-frequency
dissipation used here to convert low mode observability into a polynomial
accuracy--cost estimate as in \Cref{sec:quantitative-Malliavin-observability}.

\subsection{A quantitative infinite-dimensional H\"ormander analysis}
For the primitive equations, proving
\eqref{eq:intro-observability} is itself a quantitative
H\"ormander problem in which the required bracket order grows with the
Fourier resolution.  If $p_t=J_{t,T}^*\xi$ is the terminal adjoint where $J_{t,T}$ is the solution operator of the linearization, then the adjoint of the Malliavin derivative $\mathcal A_T$ at time $T$ satisfies 
\begin{align*}
 (\mathcal A_T^*\xi)(t)=Q^*p_t,
\end{align*}
where $Q$ is the noise operator  as in \eqref{eq:noise-operator}.  Thus the desired operator inequality \eqref{eq:intro-observability} is equivalent to observing
$p_0$ through the finitely many scalar paths $Q^*p$.  Qualitative
density of the Malliavin range is not enough: in order to choose
$\delta=r^a$ in the coupling argument one must control explicitly the
cost of observing successively finer frequencies.

We therefore develop a quantitative observation argument with three components.  First,
the Fourier algebra of the primitive equations provides, for every
unit hydrostatic Fourier mode below frequency $N$, a representation by Lie words generated from the constant noise fields with Brownian parabolic degree $L_N$ and total coefficient cost $A_N$ bounded as 
\begin{align}\label{eq:intro-lie-complexity}
 L_N\le CN, \quad A_{N} \le e^{CN}. 
\end{align}
Second, the scalar
adjoint observables generated by these Lie words admit controlled
rough-path expansions.  Rather than propagating smallness recursively
through Wiener polynomials as in \cite{MattinglyPardoux,HM,FGHRT}, we use simultaneous nondegeneracy of
truncated Brownian signatures to recover all Lie coefficients from the
observed paths.  The resulting low frequency estimate is
\begin{align*}
 \|P_Np\|_{L^2(0,T;L^2)}
 \le
 e^{CN^2}\|Q^*p\|_{L^2(0,T;\mathbb R^m)}
 +e^{-BN^2}\|\xi\|_{L^2}.
\end{align*}
Third, the parabolic high-frequency
dissipation of the backward adjoint estimate then turns control of
$P_Np$ on a positive time interval into control of the full initial
adjoint $p_0$ and therefore 
\begin{align*}
 \|J_{0,T}^*\xi\|_{L^2}^2
 \le
 e^{CN^2}\|\mathcal A_T^*\xi\|_{L^2(0,T;\mathbb R^m)}^2
 +Ce^{-cN^2}\|\xi\|_{L^2}^2.
\end{align*}
Taking $N^2$ proportional to $|\log\delta|$ gives exactly the
polynomial accuracy--cost estimate
\eqref{eq:intro-observability}.  The Brownian-signature argument may
be viewed as a quantitative pathwise counterpart of the non-adapted
Norris mechanism in \cite{HM} and the pathwise Norris mechanism in \cite{GerasimovicsHairer}, adapted here to a
bracket depth that grows with the spatial resolution.

The Fourier representations entering
\eqref{eq:intro-lie-complexity} require a new geometric analysis of the projective holonomy generated by barotropic polarization transport in \Cref{sec:Quantitative-Fourier-generation}.
At a positive vertical frequency the hydrostatic polarization fibre is
two-dimensional, while a nonzero barotropic fibre is one-dimensional.
Interaction with a barotropic mode transports polarization along
horizontal lattice paths by an explicit family of invertible
$2\times2$ matrices, thus generating a genuine projective
holonomy group through closed paths.  Irreducibility of this holonomy creates the full polarization fibre except in a single two-direction geometry, where an invariant line
persists and is overcome by an additional quadratic interaction.   Once full polarization is
available, the Euclidean algorithm in the vertical frequencies and
admissible horizontal paths then generate every hydrostatic Fourier mode.
Keeping track of this construction gives \eqref{eq:intro-lie-complexity}. In particular, we obtain
the sharp minimal forcing counts as in \Cref{rem:low-rank-example}.

The quantitative H\"ormander mechanism above and the logarithmic
transportation geometry 
together reveal a different regime of asymptotic regularization for
infinite-dimensional stochastic systems. The hypoelliptic theory of \cite{HM06,HM08,HM} produces an asymptotically Lipschitz regularity,
while the stable--compact mechanism for the damped wave equation
\cite{LLWave} yields, on bounded Lyapunov sets, a H\"older type modulus
$r^{\beta_{\mathrm W}}$ for some $\beta_{\mathrm W}\in(0,1)$.  In the present weak moment regime the natural
modulus is instead logarithmic,
\begin{align*}
 r
 \ll r^{\beta_{\mathrm W}}
 \ll [1+\log(1/r)]^{-q/2},
 \qquad r\downarrow0.
\end{align*}
Thus the primitive equations operate at a strictly weaker local
regularity scale; see
\Cref{rem:asymptotic-regularization-comparison} for more details. We
emphasize, however, a difference in logical direction. In \cite{HM06,HM08,HM}, asymptotic gradient estimates provide the
regularization input used to derive a spectral gap. Here the logarithmic transportation geometry is constructed
directly from the quantitative stable--compact mechanism, while the corresponding asymptotic regularization
estimate  is a consequence of the resulting spectral gap  as in \cite{LLWave}.

\section{Setting and main results}\label{sec:settings}

In this section we specify necessary settings and then state the main result. The abstract spectral gap criterion with its proof is given in the next section.  Let $\T^3=(\R/2\pi\mathbb Z)^3$, with coordinates $(\boldsymbol x,z)$ where $\boldsymbol x=(x_1,x_2)$ denotes the horizontal variable. The horizontal velocity $v$ is even and the vertical velocity $w$ is odd in the vertical variable $z$, which is the symmetry preserved by the primitive system. The incompressibility condition allows us to recover the vertical velocity $w$ from the horizontal $v$, 
\begin{align}\label{eq:vertical-velocity}
 w(v)(\boldsymbol x,z)=-\int_0^z\nabla_h\cdot v(\boldsymbol x,\zeta)\dd\zeta,
\end{align}
where $w(v)(\boldsymbol x,0)=0$ by oddness. We will use the barotropic--baroclinic decomposition
\begin{align}\label{eq:barotropic--baroclinic}
 u(\boldsymbol x,z)=\bar u(\boldsymbol x)+\widetilde u(\boldsymbol x,z),\qquad \bar u(\boldsymbol x)=\frac{1}{2\pi}\int_{\mathbb T}u(\boldsymbol x,z)\,\dd z,
\end{align}
where $\bar u$ is the barotropic component and $\widetilde u=u-\bar u$ is the baroclinic component. 
Denoting $\mathsf P$ as the hydrostatic Helmholtz projection given by 
\begin{align*}
\mathsf P u =
 u-\nabla_h\Delta_h^{-1}
 \bigl(\nabla_h\cdot\bar u\bigr)
\end{align*}
and setting 
\begin{align}\label{eq:hydrostatic-nonlinearity}
 B(f,g)=\mathsf P\{f\cdot\nabla_h g+w(f)\partial_zg\}, \quad f,g\in C^{\infty}(\mathbb T^3, \mathbb R^2), 
\end{align}
the primitive system \eqref{eq:intro-pe} is reduced to the following abstract evolution system for the horizontal velocity, 
\begin{align}\label{eq:pe}
 \dd v_t=\{-\nu Av_t-B(v_t,v_t)\}\dd t
 +\sum_{k=1}^m b_k\,\dd W_t^k
\end{align}
where $A$ is the hydrostatic Stokes operator and $b_k=\mathsf P g_k$. The hydrostatic energy space is 
\begin{align*}
H = \left\{v\in L_{*}^2(\mathbb T^3;\mathbb R^2): \int_{\mathbb T^3}v(\boldsymbol x, z)\dd \boldsymbol x \dd z =0, \, \int_{\mathbb T}\nabla_h\cdot v(\cdot,\zeta)\dd \zeta=0\right\}, 
\end{align*}
where $L_{*}^2(\mathbb T^3;\mathbb R^2)$ denots the space of square-integrable velocity fields that are even in $z$. We also denote $H^s=H\cap W^{s,2}$, where $W^{s,2}$ is the usual Sobolev space of order $s$.  The $L^p$ norm will be denoted by $\|\cdot\|_p$. 

We assume $b_1,\ldots,b_m\in H$ are smooth forcing profiles and $W = (W^1,\ldots,W^m)$ the standard $\mathbb R^m$ valued Wiener process over the canonical Wiener space $(\Omega,\mathscr F, \mathscr F_t, \mathbb P)$ satisfying the usual conditions. For convenience we retain the associated linear noise operator
\begin{align}\label{eq:noise-operator}
 Q:\R^m\to H,\qquad Q\boldsymbol{e}_k=b_k
\end{align}
where $\{\boldsymbol{e}_k\}$ is the canonical basis of $\mathbb R^m$. Under this setting, the equation \eqref{eq:pe} is globally well-posed in $H^1$ and defines a Feller Markov process. We denote the solution map by  $\Phi_t(u,\omega)$, for initial data $u\in H^1$ and  driver $\omega\in C_0([0,t];\R^m)$. In particular, $\Phi_t(u,W)$ is the solution to \eqref{eq:pe} with initial data $u$ and 
\begin{align}\label{eq:solution-map}
 P_t\varphi(u)=\E\varphi(\Phi_t(u,W)), \quad \varphi\in C_b(H^1)
\end{align}
defines the Markov semigroup.  We use the right action $\mu P_t$ on probability measures and denote $\mu(f)$ as the integration of $f$ against the measure $\mu$.

\subsection{Quantitative quadratic saturation}
To specify the noise structure we denote
\begin{align}\label{eq:hydrostatic-frequencies}
 \mathbf F
 =\{(\boldsymbol k,n)\in\mathbb Z^2\times\mathbb N_0:
 (\boldsymbol k,n)\ne(0,0)\}
\end{align}
as the set of hydrostatic frequencies. In particular, frequencies corresponding to $n=0$ are barotropic and the rest are called baroclinic. The complexification
$H_{\mathbb C}=H\otimes_{\mathbb R}\mathbb C$ has the standard hydrostatic Fourier basis, up to normalization,
\begin{align}\label{eq:hydrostatic-fourier-basis}
 \left\{e^{\mathrm{i}\boldsymbol k\cdot\boldsymbol x}\cos(nz) \boldsymbol e_1,\,
 e^{\mathrm{i}\boldsymbol k\cdot\boldsymbol x}\cos(nz) \boldsymbol e_2\right\}_{\boldsymbol k\in\mathbb Z^2,n\ge1}, \text{ and }
 \left\{e^{\mathrm{i}\boldsymbol k\cdot\boldsymbol x}\boldsymbol k^\perp\right\}_{\boldsymbol k\in\mathbb Z^2\setminus\{0\}},
\end{align}
where $\boldsymbol k^\perp=(-k_2,k_1)$. The corresponding real sine--cosine
modes form the hydrostatic Fourier basis of $H$. 

We will consider phase-complete finite Fourier forcing in the following sense. 
\begin{definition}\label{def:phase-complete}
The forcing range $E_0=\Ran Q$ is called \emph{phase-complete finite Fourier} if for some finite index sets $J_0,J$, it has the form
\begin{align*}
 E_0=\operatorname{span}_{\R}\Bigl(\left\{\boldsymbol a_j(z)\cos(\boldsymbol k_j\cdot\boldsymbol x),\,
 \boldsymbol a_j(z)\sin(\boldsymbol k_j\cdot\boldsymbol x):j\in J\right\}\bigcup\left\{\boldsymbol a_\ell^0(z):\ell\in J_0\right\}\Bigr)\subset H,
\end{align*}
where $\boldsymbol k_j\in\mathbb Z^2\setminus\{0\}$ and the profiles
$\boldsymbol a_j,\boldsymbol a_\ell^0$ are real vector-valued even trigonometric
polynomials. 
\end{definition}

We next give a quantitative saturation ensuring that every unit hydrostatic Fourier mode has a constant Lie-word
representation in a controlled way.  Set the symmetric quadratic interaction
\begin{align}\label{eq:symmetric-quadratic-interaction}
\mathsf B(g,h)=B(g,h)+B(h,g),
\end{align}
and define iteratively 
\begin{align}\label{eq:quadratic-saturation}
 E_{n+1}=\operatorname{span}_{\R}\bigl(E_n\cup\{\mathsf B(g,h):g,h\in E_n\}\bigr), \quad E_0=\Ran Q. 
\end{align}
Since each $E_j$ is real,
\begin{align*}
 f\in E_j^{\mathbb{C}}:= E_j\otimes_{\mathbb R}\mathbb C
 \quad\Longleftrightarrow\quad
 \Re f,\Im f\in E_j.
\end{align*}
Thus the complex Fourier identities below are merely a compact notation
for pairs of exact real sine--cosine identities.

For $(\boldsymbol k,n)\in\mathbf F$, let
\begin{align*}
 \mathcal H_{\boldsymbol k,n}
 =
 \begin{cases}
  \left\{e^{\mathrm{i}\boldsymbol k\cdot\boldsymbol x}\cos(nz)\mathsf a:\mathsf a\in\mathbb C^2\right\},
  & n\ge1,\\
  \operatorname{span}_{\mathbb C}
  \left\{e^{\mathrm{i}\boldsymbol k\cdot\boldsymbol x}\boldsymbol k^\perp\right\},
  & n=0
 \end{cases}
\end{align*}
which denotes the complex Fourier fibre at the fixed hydrostatic frequency $(\boldsymbol k,n)$. 

\begin{definition}
\label{def:fourier-certificate}
We say that the forcing satisfies \emph{quantitative quadratic
saturation} if there exists a finite index $j_*$ such that
\begin{align}
 \operatorname{span}_{\mathbb Z}
 \bigl\{
 \boldsymbol k\in\mathbb Z^2\setminus\{0\}:
 E_{j_*}^{\mathbb C}\cap\mathcal H_{\boldsymbol k,0}\ne\{0\}
 \bigr\}
 &=\mathbb Z^2,\notag\\
 \gcd\bigl\{
 n\ge1:
 E_{j_*}^{\mathbb C}\cap\mathcal H_{\boldsymbol k,n}\ne\{0\}
 \text{ for some }\boldsymbol k\in\mathbb Z^2
 \bigr\}
 &=1.
\label{eq:finite-fourier-certificate}
\end{align}
\end{definition}
The two conditions say that the exactly generated barotropic frequencies span the full horizontal lattice and that the exactly generated positive vertical levels have greatest common divisor one. The fibre intersections require exact pure modes rather than mere Fourier support.
By \Cref{prop:geometric-generation}, quantitative quadratic saturation
implies exact generation of every hydrostatic Fourier mode and, in
particular, quadratic saturation
$\overline{\bigcup_{n\ge0}E_n}^{\,H}=H$. The case in which one of the generated positive levels is already
$n=1$ is the simplest instance.  The greatest-common-divisor
formulation also covers forcing ranges for which level one is obtained
only after finitely many quadratic interactions.

\begin{remark}\label{rem:low-rank-example}
Write $\boldsymbol e_1=(1,0)$ and $\boldsymbol e_2=(0,1)$. The four profiles
\begin{align*}
 f_1&=\boldsymbol e_2\cos x_1(1+\cos z),&
 f_2&=\boldsymbol e_2\sin x_1(1+\cos z),\\
 f_3&=\boldsymbol e_1\cos x_2,&
 f_4&=\boldsymbol e_1\sin x_2
\end{align*}
satisfy quantitative quadratic saturation. This rank-four example is
optimal in the phase-complete class; see \Cref{prop:four-noises}.

If each forcing profile is supported on a single vertical Fourier level,
the five directions
\begin{align*}
 b_1&=\boldsymbol e_2\cos x_1,& b_2&=\boldsymbol e_2\sin x_1,&
 b_3&=\boldsymbol e_1\cos x_2,& b_4&=\boldsymbol e_1\sin x_2,&
 b_5&=\boldsymbol e_1\cos z
\end{align*}
satisfy quantitative quadratic saturation, and five is optimal in this
single-level class by \Cref{prop:five-noises}.

Within the same single-level class, if horizontally constant profiles
are excluded, the optimal rank is six. One example is given by the first
four profiles above together with
\begin{align*}
 b_5=\boldsymbol e_1\cos(x_1+x_2)\cos z,\qquad
 b_6=\boldsymbol e_1\sin(x_1+x_2)\cos z;
\end{align*}
see \Cref{cor:six-noises}. 
\end{remark}

\subsection{Main results}

For a distance-like function $d$ (which need not satisfy a triangle inequality) on a Polish space $X$, write
\begin{align*}
 \mathcal W_d(\mu_1,\mu_2)=\inf_{\pi\in\mathsf C(\mu_1,\mu_2)}\int_{X\times X} d(u,v)\,\pi(\dd u,\dd v),
\end{align*}
for $\mu_1,\mu_2\in \mathcal{P}(X)$ the space of probability measures on $X$, where $\mathsf C(\mu_1,\mu_2)$ is the set of couplings between $\mu_1,\mu_2$. The distance-like cost $\mathcal W_d$ is the usual Wasserstein distance when $d$ is a metric.  To define the regular phase space and the weights in the theorem, fix $\alpha=\alpha(\nu)\ge1$ as in \Cref{lem:cascade-direct} and a
constant $C_\nu\ge1$ sufficiently large as in
\Cref{lem:admissibility}. For $\eta>0$, denote by  
\begin{align}\label{eq:mixing-lyapunov}
 U(u)=\left(1+\alpha\|u\|_8^8+\|\partial_zu\|_4^4+\|\nabla u\|_2^2\right)^\eta\exp\left(C_{\nu}\eta\|u\|_H^2\right),
\end{align}
and let 
\begin{align}\label{eq:regular-phase-space}
 \mathscr E=\{u\in H^1:U(u)<\infty\}=\{u\in H^1:u\in L^8,\ \partial_zu\in L^4\}.
\end{align}
The regular phase space $\mathscr E$ is equipped, when its stronger topology is needed, with
\begin{align}\label{eq:regular-phase-norm}
 \|u\|_{\mathscr E}=\|u\|_{H^1}+\|u\|_{L^8}+\|\partial_zu\|_{L^4}.
\end{align}
For $0<q<1$ set $V_U(u,v)=1+U(u)+U(v)$ and define on
$\mathscr E$ the distance-like function 
\begin{align}\label{eq:main-explicit-cost}
 d_q(u,v)=\min\left\{\sqrt{V_U(u,v)},\,
 \frac{V_U(u,v)}{[1+\log(1/d(u,v))]^{q/2}}\right\},\qquad u\ne v,
 \qquad d_q(u,u)=0, 
\end{align}
where 
\begin{align*}
 d(u,v)=1\wedge\|u-v\|_H,\qquad u,v\in H^1.
\end{align*}
Note that $d_q$ depends on $\eta$ from \eqref{eq:mixing-lyapunov}. 

\begin{theorem}\label{thm:main}
Assume that $\Ran Q$ is phase-complete finite Fourier in the sense of \Cref{def:phase-complete} and satisfies quantitative quadratic saturation in the sense of \Cref{def:fourier-certificate}. Then there is a unique invariant measure $\mu$ with full support in $H^1$ and $\mu(\mathscr E)=1$. Moreover, there is $\eta_0=\eta_0(\nu,Q)>0$ so that for every $\eta\in(0,\eta_0]$ and $q\in(0,\eta)$, there are $C,\gamma>0$ such that 
\begin{align}\label{eq:main-explicit-gap}
 \mathcal W_{d_q}(\mu_1 P_t,\mu_2 P_t)
 &\le Ce^{-\gamma t}\mathcal W_{ d_q}(\mu_1,\mu_2),\qquad t\ge0
\end{align}
for any $\mu_1,\mu_2\in\mathcal{P}(\mathscr E)$ making the right-hand side finite. 
\end{theorem}

Our framework uses the strong topology of the regular phase space $\mathscr E$ together with a weighted logarithmic cost based on the weaker comparison distance $1\wedge\|u-v\|_H$.
For the primitive equations, positive-time smoothing upgrades the spectral gap to the corresponding cost based on $1\wedge\|u-v\|_{H^1}$ and yields exponential mixing in the truncated $H^1$ Wasserstein distance.  Denote $d_1(u,v)=1\wedge\|u-v\|_{H^1}$ 
and set $d_{q,1}$ as the cost obtained from \eqref{eq:main-explicit-cost} by replacing the bounded $H$-distance $d$  by $d_1$. 

\begin{corollary}\label{thm:exponential-mixing}
Under the assumptions of \Cref{thm:main}, for every $\eta\in(0,\eta_0]$ and $q\in(0,\frac{\eta}{\eta+2})$, there are  $\widetilde{C}, \widetilde{\gamma}>0$ such that
\begin{align}\label{eq:main-explicit-ga-H1}
 \mathcal W_{d_{q,1}}(\mu_1 P_t,\mu_2 P_t)
 &\le \widetilde{C} e^{-\widetilde{\gamma} t}\mathcal W_{ d_{q,1}}(\mu_1,\mu_2),\qquad t\ge0
\end{align}
for all $\mu_1,\mu_2\in\mathcal P(\mathscr E)$ making the right-hand side finite. In particular, there are $C,c,\kappa>0$ such that 
\begin{align}\label{eq:exponential-mixing-wp}
 \mathcal W_{d_1}(P_t(u,\cdot),\mu)\le C\exp(\kappa\|u\|_H^2-ct),\qquad u\in H^1,\ t\ge0.
\end{align}
\end{corollary}

The proofs for \Cref{thm:main} and \Cref{thm:exponential-mixing} are given in \Cref{sec:assembly}.

\begin{remark}\label{rem:asymptotic-regularization-comparison}
The asymptotic regularization induced by the logarithmic gauge above is weaker than the asymptotic regularization appearing in the literature.  The
comparison is most transparent at the level of moduli of continuity.
Schematically, the three regimes are
\begin{center}
\begin{tabular}{c|c|c}
 mechanism & representative asymptotic estimate
 & local modulus $\omega(r)$\\
\hline
gradient asymptotic strong Feller
&
$\displaystyle
 [P_t f]_{\mathrm{Lip}}
 \lesssim e^{-\gamma t}[f]_{\mathrm{Lip}}+\|f\|_\infty$
&
$\displaystyle r$
\\[3mm]
damped wave
&
$\displaystyle
 [P_tf]_{\bar d_{\mathrm W}}
 \lesssim e^{-\gamma t}[f]_{\bar d_{\mathrm W}}
     +\|f\|_\infty$
&
$\displaystyle r^{\beta_{\mathrm W}}$
\\[3mm]
primitive equations
&
$\displaystyle
 [P_tf]_{\bar d_q}
 \lesssim e^{-\gamma t}[f]_{\bar d_q}
     +\|f\|_\infty$
&
$\displaystyle
 [1+\log(1/r)]^{-q/2}$
\end{tabular}
\end{center}
Here the first line represents the asymptotic strong Feller gradient
estimate of \cite{HM06,HM08,HM}.  In contrast, the latter two estimates are
consequences of Wasserstein contraction rather than inputs to its
proof.  In the wave setting of \cite{LLWave}, the spectral-gap cost
$\bar d_{\mathrm W}$ is locally equivalent, on bounded Lyapunov sets,
to a fractional power $r^{\beta_{\mathrm W}}$ of the weak phase-space
distance.  Thus the resulting asymptotic regularization is of
H\"older type. For the present equation, put
\begin{align*}
 \bar d_q=1\wedge d_{q,1},\qquad
 [f]_{\bar d_q}
 =\sup_{u\ne v}\frac{|f(u)-f(v)|}{\bar d_q(u,v)}.
\end{align*}
The spectral gap \eqref{eq:main-explicit-ga-H1} immediately gives
\begin{align}\label{eq:pe-log-asymptotic-regularization}
 [P_tf]_{\bar d_q}
 \le Ce^{-\gamma t}[f]_{\bar d_q}+2\|f\|_\infty.
\end{align}
Moreover, on every fixed Lyapunov sublevel
$U(u)+U(v)\le R$, if $r=d_1(u,v)=1\wedge\|u-v\|_{H^1}\downarrow0$, then
\begin{align}\label{eq:pe-local-log-modulus}
 \bar d_q(u,v)
 \asymp_R
 [1+\log(1/r)]^{-q/2}.
\end{align}
Hence \eqref{eq:pe-log-asymptotic-regularization} is a fractional
logarithmic analogue of the H\"older asymptotic regularization in the
wave equation. In this sense the three mechanisms correspond, at the level of local
moduli, to
\begin{align*}
 r
 \quad\longrightarrow\quad
 r^{\beta_{\mathrm W}}
 \quad\longrightarrow\quad
 [1+\log(1/r)]^{-q/2}.
\end{align*}
Since
\begin{align*}
 r\ll r^{\beta_{\mathrm W}}
 \ll [1+\log(1/r)]^{-q/2}
 \qquad(r\downarrow0),
\end{align*}
these represent successively weaker regularity scales.

\end{remark}

\section{A quantitative stable--compact criterion}\label{sec:coupling}

In this section we prove an abstract spectral gap criterion based on a quantitative stable--compact coupling, tailored to systems such as the primitive equations where the free amplification has only weak logarithmic moments but the compact component admits quantitative compensation. Compared with \cite{LLWave,LLSNLS}, Lyapunov recurrence and irreducibility are unchanged, while the compensation must now improve with the incoming scale. The criterion is formulated solely in terms of a Markov semigroup, a weak comparison metric and an exact local coupling.

For Gaussian drivers, the scale-dependent compensation on a compact
regular core follows from quantitative control of the Malliavin
operator beyond mere density of its range; the corresponding
compact-core mechanism is given in \Cref{prop:compact-source-repair}.
In the primitive-equation application, positive-time regularization
first prepares such a core from the recurrent Lyapunov sublevel.

Let $\mathcal X$ be a Polish space and let $d:\mathcal X\times\mathcal X\to[0,1]$ be a bounded
continuous metric. The metric $d$ is used only for comparison and need
not generate the Polish topology of $\mathcal X$. Let $(P_t)_{t\ge0}$ be a Markov
semigroup on $\mathcal X$ with Borel transition kernels, and let
$U:\mathcal X\to[1,\infty)$ be continuous. 
We assume throughout that $(P_t)$ has at least one invariant probability. Fix $T>0$ and write $P=P_T$.

The first assumption provides recurrence and the short-time Lyapunov
control used to weight the local comparison and return the dynamics to a
recurrent sublevel.

\begin{assumption}\label{ass:gap-drift}
There are $0<\lambda<1$ and $b\ge1$ such that
\begin{align}\label{eq:gap-drift}
 PU\le\lambda U+b.
\end{align}
Moreover, for every $T_0<\infty$ there is $C_{T_0}<\infty$ such that
\begin{align}\label{eq:gap-shorttime-input}
 \sup_{0\le t\le T_0}P_tU(u)\le C_{T_0}U(u).
\end{align}
\end{assumption}

The second assumption controls the amplification of a small separation
under the uncorrected dynamics.  Its fractional logarithmic moment is the
weak quantitative feature that determines the comparison gauge used below.

\begin{assumption}
\label{ass:gap-amplification}
There is $0<\theta<1$ such that, for every $T_0<\infty$ and every
$u,v\in\mathcal X$, the two trajectories admit a jointly Borel common
realization $(X_t,Y_t)_{0\le t\le T_0}$ for which there is a
nonnegative random variable $H_{T_0}$ satisfying
\begin{align}\label{eq:structural-clock}
 d(X_t,Y_t)&\le e^{H_{T_0}}d(u,v),
 \qquad 0\le t\le T_0,\notag\\
 \E H_{T_0}^{\theta}
 &\le C_{T_0}\min\{U(u),U(v)\}.
\end{align}
We call $H_{T_0}$ the free amplification clock.
\end{assumption}

The third assumption records the local gain supplied by compensation.
On a recurrent sublevel it improves the incoming scale with uniformly
positive probability, while staying inside the free-amplification envelope
outside an event of algebraically small probability.

\begin{assumption}
\label{ass:gap-local-repair}
Fix $R>\frac{4b}{1-\lambda}$. There are
$p_*,a_*,\varkappa,r_*,C_*>0$ such that, whenever
\begin{align*}
 \min\{U(u),U(v)\}\le R,\qquad
 r=d(u,v)\in(0,r_*],
\end{align*}
there is a jointly Borel exact coupling $(X',Y')$ of
$P(u,\cdot)$ and $P(v,\cdot)$ such that
\begin{align}\label{eq:gap-repair-input}
 \PP\{d(X',Y')\le r^{1+a_*}\}&\ge p_*,\\\label{eq:gap-repair-input-2}
 \PP\{d(X',Y')>e^{C_*+H_T}r\}&\le C_*r^\varkappa,
\end{align}
where $H_T$ is the amplification clock satisfying
\Cref{ass:gap-amplification}.
\end{assumption}

The final assumption on uniform irreducibility globalizes the local mechanism.  It supplies a
common region that can be reached uniformly from every Lyapunov sublevel.

\begin{assumption}\label{ass:gap-accessibility}
There is $u_*\in\mathcal X$ such that, for every $R_1<\infty$ and every
neighborhood $\mathcal O$ of $u_*$ in the Polish topology of $\mathcal X$, there are
$n=n(R_1,\mathcal O)\ge1$ and $p=p(R_1,\mathcal O)>0$ satisfying
\begin{align}\label{eq:abstract-accessibility}
 \inf_{U(u)\le R_1}P^n(u,\mathcal O)\ge p.
\end{align}
\end{assumption}

For $0<q<\theta$ set $V_U(u,v)=1+U(u)+U(v)$ and define
\begin{align}\label{eq:abstract-explicit-cost}
 d_q(u,v)=\min\left\{\sqrt{V_U(u,v)},\,
 \frac{V_U(u,v)}
 {[1+\log(1/d(u,v))]^{q/2}}\right\},
 \qquad u\ne v,\qquad d_q(u,u)=0.
\end{align}
\begin{theorem}[Quantitative stable--compact spectral gap]
\label{thm:structural-gap}
Suppose that \Cref{ass:gap-drift,ass:gap-amplification,ass:gap-local-repair,ass:gap-accessibility} hold. Then the invariant probability $\mu$ is unique and for every $0<q<\theta$ there are $C,\gamma>0$ such that 
\begin{align}\label{eq:abstract-explicit-gap}
 \mathcal W_{d_q}(\mu_1 P_t,\mu_2 P_t)
 \le Ce^{-\gamma t}\mathcal W_{d_q}(\mu_1,\mu_2),\qquad t\ge0,
\end{align}
for $\mu_1,\mu_2\in\mathcal{P}(\mathcal X)$ making the right-hand side finite. 
\end{theorem}

\subsection{Proof of the spectral gap criterion}
This subsection is devoted to proving \Cref{thm:structural-gap}. We first deduce the small set consequence from the uniform irreducibility condition. 

\begin{lemma}
\label{lem:accessibility-small set}
Assume \Cref{ass:gap-accessibility}.  For every
$R_1<\infty$ there is $R_0<\infty$ such that, for every
$\varepsilon>0$, some $n_\varepsilon\ge1$ and $p_\varepsilon>0$ admit an exact
coupling of $P^{n_\varepsilon}(u,\cdot)$ and $P^{n_\varepsilon}(v,\cdot)$ with
\begin{align}\label{eq:gap-preparation}
 \PP\{d(X',Y')\le\varepsilon,\ U(X')+U(Y')\le R_0\}
 \ge p_\varepsilon
 \qquad\text{if }U(u)+U(v)\le R_1.
\end{align}
The constant $R_0$ is independent of $\varepsilon$.
\end{lemma}

\begin{proof}
By continuity of $U$ at $u_*$, choose a neighborhood
$\mathcal O_0$ of $u_*$ and $M<\infty$ such that
$\sup_{\mathcal O_0}U\le M$. Since $d$ is continuous and
$d(u_*,u_*)=0$, for every $\varepsilon>0$ there is a neighborhood
$\mathcal O_\varepsilon\subset\mathcal O_0$ of $u_*$ such that
$\operatorname{diam}_d\mathcal O_\varepsilon\le\varepsilon$.
By \eqref{eq:abstract-accessibility}, there are $n_\varepsilon\ge1$ and
$p=p(\varepsilon)>0$ such that
\begin{align*}
 P^{n_\varepsilon}(u,\mathcal O_\varepsilon)\ge p,
 \quad \text{uniformly for } U(u)\le R_1.
\end{align*}
Under the product coupling of
$P^{n_\varepsilon}(u,\cdot)$ and $P^{n_\varepsilon}(v,\cdot)$, the event
$\{X',Y'\in\mathcal O_\varepsilon\}$ has probability at least $p^2$.
On this event,
\begin{align*}
 d(X',Y')\le\varepsilon,\qquad U(X')+U(Y')\le2M.
\end{align*}
Thus \eqref{eq:gap-preparation} holds with $R_0=2M$ and
$p_\varepsilon=p^2$.
\end{proof}

For parameters $c_0,L,\beta_0>0$ and $0<q<\theta$ to be chosen below,
define
\begin{align}\label{eq:guage}
\begin{split}
 \ell(0)&=0,\,
 \ell(r)=\left(\frac{L}{L+\log(1/r)}\right)^q\quad(0<r\le1),\\
 \delta(u,v)&=1\wedge
 \bigl[(c_0+U(u)+U(v))\ell(d(u,v))\bigr].
\end{split}
\end{align}

The next lemma provides the $d$-contraction and nonexpansion property. Let $\lambda, b$ be as in \Cref{ass:gap-drift} and $p_*,a_*,\varkappa,r_*,C_*$ be as in
\Cref{ass:gap-local-repair}. 
\begin{lemma}
\label{lem:gap-local-coupling}
Suppose that \Cref{ass:gap-drift,ass:gap-amplification,ass:gap-local-repair} hold.  Then for every $0<q<\theta$, there are
$c_0,L>0$ and $\alpha_0\in(0,1)$ such that
\begin{align}\label{eq:gap-delta-local}
 \mathcal W_\delta(P(u,\cdot),P(v,\cdot))&\le\alpha_0\delta(u,v)
 \quad\text{if }\delta(u,v)<1,\\
 \mathcal W_\delta(P(u,\cdot),P(v,\cdot))&\le\delta(u,v)
 \quad\text{for all }u,v\in\mathcal X.\label{eq:gap-delta-nonexpansion}
\end{align}
\end{lemma}
\begin{proof}
Since $\delta\le1$, once \eqref{eq:gap-delta-local} is proved, 
\eqref{eq:gap-delta-nonexpansion} follows immediately. Thus it remains to prove the strict contraction for
$\delta(u,v)<1$. Fix $0<q<\theta$ and set
\begin{align}\label{eq:gap-parameters}
 h=\min\{1-(1+a_*/2)^{-q},1-\lambda\},\qquad c_0=\frac{4b}{p_*h},\qquad
 \rho_0=1-\frac{p_*h}{2}.
\end{align}
Write $F(u,v)=c_0+U(u)+U(v)$ and $S=U(u)+U(v)$.

\emph{Step 1: Weighted contraction.}
Suppose $\delta(u,v)<1$, and put $r=d(u,v)$ and
$s=\log(1/r)$. Since the truncation in $\delta$ is inactive and $U\geq 1$,
\begin{align}\label{eq:gap-log-lower}
 s>L\bigl(F(u,v)^{1/q}-1\bigr)\ge L.
\end{align}
We increase $L$ if necessary to ensure $r<r_*$.  If $S\le2R$, then $\min\{U(u),U(v)\}\le R$ and by setting  
\begin{align*}
 G=\{d(X',Y')\le r^{1+a_*}\},\qquad
 E=\{d(X',Y')\le e^{C_*+H_T}r\},
\end{align*}
it follows from \Cref{ass:gap-local-repair} that $\PP(G)\ge p_*$ and $\PP(E^c)\le C_*r^\varkappa$. Since $F(X',Y')\ge c_0$, \eqref{eq:gap-parameters}, exactness
of the marginals and \Cref{ass:gap-drift} give
\begin{align}\label{eq:weighted-core-gain}
 \E[(1-h\mathbf1_G)F(X',Y')]
 &\le c_0+\lambda S+2b-hp_*c_0
 \le\rho_0F(u,v),
\end{align}
where we used  $(1-\rho_0)c_0=2b$ and $\lambda\le\rho_0$ from \eqref{eq:gap-parameters}. If $S>2R$, take $(X',Y')$ to be the synchronous coupling from
\Cref{ass:gap-amplification}, set $G=\varnothing$, and let
$E=\{d(X',Y')\le e^{C_*+H_T}r\}$. Then $\PP(E^c)=0$, while the
Lyapunov drift gives
\begin{align}\label{eq:weighted-drift-gain}
 \E F(X',Y')\le c_0+\lambda S+2b\le\rho_0F(u,v), 
\end{align}
where we used $\rho_0-\lambda\geq (1-\lambda)/2$ by \eqref{eq:gap-parameters} and   $4b-(1-\lambda)S/2<0$ as 
$S>2R$ and $R>4b/(1-\lambda)$ by \Cref{ass:gap-local-repair}. Thus in both regimes
\begin{align}\label{eq:weighted-unified-gain}
 \E[(1-h\mathbf1_G)F(X',Y')]\le\rho_0F(u,v),\qquad
 \PP(E^c)\le C_*r^\varkappa.
\end{align}

\emph{Step 2: Logarithmic gauge and amplification errors.}
On the repair event $G$,
\begin{align*}
 \frac{\ell(d(X',Y'))}{\ell(r)}
 \le\frac{\ell(r^{1+a_*})}{\ell(r)}
 =\left(\frac{L+s}{L+(1+a_*)s}\right)^q
 \le(1+a_*/2)^{-q}\le1-h
\end{align*}
by \eqref{eq:gap-parameters} and \eqref{eq:gap-log-lower}. Choose $\varepsilon\in(0,1)$ so that
$(1-\varepsilon)^{-q}\rho_0<1$, and increase $L$ so that
$C_*\le\varepsilon s/2$ whenever $s\ge L$. On
$E\cap\{H_T\le\varepsilon s/2\}$,
\begin{align*}
 d(X',Y')\le e^{\varepsilon s}r=r^{1-\varepsilon},\qquad
 \ell(d(X',Y'))\le \ell(r^{1-\varepsilon})\le(1-\varepsilon)^{-q}\ell(r).
\end{align*}
Consequently, on
$\mathbb G:=G\cup(E\cap\{H_T\le\varepsilon s/2\})$,
\begin{align*}
 \ell(d(X',Y'))
 \le(1-\varepsilon)^{-q}(1-h\mathbf1_G)\ell(r).
\end{align*}
Combining this estimate with \eqref{eq:weighted-unified-gain}, and using $\delta\le1$ on $\mathbb G^c$, Markov's inequality, and
\Cref{ass:gap-amplification}, we obtain
\begin{align*}
 \E\delta(X',Y')
 \le(1-\varepsilon)^{-q}\rho_0\delta(u,v)
 +C\min\{U(u),U(v)\}s^{-\theta}+Ce^{-\varkappa s}.
\end{align*}
The first coefficient is strictly below one. The remaining two terms
are negligible relative to the logarithmic cost. Indeed,
\begin{align*}
 \delta(u,v)=F(u,v)\left(\frac{L}{L+s}\right)^q,
\end{align*}
and \eqref{eq:gap-log-lower}, together with $q<\theta$, gives
\begin{align}\label{eq:amplification-comparison}
 \frac{\min\{U(u),U(v)\}s^{-\theta}}{\delta(u,v)}
 \le2^qL^{-\theta},\qquad
 \frac{e^{-\varkappa s}}{\delta(u,v)}
 \le Ce^{-\varkappa L/2}.
\end{align}
Choosing $L$ sufficiently large therefore gives some
$\alpha_0\in(0,1)$ such that
\begin{align*}
 \E\delta(X',Y')\le\alpha_0\delta(u,v).
\end{align*}
Taking the infimum over couplings proves
\eqref{eq:gap-delta-local}. 
\end{proof}

We now combine the previous two lemmas to deduce contraction for a skeleton kernel. Let $c_0,L,q,\alpha_0$ be as in
\Cref{lem:gap-local-coupling}, increasing $L$ if necessary so that
$r\le\sqrt{\ell(r)}$ for $0\le r\le1$. For $\beta_0>0$ define
\begin{align}\label{eq:gap-weighted-cost}
 \mathsf d(u,v)=\sqrt{\delta(u,v)
 (1+\beta_0U(u)+\beta_0U(v))}.
\end{align}
\begin{proposition}
\label{thm:weighted-spectral-gap}
Suppose that \Cref{ass:gap-drift,ass:gap-amplification,ass:gap-local-repair,ass:gap-accessibility} hold. Then there are
$\beta_0>0$, an integer $m\ge1$ and $\chi\in(0,1)$ such that
\begin{align}\label{eq:gap-skeleton-contraction}
 \mathcal W_{\mathsf d}(\mu_1 P^{mn},\mu_2 P^{mn})
 \le\chi^n\mathcal W_{\mathsf d}(\mu_1,\mu_2),\qquad n\ge0,
\end{align}
for probability measures making the right-hand side finite. Moreover,
\begin{align}\label{eq:gap-cost-bounds}
 d(u,v)\le\sqrt{\ell(d(u,v))}\le \mathsf d(u,v)
 \le\sqrt{1+\beta_0U(u)+\beta_0U(v)}.
\end{align}
Every invariant probability has $U$ moment at most
$\frac{b}{1-\lambda}$, and the invariant probability is unique.
\end{proposition}

\begin{proof}
We first turn accessibility into a strict $\delta$-small set estimate.
Apply \Cref{lem:accessibility-small set} with
$R_1=4b/(1-\lambda)$ and choose $\varepsilon>0$ so small that
\begin{align*}
 (c_0+R_0)\ell(\varepsilon)\le\frac12.
\end{align*}
For the coupling in \eqref{eq:gap-preparation}, on the event there
specified one has $\delta(X',Y')\le1/2$, while always $\delta\le1$.
Hence
\begin{align}\label{eq:gap-intermediate-small set}
 \mathcal W_\delta
 \bigl(P^{n_\varepsilon}(u,\cdot),P^{n_\varepsilon}(v,\cdot)\bigr)
 \le1-\frac{p_\varepsilon}{2}
 \qquad\text{if }U(u)+U(v)\le\frac{4b}{1-\lambda}.
\end{align}
By \eqref{eq:gap-delta-nonexpansion}, subsequent applications of $P$
preserve this estimate.

Choose $m\ge n_\varepsilon$ so that $\lambda^m\le1/8$, put
$\widehat P=P^m$ and $\varepsilon_0=p_\varepsilon/2$. Iterating
\eqref{eq:gap-drift} gives
\begin{align}\label{eq:gap-block-drift}
 \widehat P U
 \le\lambda^mU+b\sum_{j=0}^{m-1}\lambda^j
 \le\frac18U+\frac{b}{1-\lambda}.
\end{align}
Together with \eqref{eq:gap-intermediate-small set},
\begin{align}\label{eq:gap-block-small set}
 \mathcal W_\delta
 \bigl(\widehat P(u,\cdot),\widehat P(v,\cdot)\bigr)
 \le1-\varepsilon_0
 \qquad\text{if }U(u)+U(v)\le\frac{4b}{1-\lambda}.
\end{align}
Moreover, if $\delta(u,v)<1$, the first estimate in
\eqref{eq:gap-delta-local}, followed by repeated use of
\eqref{eq:gap-delta-nonexpansion}, gives
\begin{align}\label{eq:gap-block-local}
 \mathcal W_\delta
 \bigl(\widehat P(u,\cdot),\widehat P(v,\cdot)\bigr)
 \le\alpha_0\delta(u,v).
\end{align}

We now combine the local, small set and Lyapunov mechanisms through
the weighted cost $\mathsf d$. Recall $S=U(u)+U(v)$.
Under any coupling of the two $\widehat P$-laws,
\eqref{eq:gap-block-drift} gives
\begin{align}\label{eq:gap-block-weight}
 \E\bigl[1+\beta_0U(X')+\beta_0U(Y')\bigr]
 \le1+\frac{\beta_0S}{8}
 +\frac{2\beta_0b}{1-\lambda}.
\end{align}
We distinguish the three standard regimes.

For $\delta(u,v)<1$, \eqref{eq:gap-block-local}, Cauchy--Schwarz and
$\mathsf d(u,v)^2=\delta(u,v)(1+\beta_0S)$ yield
\begin{align*}
 \frac{\mathcal W_{\mathsf d}
 (\widehat P(u,\cdot),\widehat P(v,\cdot))^2}{\mathsf d(u,v)^2}
 \le
 \alpha_0
 \frac{1+\beta_0S/8+2\beta_0b/(1-\lambda)}
      {1+\beta_0S}\le
 \alpha_0\left(1+\frac{2\beta_0b}{1-\lambda}\right)
 =:r_1.
\end{align*}

Next suppose $\delta(u,v)=1$ and
$S\le4b/(1-\lambda)$. Using
\eqref{eq:gap-block-small set}, \eqref{eq:gap-block-weight} and
Cauchy--Schwarz,
\begin{align*}
 \frac{\mathcal W_{\mathsf d}
 (\widehat P(u,\cdot),\widehat P(v,\cdot))^2}{\mathsf d(u,v)^2}
 \le
 (1-\varepsilon_0)
 \frac{1+\beta_0S/8+2\beta_0b/(1-\lambda)}
      {1+\beta_0S}\le
 (1-\varepsilon_0)
 \left(1+\frac{5\beta_0b}{2(1-\lambda)}\right)
 =:r_2.
\end{align*}

Finally suppose $\delta(u,v)=1$ and
$S>4b/(1-\lambda)$. Here the Lyapunov drift alone supplies the
contraction. Using $\delta\le1$, \eqref{eq:gap-block-weight} and
Cauchy--Schwarz,
\begin{align*}
 \frac{\mathcal W_{\mathsf d}
 (\widehat P(u,\cdot),\widehat P(v,\cdot))^2}{\mathsf d(u,v)^2}
 \le
 \frac{1+\beta_0S/8+2\beta_0b/(1-\lambda)}
      {1+\beta_0S}\le
 \frac{1+\frac{5\beta_0b}{2(1-\lambda)}}
      {1+\frac{4\beta_0b}{1-\lambda}}
 =:r_3,
\end{align*}
where the last bound follows since the preceding ratio is decreasing
in $S$.

Since $\alpha_0<1$ and $\varepsilon_0>0$, we have $r_1,r_2<1$ for
all sufficiently small $\beta_0>0$, while $r_3<1$ for every
$\beta_0>0$. Hence we may choose $\beta_0>0$ so that
\begin{align*}
 \chi^2:=\max\{r_1,r_2,r_3\}<1,
\end{align*}
which in turn shows
\begin{align}\label{eq:gap-cost-contraction}
 \mathcal W_{\mathsf d}(\widehat P(u,\cdot),\widehat P(v,\cdot))
 \le\chi \mathsf d(u,v).
\end{align}
Integrating against an initial coupling and iteration  proves \eqref{eq:gap-skeleton-contraction}.

The bounds in \eqref{eq:gap-cost-bounds} follow directly from the
definitions: $\delta(u,v)\ge\ell(d(u,v))$, while $\delta\le1$, and
$L$ was chosen so that $r\le\sqrt{\ell(r)}$. Finally, if $\nu P=\nu$,
bounded concave truncations in $PU\le\lambda U+b$ give
\begin{align*}
 \nu(U)\le\frac{b}{1-\lambda}.
\end{align*}
Thus any two invariant probabilities have finite $\mathcal W_{\mathsf d}$ cost 
and \eqref{eq:gap-skeleton-contraction} gives uniqueness. 
\end{proof}

We are now ready to prove the main result of this subsection. 
\begin{proof}[Proof of \Cref{thm:structural-gap}]

Fix the contraction time $mT$ in \Cref{thm:weighted-spectral-gap}. We first extend the skeleton contraction to continuous time. For
$0\le \tau\le mT$, use the coupling in
\Cref{ass:gap-amplification}. If $\delta(u,v)<1$ in \eqref{eq:guage}, put
$r=d(u,v)$ and $s=\log(1/r)\ge L$. On
$\{H_{mT}\le s/2\}$,
\begin{align*}
 \ell(d(X_\tau,Y_\tau))\le \ell(r^{1/2})\le2^q\ell(r).
\end{align*}
Therefore
\begin{align*}
 \E\delta(X_\tau,Y_\tau)
 \le2^q\ell(r)\E[c_0+U(X_\tau)+U(Y_\tau)]
 +\PP(H_{mT}>s/2)\le C\delta(u,v),
\end{align*}
where we used \Cref{ass:gap-amplification},
\eqref{eq:gap-shorttime-input}, and the same comparison \eqref{eq:amplification-comparison} of the
fractional logarithmic tail with $\delta(u,v)$ as in
\Cref{lem:gap-local-coupling}. At $\delta(u,v)=1$ the same conclusion
is immediate. Cauchy--Schwarz, \eqref{eq:gap-weighted-cost} and
\eqref{eq:gap-shorttime-input} now give
\begin{align}\label{eq:gap-shorttime-cost}
 \mathcal W_{\mathsf d}(P_\tau(u,\cdot),P_\tau(v,\cdot))
 \le C\mathsf d(u,v),
 \qquad 0\le \tau\le mT, 
\end{align}
and integration against an initial coupling gives the same bound for probability measures.
Write $t=nmT+\tau$, with $0\le \tau<mT$. Applying
\eqref{eq:gap-skeleton-contraction} for the $n$ complete blocks and
then \eqref{eq:gap-shorttime-cost} for the remaining time gives
\begin{align}\label{eq:gap-continuous}
 \mathcal W_{\mathsf d}(\mu_1P_t,\mu_2P_t)
 \le Ce^{-\gamma t}\mathcal W_{\mathsf d}(\mu_1,\mu_2),
 \qquad t\ge0,
\end{align}
for some $C,\gamma>0$ and all probability measures making the right-hand side finite.

It remains to identify the resulting cost. For $u\ne v$, put
$s=\log(1/d(u,v))$ and $V_U=1+U(u)+U(v)$. Since
$c_0+U(u)+U(v)\asymp V_U$, 
$1+\beta_0U(u)+\beta_0U(v)\asymp V_U$ and $\ell(d(u,v))^{1/2}\asymp[1+s]^{-q/2}$,
\begin{align}\label{eq:explicit-equivalence}
 \mathsf d(u,v)
 &\asymp
 \min\left\{\sqrt{V_U(u,v)},\,
 \frac{V_U(u,v)}{[1+s]^{q/2}}\right\}
 =d_q(u,v).
\end{align}
The pointwise equivalence \eqref{eq:explicit-equivalence} and
\eqref{eq:gap-continuous} give
\eqref{eq:abstract-explicit-gap}. The uniqueness and $U$-moment
assertions are already contained in \Cref{thm:weighted-spectral-gap}.

\end{proof}

\subsection{Quantitative Gaussian compensation on compact cores}
\label{subsec:hilbertian-repair}

The criterion above is independent of the mechanism used to verify
\Cref{ass:gap-local-repair}. For the primitive equations, the local
repair is obtained in two stages: positive-time regularization first
maps the relevant Lyapunov sublevel into a compact regular core with
uniformly positive probability, and the following Hilbertian Gaussian
argument then produces the scale-dependent compensation on that core.

Suppose in addition that the Polish state space $\mathcal X$ is a vector
subspace of a separable Hilbert space $H$ for which the inclusion
$\mathcal X\hookrightarrow H$ is continuous, and take the comparison metric  as $d(u,v)=1\wedge\|u-v\|_H$. For $T>0$ let $\Omega_T=C_0([0,T];\R^m)$ be the canonical Wiener space with Wiener measure $\gamma_T$. We identify its Cameron--Martin space with the rate
space $\mathsf H_T=L^2(0,T;\R^m)$ through the continuous injection
\begin{align}\label{eq:iota}
 \iota:\mathsf H_T\to\Omega_T,\qquad
 (\iota h)(t)=\int_0^t h(s)\,\dd s.
\end{align}
Let $\Phi_t(u,\omega)$, $0\le t\le T$, be a jointly Borel pathwise
realization of the semigroup $P_t$ whose synchronous coupling satisfies
\Cref{ass:gap-amplification}, and put
\begin{align*}
 \Psi(u,\omega)=\Phi_T(u,\omega).
\end{align*}
For a Cameron--Martin direction $h\in\mathsf H_T$, write
\begin{align*}
 \mathcal D_h\Psi(u,\omega)
 =
 \left.\frac{d}{d\varepsilon}\right|_{\varepsilon=0}
 \Psi(u,\omega+\varepsilon\iota h)
\end{align*}
as the Malliavin derivative whenever it exists.

\begin{proposition}
\label{prop:compact-source-repair} Let $\mathscr D\times\Gamma\subset \mathcal X\times \Omega_T$ be compact with
$\gamma_T(\Gamma)>0$. Assume that there exists $\rho_0>0$ such
that the following properties hold on $\mathcal N:=\mathscr D\times\mathscr N_{\rho_0}(\Gamma)$ where $\mathscr N_{\rho_0}(\Gamma)$ denotes the $\rho_0$-neighborhood
of $\Gamma$ in $\Omega_T$. 
\begin{itemize}
\item  \emph{Regular dependence.} The initial-state derivative of $\Psi$ extends to a bounded operator
\begin{align*}
 J=D_u\Psi:H\to H,
\end{align*}
and the Malliavin derivative defines a bounded operator
\begin{align}\label{eq:abstract-endpoint-operators}
 A:\mathsf H_T\to H,\qquad
 Ah=\mathcal D_h\Psi(u,\omega).
\end{align}
The maps $\Psi$, $J$, and $A$ are jointly Borel and continuous on
$\mathcal N$. In addition, regarding $\omega$ as a parameter in
$C_0([0,T];\R^m)$, the Fr\'echet derivatives $D_\omega\Psi,D_\omega^2\Psi, D_\omega J$
exist and are uniformly bounded on $\mathcal N$. 

Moreover, there are $\varepsilon_*>0$ and $L<\infty$
such that, uniformly for $(u,\omega)\in\mathcal N$, every
$v\in\mathcal X$ and $h\in \mathsf H_T$ with
$\|v-u\|_H+\|h\|_{\mathsf H_T}\le\varepsilon_*$ satisfy
\begin{align}\label{eq:repair-remainder}
 \|\Psi(v,\omega+\iota h)-\Psi(u,\omega)
 -J(u,\omega)(v-u)-A(u,\omega)h\|_H\le
 L\bigl(\|v-u\|_H+\|h\|_{\mathsf H_T}\bigr)^2.
\end{align}
\item \emph{Quantitative compensation.}
There are $\delta_0, b_*>0$ and $C_0<\infty$ such that, for every
$\delta\in(0,\delta_0)$, there is an accuracy cost $C_\delta\ge1$
for which, uniformly on $\mathscr D\times\Gamma$,
\begin{align}\label{eq:application-source-interface}
JJ^*\le C_\delta^2AA^*+\delta^2I,
\qquad
C_\delta\le C_0\delta^{-b_*}.
\end{align}
\end{itemize}
Then, for every
\begin{align}\label{eq:akappa}
 0<a<(1+2b_*)^{-1},
 \qquad
 \varkappa=1-a(1+2b_*)>0,
\end{align}
there are $C,r_*>0$ such that, whenever $u\in\mathscr D$ and
$r=\|v-u\|_H\in(0,r_*]$, one can construct a jointly Borel exact
coupling $(X',Y')$ of $P_T(u,\cdot)$ and $P_T(v,\cdot)$ satisfying
\begin{align}\label{eq:compact-core-repair}
 \PP\{d(X',Y')\le Cr^{1+a}\}
 &\ge\gamma_T(\Gamma)-Cr^\varkappa,\notag\\
 \PP\{d(X',Y')>e^{C+H_T}r\}
 &\le Cr^\varkappa,
\end{align}
where $H_T$ is the synchronous amplification clock from
\Cref{ass:gap-amplification}.
\end{proposition}

\begin{proof}
Fix $u\in\mathscr D$ and $v\in\mathcal X$ with $r=\|v-u\|_H\in(0,r_*]$, where $r_*>0$ will be chosen below. Set $\delta=r^a$
and decrease $r_*$ so that $\delta<\delta_0$. Since enlarging $C_\delta$ preserves \eqref{eq:application-source-interface}, we may take $C_\delta=C_0\delta^{-b_*}$. 
Constants below may depend on the fixed data of the proposition but not on $r$ or $\delta$. The proof is divided into three steps. 

\emph{Step 1:  Finite-dimensional Gaussian compensation.}
For $1\le k\le m$ and $j\ge1$, let
\begin{align*}
e_{k,j}(t)=\sqrt{2/T}\cos\bigl((j-\tfrac12)\pi t/T\bigr)\boldsymbol e_k,
\end{align*}
where $\{\boldsymbol e_k\}$ is the standard basis of $\R^m$. Then $\{e_{k,j}\}$ is an orthonormal basis of $\mathsf H_T$, and
\begin{align}\label{eq:abstract-time-factor}
\sum_{k,j}\|\iota e_{k,j}\|_{L^\infty}^2=mT,\qquad
\|\iota(\operatorname{Id}-P_M)\|_{\mathsf H_T\to\Omega_T}\le C\sqrt{mT/M},
\end{align}
where $P_M$ is the projection onto the first $M$ time modes in each noise coordinate $k$. Since $A=D_\omega\Psi\circ\iota$ on $\mathcal N$ and $D_\omega\Psi$ is uniformly bounded on $\mathcal N$, setting $A_M=AP_M$ gives
\begin{align}\label{eq:abstract-factor-tail}
\|A(I-P_M)\|\le C\sqrt{mT/M}.
\end{align}
We first localize the observability estimate \eqref{eq:application-source-interface} to a small driver neighborhood of the core. Set $\rho_\delta=\frac{\delta}{1+C_\delta}$. After decreasing $r_*$ if necessary, we may assume that
$\rho_\delta\le\rho_0$ and hence $\mathscr D\times\mathscr N_{\rho_\delta}(\Gamma)
\subset\mathcal N$. For $(u,\omega)\in \mathscr D\times \mathscr N_{\rho_{\delta}}(\Gamma)$, choose $\bar\omega\in\Gamma$ with $\|\omega-\bar\omega\|_{\Omega_T}\le \rho_\delta$. The bounds on $D_\omega J$ and $D_\omega^2\Psi$ give
\begin{align*}
\|J(u,\omega)-J(u,\bar\omega)\|+\|A(u,\omega)-A(u,\bar\omega)\|\le C\rho_\delta.
\end{align*}
Hence \eqref{eq:application-source-interface} implies, upon changing fixed constants, 
\begin{align}\label{eq:localized-source}
JJ^*\le C C_\delta^2AA^*+C\delta^2I
\end{align}
throughout $\mathscr D\times \mathscr N_{\rho_{\delta}}(\Gamma)$  as  $(1+C_\delta^2)\rho_\delta^2\le C\delta^2$. Since
\begin{align*}
AA^*=A_MA_M^*+A(\operatorname{Id}-P_M)\bigl(A(\operatorname{Id}-P_M)\bigr)^*,
\end{align*}
choosing $M=\left\lceil C C_\delta^2\delta^{-2}\right\rceil$ 
and using \eqref{eq:abstract-factor-tail} in \eqref{eq:localized-source} give
\begin{align}\label{eq:finite-source-reduction}
JJ^*\le C C_\delta^2A_MA_M^*+C\delta^2I.
\end{align}
Thus the retained Gaussian factor $\mathcal{G}_{mM}$ corresponding to $P_M$ has dimension $mM\le C C_\delta^2\delta^{-2}$. Set
\begin{align*}
\lambda_\delta=\delta^2/C_\delta^2,\qquad
G_\delta=(A_MA_M^*+\lambda_\delta I)^{-1},\qquad
R_\delta=A_M^*G_\delta J.
\end{align*}
Since $R_\delta$ takes values in $P_M\mathsf H_T$, this produces a
finite-dimensional compensation 
\begin{align}\label{eq:j=AR+lambldaGJ}
J=A_MR_\delta+\lambda_\delta G_\delta J. 
\end{align} 
Functional calculus gives
\begin{align}\label{eq:abstract-regularized-inverse}
\|R_\delta\|\le C C_\delta,\qquad
\|J-A_MR_\delta\|\le C\delta,
\end{align}
and
\begin{align}\label{eq:tikhonov-auxiliary-bounds}
\|A_M^*G_\delta\|\le C C_\delta/\delta,\qquad
\|G_\delta J\|\le C C_\delta^2/\delta.
\end{align}

We now estimate the Jacobian of $R_{\delta}$. Let $z=(z_{k,j})$ be the standard Gaussian coordinates on $\mathcal{G}_{mM}$, so that
differentiation in $z_{k,j}$ corresponds to differentiation of the driver
in the direction $\iota e_{k,j}$. In particular,
\begin{align*}
\partial_{z_{k,j}}J
=D_\omega J[\iota e_{k,j}],\quad 
\partial_{z_{k,j}}A_Mh
=D_\omega^2\Psi[\iota e_{k,j},\iota P_Mh].
\end{align*}
Hence by the uniform bounds on $D_\omega^2\Psi$ and $D_\omega J$,
the uniform boundedness of $\iota P_M:\mathsf H_T\to\Omega_T$, and
\eqref{eq:abstract-time-factor},
\begin{align}\label{eq:abstract-factor-derivatives}
\sum_{k=1}^m\sum_{j=1}^M
\bigl(\|\partial_{z_{k,j}}A_M\|^2+\|\partial_{z_{k,j}}J\|^2\bigr)\le C,
\end{align}
uniformly in $M$. For each Gaussian coordinate $z_{k,j}$, differentiate
$R_\delta=A_M^*G_\delta J$. Since
\begin{align*}
\partial_{z_{k,j}}G_\delta
=
-G_\delta\bigl[
(\partial_{z_{k,j}}A_M)A_M^*
+A_M(\partial_{z_{k,j}}A_M)^*
\bigr]G_\delta,
\end{align*}
regrouping gives
\begin{align*}
\partial_{z_{k,j}} R_\delta
=(I-A_M^*G_\delta A_M)
(\partial_{z_{k,j}}A_M)^*G_\delta J 
-A_M^*G_\delta(\partial_{z_{k,j}}A_M)R_\delta
+A_M^*G_\delta(\partial_{z_{k,j}}J).
\end{align*}
Therefore \eqref{eq:abstract-regularized-inverse}--\eqref{eq:abstract-factor-derivatives} yield
\begin{align}\label{eq:abstract-inverse-derivative}
\left(\sum_{k=1}^m\sum_{j=1}^M\|\partial_{z_{k,j}}R_\delta\|^2\right)^{1/2}
\le C C_\delta^2/\delta.
\end{align}

\emph{Step 2: Localization and Cameron--Martin shift.}
Choose a jointly Borel cutoff $\chi_\delta:\Omega_T\to[0,1]$, fibrewise $C^1$ in the coordinates $(z_{k,j})$ of $\mathcal G_{mM}$, equal to one on $\Gamma$, supported in its $\rho_\delta$-neighborhood $\mathscr N_{\rho_{\delta}}(\Gamma)$ and satisfying
\begin{align}\label{eq:abstract-cutoff-derivative}
\|\nabla_z\chi_\delta\|\le C\rho_\delta^{-1}\le C(1+C_\delta)/\delta.
\end{align}
Define the Cameron--Martin shift 
\begin{align*}
h=-\chi_\delta R_\delta(v-u)\in P_M\mathsf H_T.
\end{align*}
Then \eqref{eq:application-source-interface}, \eqref{eq:akappa}, \eqref{eq:abstract-regularized-inverse}, \eqref{eq:abstract-inverse-derivative}, \eqref{eq:abstract-cutoff-derivative} and $\delta=r^a$ give
\begin{align}\label{eq:abstract-shift-bounds}
\|h\|_{\mathsf H_T}\le CrC_\delta\le Cr^{1-ab_*},\qquad
\|D_zh\|_{\mathrm{HS}}\le Cr\frac{C_\delta^2}{\delta}
\le Cr^{1-a(1+2b_*)}=Cr^\varkappa.
\end{align}
Since $Cr^{1-ab_*}\to0$, after decreasing $r_*$ one has $r_*+\|h\|_{\mathsf H_T}\le\varepsilon_*$ and $\|D_zh\|\le \|D_zh\|_{\mathrm{HS}}\le1/2$.

If $\omega\in\Gamma$, then $\chi_\delta=1$ and, since $h\in P_M\mathsf H_T$,
\begin{align*}
J(v-u)+Ah=(J-A_MR_\delta)(v-u).
\end{align*}
Thus \eqref{eq:repair-remainder}, \eqref{eq:j=AR+lambldaGJ}, \eqref{eq:abstract-regularized-inverse}, and $C_\delta=C_0\delta^{-b_*}$ imply
\begin{align}\label{eq:abstract-superlinear-repair}
\|\Psi(v,\omega+\iota h)-\Psi(u,\omega)\|_H
\le C\delta r+C C_\delta^2r^2\le Cr^{1+a}+Cr^{2-2ab_*}
\le Cr^{1+a},
\end{align}
because $a(1+2b_*)<1$. More generally, on the support of $\chi_\delta$,
\begin{align*}
J(v-u)+Ah=(1-\chi_\delta)J(v-u)
+\chi_\delta(J-A_MR_\delta)(v-u),
\end{align*}
so boundedness of $J$ and \eqref{eq:repair-remainder} give
\begin{align}\label{eq:abstract-linear-repair}
\|\Psi(v,\omega+\iota h)-\Psi(u,\omega)\|_H\le Cr,
\end{align}
after decreasing $r_*$ once more. Outside the support of $\chi_\delta$, one has $h=0$.

\emph{Step 3: Change of law and exact coupling.}
Write the Wiener path as $(z,z^{\dagger})$, where $z\in\R^{mM}$ is the retained standard Gaussian factor and $z^{\dagger}$ its independent complement, and define
\begin{align*}
\mathsf T(z,z^{\dagger})=(z+h(z,z^{\dagger}),z^{\dagger}),
\end{align*}
equivalently $\mathsf T\omega=\omega+\iota h(\omega)$. Since $\|D_zh\|\le1/2$, the finite-dimensional Gaussian change-of-variables estimate, applied conditionally on $z^{\dagger}$, gives the relative entropy estimate 
\begin{align*}
\Ent(\mathsf T_\#\gamma_T\mid\gamma_T)
\le C\E\bigl(\|h\|_{\mathsf H_T}^2+\|D_zh\|_{\mathrm{HS}}^2\bigr)\le C\left(r^2C_\delta^2+r^2C_\delta^4\delta^{-2}\right)
\le Cr^{2\varkappa};
\end{align*}
see for example \cite[Lemma~B.2]{LLWave}. Pinsker's inequality therefore yields
\begin{align}\label{eq:abstract-tv-repair}
\|\mathsf T_\#\gamma_T-\gamma_T\|_{\mathrm{TV}}\le Cr^\varkappa.
\end{align}
Let $W\sim\gamma_T$ and $\widetilde W=\mathsf T(W)$. Couple $\widetilde W$ maximally with a Wiener path $W'\sim\gamma_T$, retaining the relation $\widetilde W=\mathsf T(W)$. Then
\begin{align}\label{eq:abstract-acceptance}
\PP\{W'\ne\widetilde W\}\le \|\mathsf T_\#\gamma_T-\gamma_T\|_{\mathrm{TV}}\le  Cr^\varkappa.
\end{align}
Set $X'=\Psi(u,W),\, Y'=\Psi(v,W').$
Then $(X',Y')$ is an exact coupling of $P_T(u,\cdot)$ and $P_T(v,\cdot)$. On $\{W'=\widetilde W\}\cap\{W\in\Gamma\}$, \eqref{eq:abstract-superlinear-repair} gives $d(X',Y')\le Cr^{1+a}$, and hence
\begin{align*}
\PP\{d(X',Y')\le Cr^{1+a}\}
\ge\gamma_T(\Gamma)-Cr^\varkappa.
\end{align*}

On $\{W'=\widetilde W\}$, \eqref{eq:abstract-linear-repair} gives $d(X',Y')\le Cr$ whenever $\chi_\delta(W)\ne0$. If $\chi_\delta(W)=0$, then $\widetilde W=W'=W$, so the coupling is synchronous and \Cref{ass:gap-amplification} gives $d(X',Y')\le e^{H_T}r$. Therefore, after increasing $C$,
\begin{align*}
d(X',Y')\le e^{C+H_T}r
\end{align*}
throughout $\{W'=\widetilde W\}$, and \eqref{eq:abstract-acceptance} yields
\begin{align*}
\PP\{d(X',Y')>e^{C+H_T}r\}\le Cr^\varkappa.
\end{align*}

Finally, $\delta=r^a$, $M=\lceil C C_\delta^2\delta^{-2}\rceil$, the finite Gaussian coordinates, the regularized inverse, the cutoff, and the Gaussian density are jointly Borel in the parameters. A standard measurable maximal coupling therefore yields a jointly Borel exact endpoint coupling.
\end{proof}

\section{Spectral gap for the primitive equation}\label{sec:pe-gap}

We now verify the four structural assumptions of \Cref{thm:structural-gap} for the primitive equations. We first establish the Lyapunov and fractional-amplification estimates through a retained-dissipation cascade. We then prove the quantitative Malliavin observability estimate, with the Fourier synthesis and growing-Lie-word analysis. The local repair \Cref{ass:gap-local-repair} is obtained in two steps: positive-time regularization maps recurrent Lyapunov sublevels into compact analytic cores, and the polynomial accuracy--cost observability estimate yields, through \Cref{prop:compact-source-repair}, the required scale-dependent compensation on those cores. Finally, uniform irreducibility is proved independently by phase-complete scaling and saturation.

\subsection{The Lyapunov structure and fractional amplification moment}\label{sec:stationary-moments}

Throughout this subsection let $E(v)=\|v\|_H^2$ and
\begin{align}\label{eq:K}
 K(v)=1+\alpha\|v\|_8^8+\|\partial_zv\|_4^4+\|\nabla v\|_2^2,
\end{align}
where $\alpha\ge1$ will be fixed in \Cref{lem:cascade-direct}.  The quantity $K$ also controls the growth of synchronous differences. Recall the barotropic--baroclinic decomposition
\begin{align*}
 u(\boldsymbol x,z)=\bar u(\boldsymbol x)+\widetilde u(\boldsymbol x,z),\qquad \bar u(\boldsymbol x)=\frac{1}{2\pi}\int_{\mathbb T}u(\boldsymbol x,z)\,\dd z.  
\end{align*}
as in \eqref{eq:barotropic--baroclinic}. 

\begin{lemma}\label{lem:one-sided-growth}
For every $\varepsilon>0$ and smooth velocities $u,v$,
\begin{align}\label{eq:one-sided-K}
 |\ip{B(v,u)}v|\le\varepsilon\|\nabla v\|_2^2+C_\varepsilon\bigl(\|\nabla u\|_2^2+\|\partial_zu\|_4^4\bigr)\|v\|_2^2\le\varepsilon\|\nabla v\|_2^2+C_\varepsilon K(u)\|v\|_2^2.
\end{align}
Consequently, any two strong solutions $v,v^*$ with the same noise satisfy for some $C_{\nu}>0$ 
\begin{align}\label{eq:global-difference}
 \sup_{s\le T}\|v_s-v_s^*\|_2\le L_T(v)\|v_0-v_0^*\|_2,\qquad L_T(v)=\exp\left(C_\nu\int_0^T K(v_s)\,\dd s\right).
\end{align}
\end{lemma}

\begin{proof}
Using $u=\bar u+\widetilde u$, H\"older's, interpolation and Young's inequality, 
\begin{align*}
 \left|\int_{\T^3}(v\cdot\nabla_h)\bar u\cdot v\,\dd\boldsymbol x\,\dd z\right|
 &\le\int_{\T^2}\|v\|_{L_z^2}^2|\nabla_h\bar u|\,\dd\boldsymbol x
 \le C\|\nabla u\|_2\|v\|_2\|\nabla v\|_2\\
 &\le\varepsilon\|\nabla v\|_2^2+C_\varepsilon\|\nabla u\|_2^2\|v\|_2^2.
\end{align*}
By Poincar\'e's inequality, $\|w(v)\|_{L_z^\infty}\le C\|\nabla_hv\|_{L_z^2}$ and $\|\widetilde u\|_{L_z^\infty}\le C\|\partial_zu\|_{L_z^2}$. After periodic integration by parts in the horizontal variables,
\begin{align*}
 \left|\int_{\T^3}\bigl((v\cdot\nabla_h)\widetilde u+w(v)\partial_zu\bigr)\cdot v\,\dd\boldsymbol x\,\dd z\right|
 &\le C\int_{\T^2}\|\partial_zu\|_{L_z^2}\|v\|_{L_z^2}\|\nabla_hv\|_{L_z^2}\,\dd\boldsymbol x\\
 &\le C\|v\|_2^{1/2}\|\partial_zu\|_4\|\nabla v\|_2^{3/2}
 \le\varepsilon\|\nabla v\|_2^2+C_\varepsilon\|\partial_zu\|_4^4\|v\|_2^2.
\end{align*}
Combining the two estimates proves \eqref{eq:one-sided-K}.

For two strong solutions $v,v^*$ driven by the same noise, set $u=v^*-v$. Then $B(v^*,v^*)-B(v,v)=B(v^*,u)+B(u,v)$. Since $\ip{B(v^*,u)}u=0$, the difference equation and \eqref{eq:one-sided-K} yield
\begin{align*}
 \frac12\frac{d}{dt}\|u\|_2^2+\nu\|\nabla u\|_2^2
 \le\varepsilon\|\nabla u\|_2^2+C_\varepsilon K(v)\|u\|_2^2.
\end{align*}
Choosing $\varepsilon<\nu$ and applying Gronwall's inequality gives \eqref{eq:global-difference}.
\end{proof}

Thus, to obtain the fractional amplification estimate, it remains to control the time integral of $K$. Denote its associated dissipation by
\begin{align}\label{eq:D}
 D(v)=\int_{\T^3}|v|^6|\nabla v|^2\,\dd x
 +\int_{\T^3}|\partial_zv|^2|\nabla\partial_zv|^2\,\dd x+\|\Delta v\|_2^2.
\end{align}
The full weighted gradient dissipation $\int|v|^6|\nabla v|^2$ is essential for closing the estimate of $\|\partial_zv\|_4$. For a pair $(\kappa,\eta)$, set
$\mathcal V=e^{\kappa E}K^\eta$. We call $(\kappa,\eta)$ admissible if, along every smooth solution,
\begin{align}\label{eq:admissible}
 \dd\mathcal V
 \le\mathcal V\left(C-b_1\|\nabla v\|_2^2-b_2\frac{D}{K}\right)\dd t
 +\mathcal V\left(\kappa H+\eta\frac{m}{K}\right)\dd W
\end{align}
for some $b_1,b_2>0$, where $H$ and $m$ are the martingale coefficients of $E$ and $K$ introduced below. The next three lemmas establish the estimate for $K$, verify admissibility, and provide the coercivity needed for the Lyapunov and moment bounds.

\begin{lemma}\label{lem:cascade-direct}
There exist $\alpha\ge1$, $c>0$ and $C_1<\infty$, depending only on the viscosity and geometry, such that along every strong solution $v$,
\begin{align}\label{eq:cascade-drift}
 \dd K+cD\,\dd t\le(C_0+C_1\|\nabla v\|_2^2)K\,\dd t+m\,\dd W,\qquad |m|^2\le C_QK^{7/4}.
\end{align}
Here $C_0$ is finite for every prescribed smooth finite-rank $Q$.
\end{lemma}

\begin{proof}
The estimates of $\|v\|_8$, $\|\partial_zv\|_4$, and $\|\nabla v\|_2$ form a triangular system: the weighted $L^8$ dissipation controls the source in the $\|\partial_zv\|_4$ equation, and these two estimates then control the $\|\nabla v\|_2$ equation. Denote by $D_8=\int_{\T^3}|v|^6|\nabla v|^2\,\dd x$, $D_{z,4}=\int_{\T^3}|\partial_zv|^2|\nabla\partial_zv|^2\,\dd x$, and $D_4=\int_{\T^3}|v|^2|\nabla v|^2\,\dd x$.

\emph{Step 1: Estimate of $\|v\|_8$.}
It\^o's formula gives
\begin{align}\label{eq:L8-Ito-identity}
 \dd\|v\|_8^8+8\nu D_8\,\dd t
 &+48\nu\int_{\T^3}|v|^4|v\cdot\nabla v|^2\,\dd x\,\dd t
 =-8\int_{\T^3}\nabla_hp\cdot|v|^6v\,\dd x\,\dd t\notag\\
 &\quad+\sum_{j=1}^m\int_{\T^3}\left[4|v|^6|b_j|^2+24|v|^4(v\cdot b_j)^2\right]\dd x\,\dd t
 +m_8\,\dd W_t.
\end{align}
Here $(m_8)_j=8\int_{\T^3}|v|^6v\cdot b_j\,\dd x$, so $|m_8|^2\le C_Q\|v\|_8^{14}$. The transport contribution vanishes by incompressibility. Vertical averaging of the momentum equation and horizontal integration by parts give
\begin{align*}
 -2\pi\Delta_hp&=\nabla_h\cdot\int_{\mathbb T}\bigl((v\cdot\nabla_h)v+w(v)\partial_zv\bigr)\,\dd z,\\
 \int_{\mathbb T}\bigl((v\cdot\nabla_h)v+w(v)\partial_zv\bigr)\,\dd z
 &=\nabla_h\cdot\int_{\mathbb T}v\otimes v\,\dd z.
\end{align*}
Hence Calder\'on--Zygmund and H\"older inequalities yield
\begin{align}\label{eq:pressure-L8}
 \|\nabla_hp\|_{L_h^{8/5}}
 \le C\int_{\mathbb T}\|v(\cdot,z)\|_{L_h^8}\|\nabla_hv(\cdot,z)\|_{L_h^2}\,\dd z
 \le C\|v\|_8\|\nabla v\|_2.
\end{align}
The embedding $W^{1,8/7}(\T^2)\hookrightarrow L^{8/3}(\T^2)$ gives
\begin{align}\label{eq:pressure-L8-2}
 \left\|\int_{\mathbb T}|v|^6v\,\dd z\right\|_{L_h^{8/3}}
 \le C\bigl(\|v\|_8^7+\|v\|_8^3D_8^{1/2}\bigr).
\end{align}
Indeed, $|\nabla_h(|v|^6v)|\le7|v|^6|\nabla_hv|=7|v|^3(|v|^3|\nabla_hv|)$, and horizontal H\"older followed by H\"older in $z$ gives the second term. Thus
\begin{align*}
 \left|\int_{\T^3}\nabla_hp\cdot|v|^6v\,\dd x\right|
 \le\varepsilon D_8+C_\varepsilon(1+\|\nabla v\|_2^2)\|v\|_8^8.
\end{align*}
Choosing $\varepsilon$ sufficiently small gives
\begin{align}\label{eq:velocity-cascade}
 \dd\|v\|_8^8+c_8D_8\,\dd t
 \le C(1+\|\nabla v\|_2^2)\|v\|_8^8\,\dd t+C_Q(1+\|v\|_8^6)\,\dd t+m_8\,\dd W.
\end{align}

\emph{Step 2: Estimate of $\|\partial_zv\|_4$.}
It\^o's formula gives
\begin{align*}
 \dd\|\partial_zv\|_4^4+4\nu D_{z,4}\,\dd t
 &+8\nu\sum_{k=1}^3\int_{\T^3}(\partial_zv\cdot\partial_k\partial_zv)^2\,\dd x\,\dd t\\
 &=-4\int_{\T^3}|\partial_zv|^2\partial_zv\cdot((\partial_zv\cdot\nabla_h)v)\,\dd x\,\dd t
 +4\int_{\T^3}(\nabla_h\cdot v)|\partial_zv|^4\,\dd x\,\dd t\\
 &\quad+\sum_{j=1}^m\int_{\T^3}\left[2|\partial_zv|^2|\partial_zb_j|^2
 +4(\partial_zv\cdot\partial_zb_j)^2\right]\dd x\,\dd t+m_a\,\dd W,
\end{align*}
where $(m_a)_j=4\int_{\T^3}|\partial_zv|^2\partial_zv\cdot\partial_zb_j\,\dd x$ and $|m_a|^2\le C_Q\|\partial_zv\|_4^6$. Horizontal integration by parts gives
\begin{align*}
 &\left|-4\int_{\T^3}|\partial_zv|^2\partial_zv\cdot((\partial_zv\cdot\nabla_h)v)\,\dd x
 +4\int_{\T^3}(\nabla_h\cdot v)|\partial_zv|^4\,\dd x\right|\\
 &\qquad\le C\int_{\T^3}|v||\partial_zv|^3|\nabla_h\partial_zv|\,\dd x
 \le\frac{\varepsilon}{2}D_{z,4}+C_\varepsilon\int_{\T^3}|v|^2|\partial_zv|^4\,\dd x.
\end{align*}
Integration by parts in $z$ yields
\begin{align*}
 \int_{\T^3}|v|^2|\partial_zv|^4\,\dd x
 \le3\int_{\T^3}|v|^3|\partial_zv|^2|\partial_z^2v|\,\dd x
 \le\delta D_{z,4}+C_\delta D_8.
\end{align*}
Choosing $\delta$ and then $\varepsilon$ sufficiently small gives
\begin{align}\label{eq:cascade-A4}
 \dd\|\partial_zv\|_4^4+c_aD_{z,4}\,\dd t
 \le CD_8\,\dd t+C_Q(1+\|\partial_zv\|_4^2)\,\dd t+m_a\,\dd W.
\end{align}

\emph{Step 3: Estimate of $\|\nabla v\|_2$.}
The evolution of $\|\nabla v\|_2^2$ is
\begin{align}\label{eq:H1-exact}
 \dd\|\nabla v\|_2^2+2\nu\|\Delta v\|_2^2\,\dd t
 &=2\int_{\T^3}(v\cdot\nabla_h)v\cdot\Delta_hv\,\dd x\,\dd t
 +2\int_{\T^3}w(v)\partial_zv\cdot\Delta_hv\,\dd x\,\dd t\notag\\
 &\quad-2\int_{\T^3}((\partial_zv\cdot\nabla_h)v)\cdot\partial_zv\,\dd x\,\dd t
 +2\int_{\T^3}(\nabla_h\cdot v)|\partial_zv|^2\,\dd x\,\dd t\notag\\
 &\quad+\sum_{j=1}^m\|\nabla b_j\|_2^2\,\dd t+m_v\,\dd W.
\end{align}
Here $(m_v)_j=2\langle\nabla v,\nabla b_j\rangle$, so $|m_v|^2\le C_Q\|\nabla v\|_2^2$. The horizontal convection term satisfies
\begin{align*}
 \left|2\int_{\T^3}(v\cdot\nabla_h)v\cdot\Delta_hv\,\dd x\right|
 \le CD_4^{1/2}\|\nabla\nabla_hv\|_2
 \le\varepsilon\|\nabla\nabla_hv\|_2^2+C_\varepsilon D_4.
\end{align*}
The horizontal Ladyzhenskaya inequality gives
$\|w(v)\partial_zv\|_2\le C\|\nabla_hv\|_2^{1/2}(\|\nabla_hv\|_2+\|\nabla\nabla_hv\|_2)^{1/2}\|\partial_zv\|_4$, and hence
\begin{align*}
 \left|2\int_{\T^3}w(v)\partial_zv\cdot\Delta_hv\,\dd x\right|
 \le\varepsilon\|\nabla\nabla_hv\|_2^2
 +C_\varepsilon\|\nabla v\|_2^2(1+\|\partial_zv\|_4^4).
\end{align*}
Finally,
\begin{align*}
 &\left|-2\int_{\T^3}((\partial_zv\cdot\nabla_h)v)\cdot\partial_zv\,\dd x
 +2\int_{\T^3}(\nabla_h\cdot v)|\partial_zv|^2\,\dd x\right|\\
 &\qquad\le C\int_{\T^3}|v||\partial_zv||\nabla_h\partial_zv|\,\dd x
 \le\varepsilon\|\nabla\partial_zv\|_2^2+C_\varepsilon D_4.
\end{align*}
Choosing $\varepsilon$ sufficiently small gives
\begin{align}\label{eq:cascade-V}
 \dd\|\nabla v\|_2^2+c_v\|\Delta v\|_2^2\,\dd t
 \le CD_4\,\dd t+C\|\nabla v\|_2^2(1+\|\partial_zv\|_4^4)\,\dd t+C_Q\,\dd t+m_v\,\dd W.
\end{align}

\emph{Step 4: Conclusion.}
By Young's inequality, $D_4\le\delta D_8+C_\delta\|\nabla v\|_2^2$. Choose $\alpha$ sufficiently large so that $\alpha c_8D_8$ from \eqref{eq:velocity-cascade} absorbs the $CD_8$ term in \eqref{eq:cascade-A4}, and then choose $\delta>0$ sufficiently small to absorb the corresponding $D_4$ contribution in \eqref{eq:cascade-V}. All remaining drift terms are bounded by $(C_0+C_1\|\nabla v\|_2^2)K$. Hence
\begin{align}\label{eq:combined-cascade}
 \dd K+c_DD\,\dd t\le(C_0+C_1\|\nabla v\|_2^2)K\,\dd t+m\,\dd W,
\end{align}
where $m=\alpha m_8+m_a+m_v$. Moreover,
\begin{align*}
 |m|^2
 &\le C(\alpha^2|m_8|^2+|m_a|^2+|m_v|^2)\\
 &\le C_Q\bigl(\|v\|_8^{14}+\|\partial_zv\|_4^6+\|\nabla v\|_2^2\bigr)
 \le C_QK^{7/4}.
\end{align*}
This proves \eqref{eq:cascade-drift}.
\end{proof}

The remaining growth term $C_1\|\nabla v\|_2^2K$ is controlled by coupling \eqref{eq:cascade-drift} with the energy dissipation.

\begin{lemma}\label{lem:admissibility}
Every smooth finite-rank $Q$ admits an admissible pair $(\kappa,\eta)$. One sufficient choice is
\begin{align}\label{eq:cascade-choice}
 \kappa>0,\qquad \frac{4\kappa\|Q\|_{\mathcal L(\R^m,H)}^2}{\lambda_1}\le\frac{\nu}{2},
 \qquad 0<\eta\le1,\qquad \eta C_1\le\frac{\kappa\nu}{2}
\end{align}
where $\lambda_1$ is the smallest eigenvalue of $A$ on $H$. 
\end{lemma}

\begin{proof}
Along a strong solution $v_t$, the energy satisfies
\begin{align}\label{eq:energy-balance}
 \dd E=(-2\nu\|\nabla v\|_2^2+\|Q\|_{\mathrm{HS}}^2)\,\dd t+H\,\dd W,
\end{align}
where $H_j=2\langle v,b_j\rangle$ and $|H|^2\le4\|Q\|_{\mathcal L(\R^m,H)}^2E$. Applying It\^o's formula to $\mathcal V$ and using \eqref{eq:cascade-drift} and \eqref{eq:energy-balance} gives
\begin{align*}
 \dd\mathcal V
 &\le\mathcal V\left[
 (-2\kappa\nu+\eta C_1)\|\nabla v\|_2^2-\eta c\frac{D}{K}
 +\kappa\|Q\|_{\mathrm{HS}}^2+\eta C_0\right.\\
 &\hspace{28mm}\left.
 +\frac{\kappa^2}{2}|H|^2+\kappa\eta\frac{H\cdot m}{K}
 +\frac{\eta(\eta-1)}{2}\frac{|m|^2}{K^2}\right]\dd t
 +\mathcal V\left(\kappa H+\eta\frac{m}{K}\right)\dd W.
\end{align*}
Since $0<\eta\le1$, the last term inside the drift is nonpositive. Moreover, $|m|^2/K^2\le C_Q$, and Young's inequality gives
$\kappa\eta H\cdot m/K\le\frac{\kappa^2}{2}|H|^2+\frac{\eta^2}{2}|m|^2/K^2\le\frac{\kappa^2}{2}|H|^2+C$.
Using Poincar\'e's inequality and the bound on $H$ therefore yields
\begin{align*}
 \dd\mathcal V
 \le\mathcal V\left[
 \left(-2\kappa\nu+\eta C_1+\frac{4\kappa^2\|Q\|_{\mathcal L(\R^m,H)}^2}{\lambda_1}\right)\|\nabla v\|_2^2
 -\eta c\frac{D}{K}+C\right]\dd t+\mathcal V\left(\kappa H+\eta\frac{m}{K}\right)\dd W.
\end{align*}
Under \eqref{eq:cascade-choice}, the coefficient of $\|\nabla v\|_2^2$ is at most $-\kappa\nu$. Hence \eqref{eq:admissible} holds with $b_1=\kappa\nu$ and $b_2=\eta c$. 
\end{proof}

Enlarging $C_\nu$ if necessary so that $C_\nu>2C_1/\nu$, there exists $\eta_0=\eta_0(\nu,Q)>0$ such that, for every $0<\eta\le\eta_0$, the pair $(4C_\nu\eta,4\eta)$ satisfies the strict versions of \eqref{eq:cascade-choice}, with $4\eta<1/3$. After decreasing $\eta_0$ if necessary, there is also $\mathsf k>4C_\nu\eta$ such that $\mathsf k\|Q\|_{\mathcal L(\R^m,H)}^2<\nu\lambda_1/2$. From now on fix $0<\eta\le\eta_0$ and set $\kappa=C_\nu\eta$. Then
$U=\mathcal V=e^{\kappa E}K^\eta$.

Set $\mathcal B(v)=\|\nabla v\|_2^2+D(v)/K(v)$. The next lemma provides the coercivity needed below.

\begin{lemma}\label{lem:cascade-coercivity}
There exists $C>0$ such that
\begin{align}\label{eq:cascade-coercivity}
 E\le C\mathcal B,\qquad \frac{K^{2/9}}{(1+E)^{8/9}}\le C(1+\mathcal B),\qquad K^{2/17}\le C(1+\mathcal B).
\end{align}
\end{lemma}

\begin{proof}
Poincar\'e's inequality gives $E\le C\|\nabla v\|_2^2\le C\mathcal B$. Integration by parts in $z$ and H\"older's inequality give
\begin{align*}
 \|\partial_zv\|_4^4
 &\le3\left(\int|v|^2|\nabla v|^2\right)^{1/2}
 \left(\int|\partial_zv|^2|\nabla\partial_zv|^2\right)^{1/2}\\
 &\le3\left(\int|v|^6|\nabla v|^2\right)^{1/6}\|\nabla v\|_2^{2/3}
 \left(\int|\partial_zv|^2|\nabla\partial_zv|^2\right)^{1/2}\le CD^{2/3}\mathcal B^{1/3}\le CK^{2/3}\mathcal B.
\end{align*}
The Sobolev inequality gives $\|v\|_{24}^8\le C(D_8+\|v\|_8^8)$. Interpolation between $L^2$ and $L^{24}$ then yields
\begin{align*}
 (\|v\|_8^8)^{11/9}\le E^{8/9}\|v\|_{24}^8
 \le CE^{8/9}K(1+\mathcal B).
\end{align*}
Hence Young's inequality gives
\begin{align*}
 K
 &\le1+CE^{8/11}K^{9/11}(1+\mathcal B)^{9/11}+CK^{2/3}\mathcal B+\mathcal B\\
 &\le1+\frac12K+CE^4(1+\mathcal B)^{9/2}+C\mathcal B^3+\mathcal B.
\end{align*}
Moving $\frac12K$ to the left and raising to the power $2/9$ gives
$K^{2/9}\le C[1+\mathcal B^{2/3}+E^{8/9}(1+\mathcal B)]
\le C(1+E)^{8/9}(1+\mathcal B)$, proving the second inequality in \eqref{eq:cascade-coercivity}. Combining it with $E\le C\mathcal B$ gives $K^{2/9}\le C(1+\mathcal B)^{17/9}$, and raising to the power $9/17$ proves the third.
\end{proof}

\begin{lemma}\label{lem:mixing-foster}
For $p\in\{1,2,4\}$ there are $c_p,C_p>0$ such that, whenever $U(v_0)<\infty$,
\begin{align}\label{eq:mixing-foster}
\begin{split}
 &P_tU^p(v_0)=\mathbb E_{v_0}U^p(v_t)\le e^{-c_pt}U(v_0)^p+C_p,\\
 \text{ and }\quad  &\mathbb E_{v_0}\int_0^tU(v_s)^p\left(1+\|\nabla v_s\|_2^2+\frac{D(v_s)}{K(v_s)}\right)\dd s
 \le C_p\bigl(U(v_0)^p+t\bigr),
\end{split}
\end{align}
where $v_t=\Phi_t(v_0,W)$. Moreover, for every $v_0\in H^1$ and $0<t\le1$,
\begin{align}\label{eq:mixing-coming-down}
 P_tU^p(v_0)\le C_pe^{\mathsf k\|v_0\|_H^2}\bigl(1+t^{-9p\eta/2}\bigr).
\end{align}
In particular $P_1U(v_0)\le Ce^{\mathsf k\|v_0\|_H^2}$ for every $v_0\in H^1$.
\end{lemma}

\begin{proof}
All It\^o inequalities below are first justified under the standard strong-norm localization and then passed to the limit by Fatou's lemma.

\emph{Step 1: Foster and dissipation bounds.}
Since $U^p=e^{p\kappa E}K^{p\eta}$ and $(p\kappa,p\eta)$ is admissible for $p\in\{1,2,4\}$, \Cref{lem:admissibility} gives
\begin{align*}
 \dd U^p\le U^p(C_p-c_p\mathcal B)\,\dd t
 +pU^p\left(\kappa H+\eta\frac{m}{K}\right)\dd W.
\end{align*}
By \Cref{lem:cascade-coercivity}, $U^p$ is bounded on every sublevel of $\mathcal B$. After adjusting $c_p,C_p$,
\begin{align}\label{ito-inequality-lemma4.6}
 \dd U^p+c_pU^p(1+\mathcal B)\,\dd t
 \le C_p\,\dd t+pU^p\left(\kappa H+\eta\frac{m}{K}\right)\dd W.
\end{align}
Integrating in time and taking expectations gives
\begin{align*}
 \mathbb E_{v_0}U^p(v_t)+c_p\mathbb E_{v_0}\int_0^tU^p(v_s)(1+\mathcal B(v_s))\,\dd s
 \le U(v_0)^p+C_pt,
\end{align*}
which gives the second estimate in \eqref{eq:mixing-foster}. Multiplying \eqref{ito-inequality-lemma4.6} by $e^{c_pt}$ and discarding the remaining nonnegative dissipation yields the first estimate.

\emph{Step 2: Exponential energy bound.}
By \eqref{eq:energy-balance},
\begin{align*}
 \dd(e^{\mathsf kE})
 &\le e^{\mathsf kE}\left[-2\mathsf k\nu\|\nabla v\|_2^2
 +\mathsf k\|Q\|_{\mathrm{HS}}^2
 +2\mathsf k^2\|Q\|_{\mathcal L(\R^m,H)}^2E\right]\dd t
 +\mathsf ke^{\mathsf kE}H\,\dd W.
\end{align*}
Since $\mathsf k\|Q\|_{\mathcal L(\R^m,H)}^2<\nu\lambda_1/2$, Poincar\'e's inequality gives
\begin{align*}
 \dd(e^{\mathsf kE})\le(-ce^{\mathsf kE}+C)\,\dd t+\mathsf ke^{\mathsf kE}H\,\dd W.
\end{align*}
Hence, for $0\le t\le1$,
\begin{align*}
 \mathbb E_{v_0}e^{\mathsf kE(v_t)}
 &\le e^{-ct}e^{\mathsf kE(v_0)}+\frac{C}{c}(1-e^{-ct}),\\
 \sup_{0\le t\le1}\mathbb E_{v_0}e^{\mathsf kE(v_t)}
 &\le Ce^{\mathsf kE(v_0)}=:M.
\end{align*}

\emph{Step 3: Coming down from the energy space.}
Fix $p\in\{1,2,4\}$ and set $\beta=2/(9p\eta)$. Let $g(t)=\mathbb E_{v_0}U^p(v_t)$. Since $p\kappa<\mathsf k$, one has $e^{p\kappa E}(1+E)^{4p\eta}\le Ce^{\mathsf kE}$ and 
\begin{align*}
 \sup_{0\le t\le1}\mathbb E_{v_0}\left[e^{p\kappa E(v_t)}(1+E(v_t))^{4p\eta}\right]\le M.
\end{align*}
Since $p\eta\beta=2/9$,
\begin{align*}
 \frac{(U^p)^{1+\beta}}{\bigl(e^{p\kappa E}(1+E)^{4p\eta}\bigr)^\beta}
 =U^p\frac{K^{2/9}}{(1+E)^{8/9}}.
\end{align*}
H\"older's inequality therefore gives
\begin{align}\label{lemma4.6-holder}
 g(t)^{1+\beta}
 \le M^\beta\mathbb E_{v_0}\left[
 U^p(v_t)\frac{K(v_t)^{2/9}}{(1+E(v_t))^{8/9}}\right].
\end{align}
By \eqref{eq:cascade-coercivity} and \eqref{ito-inequality-lemma4.6},
\begin{align*}
 \dd U^p
 \le\left(C-cU^p\frac{K^{2/9}}{(1+E)^{8/9}}\right)\dd t
 +pU^p\left(\kappa H+\eta\frac{m}{K}\right)\dd W.
\end{align*}
Integrating on $[r,t]\subset(0,1]$, taking expectations and using \eqref{lemma4.6-holder} gives
\begin{align*}
 g(t)+cM^{-\beta}\int_r^tg(s)^{1+\beta}\,\dd s
 \le g(r)+C(t-r).
\end{align*}
Thus $\frac{d}{dt}g(t)\le C-cM^{-\beta}g(t)^{1+\beta}$ for almost every $t>0$. Whenever $g(t)^{1+\beta}\ge(2C/c)M^\beta$,
\begin{align*}
 \frac{d}{dt}g(t)^{-\beta}\ge\frac{\beta c}{2}M^{-\beta}.
\end{align*}
It follows, independently of $g(0)$, that
\begin{align*}
 g(t)
 \le CM^{\beta/(1+\beta)}+CMt^{-1/\beta}
 \le CM(1+t^{-1/\beta})\le C_pe^{\mathsf kE(v_0)}\bigl(1+t^{-9p\eta/2}\bigr),
 \qquad 0<t\le1.
\end{align*}

\emph{Step 4: Approximation.}
The preceding calculations are first valid for smooth initial states. If $U(v_0)<\infty$, approximate $v_0$ by $e^{-\varepsilon A}v_0$. Then $e^{-\varepsilon A}v_0\to v_0$ in the norms entering $U$, so $U(e^{-\varepsilon A}v_0)\to U(v_0)$. Common-driver stability and Fatou's lemma give \eqref{eq:mixing-foster}.

For arbitrary $v_0\in H^1$, let $v_0^n$ be Fourier truncations. Then $v_0^n\to v_0$ in $H^1$, $\|v_0^n\|_H\le\|v_0\|_H$, and $U(v_0^n)<\infty$. For each fixed $t>0$, common-driver stability gives $\Phi_t(v_0^n,W)\to\Phi_t(v_0,W)$ in $H^1$, while the $L^8$ and vertical $L^4$ terms in $K$ are lower semicontinuous. Hence
\begin{align*}
 P_tU^p(v_0)\le\liminf_{n\to\infty}P_tU^p(v_0^n)\le C_pe^{\mathsf k\|v_0\|_H^2}\bigl(1+t^{-9p\eta/2}\bigr),
 \qquad 0<t\le1.
\end{align*}
This proves \eqref{eq:mixing-coming-down}. Taking $p=1$ and $t=1$ gives the final assertion.
\end{proof}

We now return to the amplification estimate in \Cref{lem:one-sided-growth}. Since $\log L_T(v)$ is proportional to $\int_0^TK(v_s)\,\dd s$, it remains to control a fractional moment of this integral.

\begin{proposition}\label{lem:mixing-conditional-clock}
For every fixed $T<\infty$, there exists $C_T<\infty$ such that, for $v_t=\Phi_t(v_0,W)$,
\begin{align}\label{eq:mixing-conditional-clock}
 \mathbb E_{v_0}\left(\int_0^TK(v_s)\,\dd s\right)^\eta
 \le C_TU(v_0),\qquad U(v_0)<\infty.
\end{align}
\end{proposition}

\begin{proof}
The martingale coefficient of $U=e^{\kappa E}K^\eta$ is $G_U=U(\kappa H+\eta m/K)$. Since $|H|^2\le C_QE$, $|m|^2\le C_QK^{7/4}$ and $K\ge1$,
\begin{align}\label{eq:U-bracket-bound}
 |G_U|^2\le CU^2(1+E).
\end{align}
By \Cref{lem:admissibility,lem:cascade-coercivity}, after adjusting constants,
\begin{align*}
 \dd U+cU(1+\mathcal B)\,\dd t\le C\,\dd t+G_U\,\dd W.
\end{align*}
Dropping the nonnegative dissipation, taking the supremum over $[0,T]$, and applying the Burkholder--Davis--Gundy and Young inequalities together with \eqref{eq:U-bracket-bound} gives
\begin{align*}
 \mathbb E_{v_0}\sup_{s\le T}U(v_s)
 \le U(v_0)+C_T+\frac12\mathbb E_{v_0}\sup_{s\le T}U(v_s)+C\mathbb E_{v_0}\int_0^TU(v_s)(1+E(v_s))\,\dd s.
\end{align*}
Since $E\le C\mathcal B$, the last term is bounded by $C_TU(v_0)$ by \eqref{eq:mixing-foster} with $p=1$. Hence
\begin{align}\label{eq:maximal-U}
 \mathbb E_{v_0}\sup_{0\le s\le T}U(v_s)\le C_TU(v_0).
\end{align}
Finally, since $0<\eta<1$ and $K^\eta\le U$,
\begin{align*}
 \left(\int_0^TK(v_s)\,\dd s\right)^\eta
 \le T^\eta\sup_{0\le s\le T}K(v_s)^\eta
 \le T^\eta\sup_{0\le s\le T}U(v_s).
\end{align*}
Taking expectations and applying \eqref{eq:maximal-U} proves \eqref{eq:mixing-conditional-clock}.
\end{proof}

Combining \eqref{eq:mixing-conditional-clock} with \Cref{lem:one-sided-growth}, consider two solutions $v,v^*$ driven by the same noise and, after interchanging them if necessary, assume $U(v_0)\le U(v_0^*)$. Set $H_T=C_\nu\int_0^TK(v_s)\,\dd s$. Then, for $0\le t\le T$, $\|v_t-v_t^*\|_H\le e^{H_T}\|v_0-v_0^*\|_H$, while
\begin{align}\label{eq:eht}
 \mathbb EH_T^\eta\le C_TU(v_0)=C_T\min\{U(v_0),U(v_0^*)\}.
\end{align}
Hence \Cref{ass:gap-amplification} holds with $\theta=\eta$.

\subsection{A quantitative Malliavin observability estimate}\label{sec:quantitative-Malliavin-observability}
For a reference pair $(u,\omega)$, write $v_t=\Phi_t(u,\omega)$ and
let $J_{s,t}$ be the tangent flow along $v$. The endpoint
noise-control operator is
\begin{align}\label{eq:control-operator}
 \mathcal A_T h=\int_0^T J_{s,T}Qh(s)\dd s,\qquad h\in L^2(0,T;\R^m)
\end{align}
which coincides with the Malliavin derivative of $\Phi_T(u,\omega)$ in the direction $h$. 
 For $r\ge0$ and $a\in\R$, let
\begin{align}\label{eq:analytic-space}
\mathbb H_{r,a}=\{v\in H:\|v\|_{r,a}:=\|e^{r\Lambda}v\|_{H^a}<\infty\}, \quad \text{ with }\Lambda=(-\Delta)^{1/2}.   
\end{align}
The main result we shall prove in this section is the following: 
\begin{proposition}\label{prop:polynomial-source}
Assume the forcing hypotheses of \Cref{thm:main}. Fix $T>0$ and an
integer $s\ge6$. Let $\mathscr D\subset H^s$ be compact and suppose that,
for some $\rho_{\mathscr D}>0$ and $R_{\mathscr D}<\infty$,
\begin{align*}
 \sup_{u\in\mathscr D}\|u\|_{\rho_{\mathscr D},s}\le R_{\mathscr D}.
\end{align*}
There are a compact set $\Gamma\subset C_0([0,T];\mathbb R^m)$ of
strictly positive Wiener measure and constants $C,b,\delta_0>0$
such that, for every $(u,\omega)\in \mathscr D\times\Gamma$ and
$0<\delta<\delta_0$,
\begin{align}\label{eq:target-source}
 J_{0,T}J_{0,T}^*
 \le C_\delta^2\mathcal A_T\mathcal A_T^*+\delta^2I,\qquad
 C_\delta\le C\delta^{-b}.
\end{align}
Moreover, for some $\rho>0$ and $R<\infty$, $\mathscr D\times\Gamma\subset\mathcal C_{\rho,s}(R, T)$
where $\mathcal C_{\rho,s}(R, T)$ is the analytic reference class
introduced in \eqref{eq:analytic-reference-core} and  the regular-dependence
hypotheses of \Cref{prop:compact-source-repair} hold uniformly on $\mathscr D\times\mathscr N_{\rho}(\Gamma)$. 
\end{proposition}

For $\xi\in H$ set $p_t=J_{t,T}^*\xi$. Since
\begin{align*}
 p_0=J_{0,T}^*\xi,\qquad (\mathcal A_T^*\xi)(s)=Q^*p_s,
\end{align*}
\eqref{eq:target-source} is equivalent to
\begin{align}\label{eq:adjoint-source-target}
 \|p_0\|_H^2
 \le C_\delta^2\|Q^*p\|_{L^2(0,T;\R^m)}^2+\delta^2\|\xi\|_H^2.
\end{align}
We prove \eqref{eq:adjoint-source-target} in three stages. Repeated
application of the observable hierarchy first converts the directly
observed quantity $Q^*p$ into Lie coefficients. These coefficients are
then assembled into the low-frequency component $P_Np$ using the
quantitative Fourier--Lie synthesis proved in
\Cref{subsec:fourier-lie-generation}. Finally, parabolic dissipation
controls the complementary high frequencies and recovers $p_0$.

We use the bracket convention $[U,V]=DV\,U-DU\,V$ and
$\operatorname{ad}_U V=[U,V]$. Let $F(u)=-\nu Au-B(u,u)$. For a word
$\mathtt{w}=(a_1,\ldots,a_r)$ with
$a_\ell\in\{F,b_1,\ldots,b_m\}$, write the iterated Lie bracket
\begin{align}\label{eq:iteratedLiebracket}
 \operatorname{ad}_{\mathtt{w}} b_j
 =\operatorname{ad}_{a_1}\cdots\operatorname{ad}_{a_r}b_j
\end{align}
whose bracket depth is $r$. For empty word $\varnothing$ write $\operatorname{ad}_{\varnothing} b_j=b_j$. We assign Brownian parabolic degree one to
each noise field $b_i$ and degree two to the drift $F$, and set the degree of a word $\mathtt{w}$ as 
\begin{align}\label{eq:brownian-parabolic-degree}
 |\mathtt{w}|_{\rm deg}
 =\sum_{\ell=1}^r\operatorname{deg}(a_\ell),
 \qquad
 \operatorname{deg}(b_i)=1,\qquad
 \operatorname{deg}(F)=2.
\end{align}
When passing to iterated integrals, $F$ corresponds to the time letter
and $b_i$ to the $i$th Brownian letter. Thus on an interval of length
$\tau$ the mixed Stratonovich integral associated with $\mathtt{w}$ scales as
$\tau^{|\mathtt{w}|_{\rm deg}/2}$. This is the classical Brownian grading in
stochastic Taylor and cubature expansions
\cite{BenArous1989,Kusuoka2001,LyonsVictoir,BayerTeichmann}.  The terminal noise field
in $\operatorname{ad}_{\mathtt{w}}b_j$ contributes one additional unit of total
Lie-word weight, which is harmless in the integer cutoffs below.

For a chosen representation
$g=\sum_{j,\mathtt{w}}a_{g,j,\mathtt{w}}\operatorname{ad}_{\mathtt{w}}b_j$, we call
$a_{g,j,\mathtt{w}}$ the bracket coefficients and measure their cost by
$\sum_{j,\mathtt{w}}|a_{g,j,\mathtt{w}}|$. Bracket depth, parabolic degree and
coefficient cost play different roles: depth records algebraic
complexity, parabolic degree matches the stochastic scaling, and the
coefficient cost measures the amplification incurred when generated
directions are assembled. We enhance Brownian motion to its geometric
$\mathsf p$-H\"older rough path $\mathbf W=(W,\mathbb W)$,
$\mathsf p\in(1/3,1/2)$, and use controlled rough path calculus
\cite{FH,GerasimovicsHairer} to justify the resulting pathwise
observable hierarchy, including for the anticipative terminal adjoint. Denote the homogeneous rough path H\"older norm as 
\begin{align*}
 \|\mathbf W\|_{\mathsf p}
 =
 [W]_{C^{\mathsf p}([0,T])}
 +
 [\mathbb W]_{C^{2\mathsf p}([0,T])}^{1/2}.
\end{align*}
For any path $X$, we use the increment convention
\begin{align*}
 X_{s,t}:=X_t-X_s,
\end{align*}
which should not be confused with two time subscripts such as
$J_{s,T}$.

\subsubsection{Analytic control of growing Lie words}
\label{subsec:growing-lie-words}
For $f_{g}(t)=\langle p_t,g(v_t)\rangle_H$, the forward--adjoint dynamics
generate Lie observables whose degree grows with the Fourier cutoff.
The next lemma controls this growth by trading differential complexity
against analytic radius, yielding bounds of order
$\exp\{Cn\log(n+2)\}$.

\begin{lemma}
\label{lem:fixed-analytic-scale}
Let $n\ge1$ be an integer and $s>3/2$. For every iterated Lie bracket
\begin{align*}
 g=\operatorname{ad}_{\mathtt{w}} b_k
 =\operatorname{ad}_{a_1}\cdots\operatorname{ad}_{a_r}b_k,
 \qquad
 a_\ell\in\{F,b_1,\ldots,b_m\},
 \qquad
 |\mathtt{w}|_{\rm deg}\le n,
\end{align*}
one has
\begin{align*}
 \sup_{\|u\|_{3\rho/4,s}\le R}
 \|g(u)\|_{\rho/4,s+2}
 \le \exp\{Cn\log(n+2)\},
\end{align*}
and for $1\le j\le3$,
\begin{align}\label{eq:word-derivative-norms}
 \sup_{\|u\|_{3\rho/4,s}\le R}
 \|D^jg(u)\|_{\mathcal L^j(
 \mathbb H_{3\rho/4,s};\mathbb H_{\rho/4,s+2})}
 \le \exp\{Cn\log(n+2)\},
\end{align}
where $C$ is independent of $n$, $k$, and the word $\mathtt{w}$.
\end{lemma}

\begin{proof}
Since $s>3/2$, for $0<r'<r$, submultiplicativity of the exponential
Fourier weight and the Sobolev algebra estimate give
\begin{align}\label{eq:analytic-radius-estimates}
 \|uv\|_{r,s}&\le C_s\|u\|_{r,s}\|v\|_{r,s},\notag\\
 \|\nabla u\|_{r',s}+\|w(u)\|_{r',s}
 &\le \frac{C}{r-r'}\|u\|_{r,s},\qquad
 \|u\|_{r',s+2}\le\frac{C}{(r-r')^2}\|u\|_{r,s}.
\end{align}
Here we use the fact that $u\mapsto w(u)$ is a Fourier multiplier of
order one by \eqref{eq:vertical-velocity}.

Since the noise coefficients $\{b_j\}$ are constant and the
nonlinearity \eqref{eq:hydrostatic-nonlinearity} is quadratic, an
induction on the parabolic weight \eqref{eq:brownian-parabolic-degree}
gives an expansion
\begin{align}\label{eq:word-monomial-expansion}
 g(u)=\sum_{\ell=1}^{N_n}c_\ell M_\ell(u),
\end{align}
where each $M_\ell$ is obtained from $u,b_1,\ldots,b_m$ by products,
hydrostatic projections, spatial derivatives, and the diagnostic
operator \eqref{eq:vertical-velocity}. Moreover,
\begin{align}\label{eq:word-complexity}
 N_n+\sum_{\ell=1}^{N_n}|c_\ell|
 \le (Cn)^{Cn},
\end{align}
while the number of factors and spatial derivative operations in each
$M_\ell$ is bounded by $Cn$.

Divide the radius gap from $3\rho/4$ to $\rho/2$ into $Cn$ equal
parts. Applying \eqref{eq:analytic-radius-estimates} successively to
\eqref{eq:word-monomial-expansion}, every spatial derivative or
diagnostic operation costs at most $C_\rho n$, whereas products
preserve the current radius. Hence
\begin{align*}
 \sup_{\|u\|_{3\rho/4,s}\le R}
 \|g(u)\|_{\rho/2,s}
 \le (C_\rho n)^{Cn}(1+R)^{Cn}.
\end{align*}
Using the fixed remaining radius gap from $\rho/2$ to $\rho/4$ in the
last estimate of \eqref{eq:analytic-radius-estimates}, we obtain
\begin{align*}
 \sup_{\|u\|_{3\rho/4,s}\le R}
 \|g(u)\|_{\rho/4,s+2}
 \le \exp\{Cn\log(n+2)\}.
\end{align*}

Finally, differentiating \eqref{eq:word-monomial-expansion} $j$ times
with respect to the state, $1\le j\le3$, and applying the Leibniz rule
produces at most $(Cn)^j$ terms from each $M_\ell$, without increasing
the spatial derivative count. The same radius allocation therefore
yields
\begin{align*}
 \sup_{\|u\|_{3\rho/4,s}\le R}
 \|D^jg(u)\|_{\mathcal L^j(
 \mathbb H_{3\rho/4,s};\mathbb H_{\rho/4,s+2})}
 \le \exp\{Cn\log(n+2)\},
\end{align*}
which completes the proof.
\end{proof}

A scalar path $X$ on $[0,T]$ is controlled by $W$ if there is
$X'\in C^{\mathsf p}([0,T];\mathbb R^m)$, called its Gubinelli derivative
\cite{Gubinelli}, such that
\begin{align*}
 X_t-X_r=X'_rW_{r,t}+R^X_{r,t},
 \qquad
 \sup_{0\le r<t\le T}\frac{|R^X_{r,t}|}{|t-r|^{2\mathsf p}}<\infty.
\end{align*}
The same definition is used for Banach-valued paths. Write the controlled path norm 
\begin{align}\label{eq:controlled-norm}
 \|X\|_{\mathscr C^{\mathsf p}([0,T])}
 =
 \|X\|_{L^\infty([0,T])}
 +\|X'\|_{C^{\mathsf p}([0,T];\mathbb R^m)}
 +\sup_{0\le r<t\le T}\frac{|R^X_{r,t}|}{|t-r|^{2\mathsf p}}.
\end{align}

Fix $s\ge6$, $\rho>0$, and $R\ge1$, and work with reference
pairs $(u,\omega)$ in the analytic class
$\mathcal C_{\rho,s}(R, T)$ of \eqref{eq:analytic-reference-core}.
Writing $v_t=\Phi_t(u,\omega)$, the defining bound is
\begin{align}\label{eq:reference-analytic-bound}
 \sup_{t\in[0,T]}\|v_t\|_{\rho,s}
 +\sum_{i=1}^m\|b_i\|_{\rho,s}
 \le R.
\end{align}
By \Cref{lem:analytic-core}, every compact analytic initial family
under a bounded driver set is contained, after decreasing $\rho$ once,
in such a class; compact positive-Wiener-mass subsets give the analytic
reference cores used below. We use this value of $\rho$ throughout the
rest of the subsection. Recall that
$f_g(t)=\langle p_t,g(v_t)\rangle_H$ for an iterated Lie bracket $g$.

\begin{lemma}
\label{lem:controlled-norms}
Fix $\mathsf p\in(1/3,1/2)$ and a reference pair
$(u,\omega)\in\mathcal C_{\rho,s}(R, T)$. Suppose that the enhanced
Brownian driver satisfies $\|\mathbf W\|_{\mathsf p}\le R_W$ on
$[0,T]$. Then, simultaneously for every terminal direction $\xi$ and
every iterated Lie bracket $g$ with word $\mathtt w$ of degree at most $n\ge1$,
\begin{align}
 df_{g}
 &=f_{[F,g]}\,dt
   +\sum_{i=1}^m f_{[b_i,g]}\circ dW^i,
 \qquad
 f_{g}'=(f_{[b_1,g]},\ldots,f_{[b_m,g]}),
 \label{eq:observable-expansion}\\
 \|f_{g}\|_{\mathscr C^{\mathsf p}([0,T])}
 &\le
 \exp\{Cn\log(n+2)\}\|\xi\|_H.
 \label{eq:observable-remainder}
\end{align}
Here $f_{g}'$ is the Gubinelli derivative with respect to $W$, and the
differential identity in \eqref{eq:observable-expansion} is understood
pathwise in the geometric rough path sense. The constant may also
depend on $T,\mathsf p$, and $R_W$, but is independent of $n$, $g$, and
$\xi$.
\end{lemma}
\begin{proof}
The radius gap from $\rho$ to $3\rho/4$ absorbs the two spatial
derivatives in the drift $F$. Hence by \eqref{eq:pe} and
\eqref{eq:reference-analytic-bound},
\begin{align}\label{eq:state-increments}
 v_t-v_r
 =\sum_{i=1}^m b_iW^i_{r,t}+R^v_{r,t},
 \qquad
 \|R^v_{r,t}\|_{3\rho/4,s}\le C|t-r|.
\end{align}
Since $2\mathsf p<1$ and $\|\mathbf W\|_{\mathsf p}\le R_W$,
\eqref{eq:state-increments} shows that $v$ is controlled by $W$ with
Gubinelli derivative $(b_1,\ldots,b_m)$ and
$\|v_t-v_r\|_{3\rho/4,s}\le C|t-r|^{\mathsf p}$. Taylor expansion, \eqref{eq:state-increments}, and
\Cref{lem:fixed-analytic-scale} then give
\begin{align*}
 g(v_t)-g(v_r)
 =\sum_{i=1}^m[b_i,g](v_r)W^i_{r,t}+R^{g(v)}_{r,t},
\end{align*}
where
\begin{align*}
 \|R^{g(v)}_{r,t}\|_{\rho/4,s+2}
 &\le \exp\{Cn\log(n+2)\}
 \bigl(|t-r|+|t-r|^{2\mathsf p}\bigr),\\
 \|[b_i,g](v_t)-[b_i,g](v_r)\|_{\rho/4,s+2}
 &\le \exp\{Cn\log(n+2)\}|t-r|^{\mathsf p}.
\end{align*}
Thus $g(v)$ is controlled by $W$ with Gubinelli derivative
$\bigl([b_1,g](v),\ldots,[b_m,g](v)\bigr)$.

For $f_{g}(t)=\langle p_t,g(v_t)\rangle_H$, define
\begin{align*}
 R^{f_{g}}_{r,t}
 =f_{g}(t)-f_{g}(r)-\sum_{i=1}^m f_{[b_i,g]}(r)W^i_{r,t}.
\end{align*}
The preceding expansion gives
\begin{align*}
 R^{f_{g}}_{r,t}
 =
 \langle p_t-p_r,g(v_t)\rangle_{H^{-2},H^2}
 +\langle p_r,R^{g(v)}_{r,t}\rangle_H.
\end{align*}
By \Cref{lem:tangent-adjoint-core}, $p\in C([0,T];H)\cap W^{1,\infty}([0,T];H^{-2})$ and 
\begin{align}\label{eq:p-supnorm}
 \|p\|_{L^\infty([0,T];H)}
 +\|\partial_t p\|_{L^\infty([0,T];H^{-2})}
 \le C\|\xi\|_H.
\end{align}
Since $g(v_t)\in\mathbb H_{\rho/4,s+2}\hookrightarrow H^2$, it follows again from \Cref{lem:fixed-analytic-scale} that
\begin{align*}
 |R^{f_{g}}_{r,t}|
 &\le
 \exp\{Cn\log(n+2)\}\|\xi\|_H
 \bigl(|t-r|+|t-r|^{2\mathsf p}\bigr),\\
 |f_{[b_i,g]}(t)-f_{[b_i,g]}(r)|
 &\le
 \exp\{Cn\log(n+2)\}\|\xi\|_H|t-r|^{\mathsf p},
\end{align*}
where in the second estimate we also use that $[b_i,g]$ has degree at
most $n+1$. Together with the corresponding $L^\infty$ bounds of $f_{g}$ and
$f_{[b_i,g]}$ which follow from \eqref{eq:p-supnorm} and
\Cref{lem:fixed-analytic-scale}, this
proves \eqref{eq:observable-remainder} and gives
\begin{align*}
 f_{g}'=\bigl(f_{[b_1,g]},\ldots,f_{[b_m,g]}\bigr).
\end{align*}

Finally, $g(v)$ is controlled with values in
$\mathbb H_{\rho/4,s+2}\hookrightarrow H^2$, while $p$ has finite
variation in $H^{-2}$. The rough chain rule and the product rule for
the pairing $H^{-2}\times H^2$ yield
\begin{align*}
 df_{g}
 &=
 \langle p,Dg(v)F(v)-DF(v)g(v)\rangle_H\,dt
 +\sum_{i=1}^m\langle p,Dg(v)b_i\rangle_H\circ dW^i\\
 &=
 f_{[F,g]}\,dt+\sum_{i=1}^m f_{[b_i,g]}\circ dW^i,
\end{align*}
which is \eqref{eq:observable-expansion}.
\end{proof}

Recall the iterated Lie bracket notation \eqref{eq:iteratedLiebracket}. 
Then \eqref{eq:observable-expansion} gives
\begin{align}\label{eq:observable-hierarchy}
 d f_{\operatorname{ad}_{\mathtt w}g}
 =f_{\operatorname{ad}_{0\mathtt w}g}\,dt
 +\sum_{i=1}^m f_{\operatorname{ad}_{i\mathtt w}g}\circ dW^i, \text{ with }
 f_{\operatorname{ad}_{\mathtt w}g}'=\bigl(f_{\operatorname{ad}_{1\mathtt w}g},\ldots,
 f_{\operatorname{ad}_{m\mathtt w}g}\bigr),
\end{align}
where $i\mathtt w$ denotes the word obtained by prepending $i\in\{0,1,\ldots,m\}$ to $\mathtt w$, and $f_{\operatorname{ad}_{\mathtt w}g}'$ denotes the Gubinelli derivative. Thus each differentiation produces the next generation of Lie observables. For $\mathtt w=(i_1,\ldots,i_k)$, denote
\begin{align*}
 S_{\mathtt w}(s,t)
 =\int_{s<r_1<\cdots<r_k<t}1\,
 \circ dZ_{r_1}^{i_1}\cdots\circ dZ_{r_k}^{i_k},
 \qquad S_{\varnothing}(s,t)=1,
\end{align*}
where $Z_t^0=t$ and $Z_t^i=W_t^i$ for $1\le i\le m$. Thus $0$ denotes the time letter of parabolic degree two, while $1,\ldots,m$ denote Brownian letters of degree one, consistently with \eqref{eq:brownian-parabolic-degree}. With this convention, iterating \eqref{eq:observable-hierarchy} pairs $f_{\operatorname{ad}_{\mathtt w}g}(s)$ with $S_{\mathtt w}(s,t)$.

The following lemma gives a quantitative estimate for the remainder in the stochastic Taylor expansion of $f_g$.

\begin{lemma}
\label{lem:weighted-taylor}
Under the assumptions of \Cref{lem:controlled-norms}, let $g$ be an iterated Lie bracket of parabolic degree at most $n$, $n\ge1$. Then, for every interval $I=[s,s+\tau]\subset[0,T]$ with $0<\tau\le1$,
\begin{align}
 f_g(t)
 &=\sum_{|\mathtt w|_{\rm deg}\le n}
 f_{\operatorname{ad}_{\mathtt w}g}(s)S_{\mathtt w}(s,t)
 +\mathcal R_{g,n}(s,t),\notag\\
 \sup_{s\le t\le s+\tau}|\mathcal R_{g,n}(s,t)|
 &\le \exp\{Cn\log(n+2)\}\tau^{\mathsf p(n+1)}\|\xi\|_H.
 \label{eq:weighted-taylor}
\end{align}
Here $C=C(m,T,\mathsf p,R_W)$ is independent of $n$, $g$, and $\xi$.
\end{lemma}

\begin{proof}
For a scalar controlled path $(f,f')$ on $I$, set the scaled controlled norm 
\begin{align*}
 N_I(f,f')
 =\|f\|_{L^\infty(I)}
 +\tau^{\mathsf p}\|f'\|_{L^\infty(I;\mathbb R^m)}
 +\tau^{2\mathsf p}[f']_{C^{\mathsf p}(I;\mathbb R^m)}
 +\tau^{2\mathsf p}\sup_{r<t}\frac{|R^f_{r,t}|}{|t-r|^{2\mathsf p}}.
\end{align*}
The controlled-path identity $f_t-f_r=f'_rW_{r,t}+R^f_{r,t}$ and $\|\mathbf W\|_{\mathsf p}\le R_W$ imply
\begin{align}\label{eq:scaled-controlled-holder}
 \tau^{\mathsf p}[f]_{C^{\mathsf p}(I)}
 \le C_{R_W}N_I(f,f').
\end{align}
For $Y_t=\int_s^t f_r\circ dW_r^i$ with Gubinelli derivative $Y'_t=f_te_i$, the rough integral expansion \cite{FH} gives
\begin{align*}
 Y_{r,t}
 =f_rW_{r,t}^i+\sum_{j=1}^m(f'_r)^j\mathbb W_{r,t}^{j,i}+E_{r,t},
\end{align*}
with
\begin{align}\label{eq:rough-integration-remainder}
 |E_{r,t}|
 \le C_{\mathsf p,R_W}
 \left([f']_{C^{\mathsf p}(I;\mathbb R^m)}
 +\sup_{a<b}\frac{|R^f_{a,b}|}{|b-a|^{2\mathsf p}}\right)|t-r|^{3\mathsf p}.
\end{align}
Since
\begin{align*}
 R^Y_{r,t}
 =Y_{r,t}-Y'_rW_{r,t}
 =\sum_{j=1}^m(f'_r)^j\mathbb W_{r,t}^{j,i}+E_{r,t},
\end{align*}
\eqref{eq:scaled-controlled-holder} yields
\begin{align}\label{eq:brownian-integration-gain}
 N_I(Y,Y')
 \le C_{\mathsf p,R_W}\tau^{\mathsf p}N_I(f,f').
\end{align}
For $Y_t=\int_s^t f_r\,dr$, one has $Y'=0$ and
\begin{align}\label{eq:time-integration-gain}
 N_I(Y,0)
 \le C\tau N_I(f,f')
 \le C\tau^{2\mathsf p}N_I(f,f').
\end{align}
Hence, if
\begin{align*}
 \mathcal S_{\mathtt w}f(t)
 =\int_{s<r_1<\cdots<r_\ell<t}
 f_{r_1}\,\circ dZ_{r_1}^{i_1}\cdots\circ dZ_{r_\ell}^{i_\ell}
\end{align*}
and $\mathtt w=(i_1,\ldots,i_\ell)$ has parabolic degree $k$, then iterating \eqref{eq:brownian-integration-gain} and \eqref{eq:time-integration-gain} gives 
\begin{align}\label{eq:nested-integral-gain}
 N_I\bigl(\mathcal S_{\mathtt w}f,(\mathcal S_{\mathtt w}f)'\bigr)
 \le C_{m,\mathsf p,R_W}^k\tau^{\mathsf pk}N_I(f,f').
\end{align}

Define the boundary set 
\begin{align*}
 \partial\mathcal W_n
 =\bigl\{\mathtt w=i\mathtt v:
 |\mathtt v|_{\rm deg}\le n<|\mathtt w|_{\rm deg}\bigr\}.
\end{align*}
Repeated integration of \eqref{eq:observable-hierarchy} gives
\begin{align*}
 f_g(t)
 =
 \sum_{|\mathtt w|_{\rm deg}\le n}
 f_{\operatorname{ad}_{\mathtt w}g}(s)S_{\mathtt w}(s,t)+ 
 \sum_{\substack{\mathtt w=(i_1,\ldots,i_\ell)\in\partial\mathcal W_n}}
 \int_{s<r_1<\cdots<r_\ell<t}
 f_{\operatorname{ad}_{\mathtt w}g}(r_1)\,
 \circ dZ_{r_1}^{i_1}\cdots\circ dZ_{r_\ell}^{i_\ell}.
\end{align*}
The second sum is $\mathcal R_{g,n}(s,t)$. Since every letter has parabolic degree one or two,
\begin{align*}
 n+1\le|\mathtt w|_{\rm deg}\le n+2,
 \qquad \mathtt w\in\partial\mathcal W_n,
\end{align*}
and $\#\partial\mathcal W_n\le(m+2)^{n+2}$. Since $g$ has degree at most $n$, every $\operatorname{ad}_{\mathtt w}g$ in the boundary remainder has degree at most $2n+2$. Thus \Cref{lem:controlled-norms}, applied at degree $2n+2$, gives
\begin{align*}
 \|f_{\operatorname{ad}_{\mathtt w}g}\|_{\mathscr C^{\mathsf p}([0,T])}
 \le \exp\{Cn\log(n+2)\}\|\xi\|_H.
\end{align*}
Combining this with \eqref{eq:nested-integral-gain} and absorbing
$C_{m,\mathsf p,R_W}^{n+2}$ and $\#\partial\mathcal W_n$ into the exponential constant gives \eqref{eq:weighted-taylor}.
\end{proof}

\subsubsection{Lie-coefficient extraction and low-mode observability}
By \eqref{eq:observable-expansion} and \Cref{lem:weighted-taylor}, on
every interval $[s,s+\tau]$ one has
\begin{align*}
 f_{b_j}(t)
 =
 \sum_{|\mathtt w|_{\rm deg}\le n}
 f_{\operatorname{ad}_{\mathtt w}b_j}(s)
 S_{\mathtt w}(s,t)
 +\mathcal R_{b_j}(s,t),
\end{align*}
where the remainder is controlled by \eqref{eq:weighted-taylor}.
The simultaneous frame of \Cref{prop:interval-frames} then gives
quantitative control of the Lie coefficients
$f_{\operatorname{ad}_{\mathtt w}b_j}(s)$ from the observed paths
$f_{b_j}$. To convert these coefficients into a bound for $P_Np_s$,
we need to represent every Fourier basis vector $g\in P_NH$ in the form
\begin{align*}
 g
 =
 \sum_{j,\mathtt w}
 a_{g,j,\mathtt w}\operatorname{ad}_{\mathtt w}b_j
\end{align*}
with word degree of order $N$ and exponentially controlled coefficient
cost. The required quantitative Lie bracket representation of Fourier modes is given in
\Cref{thm:fourier-lie-generation}; using it together with the temporal
frame yields the following low mode observation estimate. Keep the parameters $s\ge6$, $\rho>0$, $R\ge1$ from \eqref{eq:reference-analytic-bound}. 

\begin{proposition}
\label{prop:source-occupation} Assume the forcing hypotheses of \Cref{thm:main} and fix a nontrivial interval
$I=[0,T_0]\subset[0,T]$ and $R_W<\infty$.
For every $B>0$ and $\varepsilon>0$, there exist
$C_B<\infty$, $N_0\ge1$, and a Brownian event
$\Omega_{\rm obs}$ with
\begin{align*}
 \mathbb P(\Omega_{\rm obs})\ge1-\varepsilon
\end{align*}
such that the following holds. For every $N\ge N_0$, $\xi\in H$, and every reference pair
$(u,\omega)\in\mathcal C_{\rho,s}(R,T)$ satisfying
$\|\mathbf W(\omega)\|_{\mathsf p}\le R_W$ and
$\omega\in\Omega_{\rm obs}$,
\begin{align}\label{eq:low-observation}
 \|P_Np\|_{L^2(I;H)}
 \le
 e^{C_BN^2}\|Q^*p\|_{L^2(I;\mathbb R^m)}
 +e^{-BN^2}\|\xi\|_H.
\end{align}
The constants are independent of the reference pair and the terminal
direction $\xi$.
\end{proposition}

\begin{proof}
By \Cref{thm:fourier-lie-generation}, every unit Fourier basis vector
$g\in P_NH$ admits
\begin{align}\label{eq:fourier-lie-input}
 g
 =
 \sum_{j=1}^m
 \sum_{|\mathtt w|_{\rm deg}\le L_N}
 a_{g,j,\mathtt w}\operatorname{ad}_{\mathtt w}b_j, \text{ with }
 N\le L_N\le c_0N \text{ and }
 \sum_{j,\mathtt w}|a_{g,j,\mathtt w}|
 \le e^{CL_N}.
\end{align}
Set $L=L_N$. By
\Cref{lem:fixed-analytic-scale,lem:controlled-norms}, all Lie
observables of degree $O(L)$ occurring below have controlled norms
bounded by
\begin{align}\label{eq:observation-controlled-cost}
 \exp\{CL\log(L+2)\}\|\xi\|_H,
\end{align}
where the factor $e^{CL}$ in \eqref{eq:fourier-lie-input} is absorbed
into the same bound. Let $a>0$ to be fixed below and apply \Cref{prop:interval-frames} with $n_0\ge c_0$ for $c_0$ in \eqref{eq:fourier-lie-input}, and with $I$ partitioned into intervals of equal length $\tau=\tau_N\asymp e^{-aL}.$ 
On one interval
$[s,s+\tau]$, \Cref{lem:weighted-taylor}, applied to $f_{b_j}$ through
degree $4L$, gives
\begin{align*}
 f_{b_j}(t)
 =
 \sum_{|\mathtt w|_{\rm deg}\le4L}
 f_{\operatorname{ad}_{\mathtt w}b_j}(s)
 S_{\mathtt w}(s,t)
 +\mathcal R_{b_j}(s,t),
\end{align*}
with
\begin{align}\label{eq:first-taylor-remainder}
 \|\mathcal R_{b_j}\|_{L^\infty([s,s+\tau])}
 \le
 e^{CL\log(L+2)}
 \tau^{\mathsf p(4L+1)}\|\xi\|_H.
\end{align}
Under Brownian rescaling through $t=s+\tau r$, 
\begin{align*}
 f_{b_j}(s+\tau r) - \mathcal R_{b_j}(s,s+\tau r)
=\sum_{|\mathtt w|_{\rm deg}\le4L}
 f_{\operatorname{ad}_{\mathtt w}b_j}(s)\tau^{\frac{|\mathtt w|_{\rm deg}}{2}}
 S_{\mathtt w}^{[s,s+\tau]}(0,r)
\end{align*}
where $S_{\mathtt w}^{[s,s+\tau]}$ denotes the rescaled Stratonovich
signature of the Brownian increment on $[s,s+\tau]$. Therefore, combining this with \eqref{eq:first-taylor-remainder}, \eqref{eq:cluster-frame} and $Q^* p=\left(f_{b_1}, \ldots, f_{b_m}\right)$ yields 
\begin{align}\label{eq:coefficient-extraction}
\begin{split}
 \mathscr F(s):=&\left(
 \sum_{j=1}^m
 \sum_{|\mathtt w|_{\rm deg}\le2L}
 |f_{\operatorname{ad}_{\mathtt w}b_j}(s)|^2
 \right)^{1/2}
 \le \left(
 \sum_{j=1}^m
 \tau^{-2L}\sum_{|\mathtt w|_{\rm deg}\le4L}
 |f_{\operatorname{ad}_{\mathtt w}b_j}(s)|^2\tau^{|\mathtt w|_{\rm deg}}
 \right)^{1/2}\\
 &\qquad \le \tau^{-L}\left(
 \sum_{j=1}^m
 e^{C_* L^2+C_a L}\left\|f_{b_j}(s+\tau \cdot)-\mathcal{R}_{b_j}(s, s+\tau \cdot)\right\|_{L^2(0,1)}^2
 \right)^{1/2}\\
 &\qquad \le
 e^{C_1L^2+C_{2,a}L}
 \tau^{-L-1/2}
 \|Q^*p\|_{L^2([s,s+\tau];\mathbb R^m)}+
 e^{C_1L^2+C_{2,a}L}
 \tau^{\mathsf p(4L+1)-L}\|\xi\|_H,
\end{split}
\end{align}
where $C_1$ is independent of $a$ and we used the fact that $C$ in \eqref{eq:observation-controlled-cost} does not depend on $a$. 

Now fix a unit Fourier vector $g\in P_NH$. Applying
\Cref{lem:weighted-taylor} through degree $L$ to each iterated Lie bracket
appearing in the representation \eqref{eq:fourier-lie-input}, and
combining the resulting expansions linearly yields 
\begin{align*}
f_g(t)=\sum_{|\mathtt u|_{\operatorname{deg}} \leq L} c_{g, \mathtt u}(s) S_{\mathtt u}(s, t)+\mathcal{R}_g(s, t), \text{ with } c_{g, \mathtt u}(s)=\sum_{j, |\mathtt v|_{\rm deg}\leq L} a_{g, j, \mathtt v} f_{\operatorname{ad}_{\mathtt u} \operatorname{ad}_{\mathtt v} b_j}(s).
\end{align*}
Since $\mathsf p>1/3$, it follows from \eqref{eq:weighted-taylor}, \eqref{eq:nested-integral-gain}, \eqref{eq:fourier-lie-input} and \eqref{eq:coefficient-extraction} that 
\begin{align}\notag
\sup _{s \leq t \leq s+\tau}\left|f_g(t)\right|&\leq \sum_{|\mathtt u|_{\operatorname{deg}} \leq L}\sum_{j, |\mathtt v|_{\rm deg}\leq L}|a_{g, j, \mathtt v}|C_{R_W}^L\mathscr F(s)+ \sum_{j, |\mathtt v|_{\rm deg}\leq L}|a_{g, j, \mathtt v}| e^{C L \log (L+2)} \tau^{\mathsf p(L+1)}\|\xi\|_H\\
&\leq (m+2)^{L+1}C_{R_W}^L e^{C L} \mathscr F(s) + e^{C L \log (L+2)} \tau^{\mathsf p(L+1)}\|\xi\|_H. \label{eq:supfgt}
\end{align}
Let $\left\{g_1, \ldots, g_{d_N}\right\}$ be an ONB of $P_NH$, where $d_N=O\left(N^3\right)$. As $\log\tau^{-1}=aL+O(1)$, $N\le L\le c_0N$ and $C_1$ is independent of $a$, \eqref{eq:supfgt} and \eqref{eq:coefficient-extraction} give 
\begin{align*}
\begin{split}
 \sup_{s\le t\le s+\tau}\|P_Np_t\|_H&=\sup_{s\le t\le s+\tau}\left(\sum_{k=1}^{d_N} \left|f_{g_k}(t)\right|^2\right)^{1/2}\\
 &\le
 e^{C_1L^2+C_{2,a}L}
 \left(\tau^{-L-1/2}
 \|Q^*p\|_{L^2([s,s+\tau];\mathbb R^m)} +
 \tau^{\mathsf p(L+1)}\|\xi\|_H\right)\\
 &\leq  e^{C_BN^2}
 \|Q^*p\|_{L^2([s,s+\tau];\mathbb R^m)}
 +
 \frac{e^{-2BN^2}}{1+T_0}\|\xi\|_H 
\end{split}
\end{align*}
by first taking $a$ sufficiently large such that $a \mathsf p>C_1+2 B+2$ and then choosing $N_0$ large.  Finally, square the preceding estimate, multiply by $\tau$,
and sum over the disjoint intervals of the partition. The observation
terms sum to $\|Q^*p\|_{L^2(I;\mathbb R^m)}^2$, while the remainder
contributes at most a fixed multiple of
$e^{-4BN^2}\|\xi\|_H^2$. After increasing $C_B$ and $N_0$ and taking
square roots, we obtain \eqref{eq:low-observation}.
\end{proof}

\subsubsection{Proof of \Cref{prop:polynomial-source}}

We now select one compact Brownian core on which the low-mode observation
holds simultaneously at every sufficiently large resolution and complete
the proof of \Cref{prop:polynomial-source}.

\begin{proof}
Recall that  $\mathsf p\in(1/3,1/2)$. 
Choose $M,R_W<\infty$ so that
\begin{align}\label{eq:source-reference-event}
 \Omega_{\rm ref}
 =
 \left\{
 \|\omega\|_{L^\infty([0,T];\mathbb R^m)}\le M,\ 
 \|\mathbf W\|_{\mathsf p}\le R_W
 \right\}
\end{align}
has Wiener measure $p_0>0$. Since $\sup_{u\in\mathscr D}\|u\|_{\rho_{\mathscr D},s}\le R_{\mathscr D}$ from the assumption, 
\Cref{lem:analytic-core} applied with the driver bound $M$ gives $\rho>0$ and $R<\infty$ such that
\begin{align}\label{eq:uniform-source-analytic-core}
 \mathscr D\times
 \left\{\omega:\|\omega\|_{L^\infty([0,T];\mathbb R^m)}\le M\right\}
 \subset\mathcal C_{\rho,s}(R,T).
\end{align}
Hence every reference pair in $\mathscr D\times\Omega_{\rm ref}$ satisfies
\eqref{eq:reference-analytic-bound}. 
Set $T_0=T/2$,
$\varepsilon\in(0,p_0/2)$ and  $B>0$. By
\Cref{prop:source-occupation}, there are $N_0\ge1$, $C_B<\infty$, and
a Brownian event $\Omega_{\rm obs}$ with
\begin{align*}
 \mathbb P(\Omega_{\rm obs})\ge1-\varepsilon
\end{align*}
such that, simultaneously for every $N\ge N_0$, $u\in\mathscr D$, and
$\xi\in H$,
\begin{align}\label{eq:source-core-occupation}
 \|P_Np\|_{L^2([0,T/2];H)}
 \le
 e^{C_BN^2}\|Q^*p\|_{L^2([0,T/2];\mathbb R^m)}
 +e^{-BN^2}\|\xi\|_H
\end{align}
whenever $\omega\in\Omega_{\rm ref}\cap\Omega_{\rm obs}$. Note $\mathbb P(\Omega_{\rm ref}\cap\Omega_{\rm obs})
\ge p_0-\varepsilon>0$ as $\varepsilon<p_0/2$. 
By inner regularity of the Wiener measure, 
there is a compact set $\Gamma\subset\Omega_{\rm ref}\cap\Omega_{\rm obs}$ with  $\mathbb P(\Gamma)>0$, and   
$\mathscr D\times\Gamma\subset\mathcal C_{\rho,s}(R,T)$ by \eqref{eq:uniform-source-analytic-core}. By compactness of $\mathscr D\times\Gamma$ and
\Cref{lem:driver-chart}, after decreasing $\rho>0$ if necessary,
the driver and mixed-derivative bounds required in
\Cref{prop:compact-source-repair} hold uniformly on $\mathscr D\times\mathscr N_{\rho}(\Gamma)$. 
Together with \Cref{lem:endpoint-remainder}, this verifies the
regular-dependence hypotheses of that proposition.

Fix $(u,\omega)\in \mathscr D\times\Gamma, \xi\in H$, and recall 
$p_t=J_{t,T}^*\xi$. On $\mathscr D\times\Gamma$,
\Cref{lem:tangent-adjoint-core} gives
\begin{align}\label{eq:source-core-adjoint}
 \frac{d}{dt}\|p_t\|_H^2
 \ge
 \nu\|\nabla p_t\|_2^2-C_0\|p_t\|_H^2,
 \qquad
 \sup_{t\le T}\|p_t\|_H\le C_0\|\xi\|_H.
\end{align}
For the Fourier projector $P_N$, the spectral inequality gives $\|\nabla p_t\|_2^2
\ge c_0N^2\bigl(\|p_t\|_H^2-\|P_Np_t\|_H^2\bigr)$ for some embedding constant $c_0>0$. 
Thus, 
\begin{align*}
 \frac{d}{dt}\|p_t\|_H^2\ge\lambda_N \|p_t\|_H^2-c\nu N^2\|P_Np_t\|_H^2, \quad \lambda_N=c\nu N^2-C_0>0,
\end{align*}
for large $N$.  Multiplying by $e^{-\lambda_Nt}$ and integrating over $[0,T/2]$ yields
\begin{align}\label{eq:source-energy-transfer}
 \|J_{0,T}^*\xi\|_H^2=\|p_0\|_H^2\le
 e^{-\lambda_NT/2}\|p_{T/2}\|_H^2
 +c\nu N^2\int_0^{T/2}e^{-\lambda_Nt}\|P_Np_t\|_H^2\,dt.
\end{align}
Since $(\mathcal A_T^*\xi)(t)=Q^*p_t$ for $0\le t\le T$ by \eqref{eq:control-operator}, 
squaring \eqref{eq:source-core-occupation} gives
\begin{align*}
 \int_0^{T/2}\|P_Np_t\|_H^2\,dt
 \le
 2e^{2C_BN^2}\|\mathcal A_T^*\xi\|_{L^2([0,T];\mathbb R^m)}^2
 +2e^{-2BN^2}\|\xi\|_H^2.
\end{align*}
Combining this with \eqref{eq:source-energy-transfer} and
\eqref{eq:source-core-adjoint} yields $c_1,C_1>0$, uniformly on
$\mathscr D\times\Gamma$, such that
\begin{align*}
 \|J_{0,T}^*\xi\|_H^2
 \le
 e^{C_1N^2}\|\mathcal A_T^*\xi\|_{L^2([0,T];\mathbb R^m)}^2
 +C_1e^{-c_1N^2}\|\xi\|_H^2,
 \qquad N\ge N_0.
\end{align*}
Since this holds for every $\xi\in H$,
\begin{align}\label{eq:source-core-operator}
 J_{0,T}J_{0,T}^*
 \le
 e^{C_1N^2}\mathcal A_T\mathcal A_T^*
 +C_1e^{-c_1N^2}I.
\end{align}

For sufficiently small $\delta>0$, choose the smallest $N\ge N_0$ such
that $C_1e^{-c_1N^2}\le\delta^2$. 
By minimality of $N$, 
$
 N^2\le C(1+|\log\delta|),
$
and \eqref{eq:source-core-operator} gives
\begin{align*}
 J_{0,T}J_{0,T}^*
 \le C_\delta^2\mathcal A_T\mathcal A_T^*+\delta^2I,
 \qquad
 C_\delta=e^{C_1N^2/2}\le C\delta^{-b}
\end{align*}
for some fixed $C,b<\infty$. Enlarging $C$ to cover the remaining
$\delta$ in a fixed compact subinterval of $(0,1)$ completes the proof.
\end{proof}

\subsection{Quantitative Fourier generation and sharp rank counts}\label{sec:Quantitative-Fourier-generation}
In this section, we prove that each hydrostatic Fourier mode admits a Lie-polynomial representation in a quantitative way, as required in \Cref{prop:source-occupation}. We also provide an equivalent characterization of quantitative saturation as in \Cref{def:fourier-certificate} and minimal rank counts as mentioned in \Cref{rem:low-rank-example}.

\subsubsection{Quantitative Fourier--Lie generation}
\label{subsec:fourier-lie-generation}
The Malliavin argument in \Cref{sec:quantitative-Malliavin-observability} uses the following  Fourier resolution-dependent
consequence of the forcing geometry, which we shall prove in this subsection. 

\begin{theorem}\label{thm:fourier-lie-generation}
Under the forcing hypotheses of \Cref{thm:main}, every unit real
hydrostatic Fourier basis vector $g\in P_NH$ can be represented as
\begin{align}\label{eq:fourier-lie-generation}
 g = \sum_{j=1}^m \sum_{|\mathtt w|_{\rm deg}\le L_N} a_{g,j,\mathtt w}\operatorname{ad}_{\mathtt w}b_j, \qquad N\le L_N\le c_0N, \qquad \sum_{j,\mathtt w}|a_{g,j,\mathtt w}| \le e^{CL_N}.
\end{align}
The constants $c_0,C$ depend on the fixed forcing family, but not on
$N$ or $g$.
\end{theorem}
The underlying quantitative generation statement is more naturally
formulated in terms of the individual Fourier frequency
$(\boldsymbol k,n)$ than in terms of the auxiliary resolution cutoff
$N$. Recall the hydrostatic frequency set
\begin{align*}
 \mathbf F = \{(\boldsymbol k,n)\in\mathbb Z^2\times\mathbb N_0: (\boldsymbol k,n)\ne(0,0)\}
\end{align*}
from \eqref{eq:hydrostatic-frequencies}, and the quadratic extensions
$E_j$ from \eqref{eq:quadratic-saturation}. 
Denote the complexified full extension as 
\begin{align}\label{eq:complexified-full-extension}
\mathcal E_\infty = \bigcup_{j\ge0}E_j\otimes_{\mathbb R}\mathbb C. 
\end{align}
In the Fourier calculations below, real lattice vectors are written in
bold and complex polarization vectors in sans serif. For
$\boldsymbol k\in\mathbb Z^2$, $n\in\mathbb N_0$ and
$\mathsf a\in\mathbb C^2$, write
\begin{align}\label{eq:fourier-modes}
 \phi_{\boldsymbol k,n,\mathsf a}(\boldsymbol x,z) = e^{\mathrm{i}\boldsymbol k\cdot\boldsymbol x}\cos(nz)\mathsf a.
\end{align}
For $(\boldsymbol k,n)\in\mathbf F$, define
\begin{align}\label{eq:hydrostatic-fibre}
 \mathfrak P_{(\boldsymbol k,n)} = \begin{cases} \mathbb C^2,&n\ge1,\\
 \{\mathsf a\in\mathbb C^2:\boldsymbol k\cdot\mathsf a=0\},&n=0. \end{cases}
\end{align}
\Cref{thm:fourier-lie-generation} is then a direct consequence of the following proposition. 

\begin{proposition}\label{prop:geometric-generation}
Let $E_0=\mathrm{Ran} \, Q$. The
following conditions are equivalent:
\begin{enumerate}[label=\textup{(\roman*)},ref=\roman*,leftmargin=*]
\item $E_0$ satisfies quantitative quadratic saturation in the sense of
\Cref{def:fourier-certificate};
\item every hydrostatic Fourier mode is generated algebraically:
\begin{align*}
 \phi_{\boldsymbol k,n,\mathsf a}\in\mathcal E_\infty \qquad \bigl((\boldsymbol k,n)\in\mathbf F,\ \mathsf a\in\mathfrak P_{(\boldsymbol k,n)}\bigr).
\end{align*}
\end{enumerate}
Whenever these conditions hold, for every
$(\boldsymbol k,n)\in\mathbf F$ and every
$\mathsf a\in\mathfrak P_{(\boldsymbol k,n)}$ with $|\mathsf a|=1$,
there exist finitely many coefficients $a_{j,\mathtt w}\in\mathbb C$
such that
\begin{align}\label{eq:geometric-quantitative-bounds}
 \phi_{\boldsymbol k,n,\mathsf a}
 &=
 \sum_{j=1}^m\sum_{\mathtt w}
 a_{j,\mathtt w}\operatorname{ad}_{\mathtt w}b_j,\notag\\
 \max_{a_{j,\mathtt w}\ne0}
 \bigl(|\mathtt w|_{\rm deg}+1\bigr)
 &\le C(1+|\boldsymbol k|+n),\qquad
 \sum_{j,\mathtt w}|a_{j,\mathtt w}|
 \le \exp\{C(1+|\boldsymbol k|+n)\}.
\end{align}
\end{proposition}

\begin{proof}[Proof of \Cref{thm:fourier-lie-generation}]
By \Cref{prop:geometric-generation}, every unit hydrostatic Fourier mode
at frequency $(\boldsymbol k,n)$ has a right-nested Lie-word
representation as in \eqref{eq:geometric-quantitative-bounds}.
For $g\in P_NH$, the defining Fourier cutoff gives
$1+|\boldsymbol k|+n\le C_0N$. Hence, after choosing a sufficiently
large integer $C_1\ge1$, all words occurring in these representations
satisfy $|\mathtt w|_{\rm deg}\le C_1N=:L_N$. 
Extending the coefficients by zero up to degree $L_N$ leaves both the
representation and its coefficient sum unchanged. Thus
$N\le L_N\le c_0N$ and, after increasing $C$ if necessary,
\begin{align*}
 \sum_{j,\mathtt w}|a_{g,j,\mathtt w}|\le e^{CL_N}
\end{align*}
uniformly over all unit Fourier basis vectors $g\in P_NH$.
Taking real and imaginary parts changes only fixed normalization
constants, which gives \eqref{eq:fourier-lie-generation}.
\end{proof}

The rest of this subsection proves \Cref{prop:geometric-generation}.
Assume first that quantitative quadratic saturation holds and fix $j_*$
as in \Cref{def:fourier-certificate}. Set
\begin{align}\label{eq:mathcal-S}
 \mathcal S
 =\{\boldsymbol k:E_{j_*}^{\mathbb C}\cap
 \mathcal H_{\boldsymbol k,0}\ne\{0\}\}.
\end{align}
Then $\mathcal S$ is finite and symmetric and
$\operatorname{span}_{\mathbb Z}\mathcal S=\mathbb Z^2$. Choose positive
levels $n_1,\ldots,n_{\ell_*}$ from the second set in
\eqref{eq:finite-fourier-certificate} such that
$\gcd(n_1,\ldots,n_{\ell_*})=1$. For each $n_r$, fix
$\boldsymbol k_r$ and $0\ne\mathsf a_r\in\mathbb C^2$ such that
\begin{align}\label{eq:select-kna}
\phi_{\boldsymbol k_r,n_r,\mathsf a_r}\in E_{j_*}^{\mathbb C}.
\end{align}
Since every nonzero barotropic fibre is one-dimensional, after
rescaling we may also fix $\phi_{\boldsymbol s,0,\boldsymbol s^\perp}\in E_{j_*}^{\mathbb C}$ for $\boldsymbol s\in\mathcal S$. 

The quantitative saturation condition supplies barotropic frequencies
generating the horizontal lattice and finitely many positive vertical
levels with greatest common divisor one.  At each such positive level,
however, only one nonzero polarization may initially be available,
whereas the corresponding fibre is two-dimensional over $\mathbb C$ in \eqref{eq:hydrostatic-fibre}.  We shall complete these polarization
fibres by horizontal transport, then use quadratic interactions to
generate all vertical levels, and finally quantify the resulting
generation procedure.

We first keep the vertical level fixed and use the available barotropic
modes to transport a polarization through the horizontal lattice. For a
nonzero barotropic frequency $\boldsymbol s$, write
$\boldsymbol s^\perp=(-s_2,s_1)$. Direct substitution in the quadratic interaction \eqref{eq:symmetric-quadratic-interaction} and \eqref{eq:hydrostatic-nonlinearity} gives, for
$n\ge1$, 
\begin{align}\label{eq:barotropic-gcd-shift}
 \mathsf B(\phi_{\boldsymbol k,n,\mathsf a},
 \phi_{\boldsymbol s,0,\boldsymbol s^\perp})
 =\mathrm{i}\phi_{\boldsymbol k+\boldsymbol s,n,
 M_{\boldsymbol s}(\boldsymbol k)\mathsf a}, \quad \text{ where }
 M_{\boldsymbol s}(\boldsymbol k)
 =(\boldsymbol s\wedge\boldsymbol k)I
 +\boldsymbol s^\perp\otimes\boldsymbol s.
\end{align}
Thus the interaction implements the horizontal move
$(\boldsymbol k,n)\mapsto(\boldsymbol k+\boldsymbol s,n)$ while
transporting the polarization by $M_{\boldsymbol s}(\boldsymbol k)$.
Moreover, $\det M_{\boldsymbol s}(\boldsymbol k)
=(\boldsymbol s\wedge\boldsymbol k)^2$ and 
\begin{align}\label{eq:barotropic-transport-inverse}
 M_{-\boldsymbol s}(\boldsymbol k+\boldsymbol s)
 M_{\boldsymbol s}(\boldsymbol k)
 =- (\boldsymbol s\wedge\boldsymbol k)^2I.
\end{align}
Hence this move is reversible whenever
$\boldsymbol s\wedge\boldsymbol k\ne0$. The next lemma shows that this
restriction does not obstruct horizontal generation: every nonzero
frequency remains reachable from a fixed nonzero vertex, with path
length and intermediate frequencies growing at most linearly in the
target frequency.

\begin{lemma}\label{lem:barotropic-paths}
On $\mathbb Z^2\setminus\{0\}$ join $\boldsymbol k$ to
$\boldsymbol k+\boldsymbol s$ whenever $\boldsymbol s\in\mathcal S$ and
$\boldsymbol s\wedge\boldsymbol k\ne0$. Then the resulting graph $\mathscr G$  is connected. For each
fixed starting vertex $\boldsymbol h\ne0$, the path to any
$\boldsymbol k\ne0$ can be chosen with at most
$C_{\mathcal S,\boldsymbol h}(1+|\boldsymbol k|)$ edges and with all its
vertices of magnitude at most
$C_{\mathcal S,\boldsymbol h}(1+|\boldsymbol k|)$.
\end{lemma}

\begin{proof}
Since $\operatorname{span}_{\mathbb Z}\mathcal S=\mathbb Z^2$, fix
finite sequences of elements of $\mathcal S$ whose sums are
$\boldsymbol e_1$ and $\boldsymbol e_2$, together with the reversed
sequences giving $-\boldsymbol e_1$ and $-\boldsymbol e_2$. Repeating
these fixed sequences according to the coordinates of
$\boldsymbol k-\boldsymbol h$ gives a path from $\boldsymbol h$ to
$\boldsymbol k$ whose steps belong to $\mathcal S$, with
$O_{\mathcal S,\boldsymbol h}(1+|\boldsymbol k|)$ edges and with all
vertices of magnitude $O_{\mathcal S,\boldsymbol h}(1+|\boldsymbol k|)$.
This path need not belong to $\mathscr G$; we make it admissible by
bounded local modifications.

First remove every visit to the origin. Such a visit contains a segment $\boldsymbol a\to0\to\boldsymbol b$ where $\boldsymbol a,\boldsymbol b\in\mathcal S$ by symmetry of $\mathcal S$. If
$\boldsymbol a\wedge\boldsymbol b\ne0$, replace this segment by $\boldsymbol a\to \boldsymbol a+\boldsymbol b\to \boldsymbol b.$
Both new edges belong to $\mathscr G$. If
$\boldsymbol a\wedge\boldsymbol b=0$, choose
$\boldsymbol t\in\mathcal S$ not parallel to $\boldsymbol a$; such a
choice exists because $\mathcal S$ spans $\mathbb Z^2$. Since
$\boldsymbol a$ and $\boldsymbol b$ are parallel, the same
$\boldsymbol t$ is not parallel to $\boldsymbol b$. Replace the segment
by
\begin{align*}
 \boldsymbol a\to \boldsymbol a+\boldsymbol t\to \boldsymbol t\to\boldsymbol t+\boldsymbol b\to \boldsymbol b.
\end{align*}
All four edges belong to $\mathscr G$, and none of the new vertices is
zero. Repeating this operation removes every visit to the origin.

It remains to remove inadmissible edges. Consider an edge $\boldsymbol k\to\boldsymbol k+\boldsymbol s$
with both endpoints nonzero and
$\boldsymbol k\wedge\boldsymbol s=0$. Choose
$\boldsymbol t\in\mathcal S$ not parallel to $\boldsymbol k$ and replace
the edge by
\begin{align*}
 \boldsymbol k\to  \boldsymbol k+\boldsymbol t\to 
 \boldsymbol k+\boldsymbol s+\boldsymbol t\to 
 \boldsymbol k+\boldsymbol s.
\end{align*}
The three steps are admissible and the two new
intermediate vertices are nonzero.

Thus every forbidden portion of the original path is replaced by a
uniformly bounded number of edges of $\mathscr G$. Since $\mathcal S$
is finite, each modification changes the intermediate vertices only by
a uniformly bounded amount. Consequently the resulting path in
$\mathscr G$ still has at most
$C_{\mathcal S,\boldsymbol h}(1+|\boldsymbol k|)$ edges, and all its
vertices have magnitude at most
$C_{\mathcal S,\boldsymbol h}(1+|\boldsymbol k|)$. In particular,
$\mathscr G$ is connected.
\end{proof}

An admissible path in $\mathscr G$ transports polarization by the product of the
matrices in \eqref{eq:barotropic-gcd-shift}. By \eqref{eq:barotropic-transport-inverse}, the reverse path gives its
projective inverse. We define the following group of projective classes of the closed-path transport
maps, which is independent of the positive vertical level.

\begin{definition}
For a base vertex $\boldsymbol k\ne0$, define the holonomy group 
\begin{align*}
 \operatorname{Hol}_{\mathcal S}(\boldsymbol k)
 =\{[M_{\boldsymbol s_\ell}(\boldsymbol k_{\ell-1})\cdots
 M_{\boldsymbol s_1}(\boldsymbol k_0)]:
 \boldsymbol k_0=\boldsymbol k_\ell=\boldsymbol k,\
 \boldsymbol k_0\to\ldots\to\boldsymbol k_\ell
 \in\mathscr G\}
 \subset\operatorname{PGL}(2,\mathbb C).
\end{align*}
We call it
\emph{irreducible} if its representatives in $\operatorname{GL}(2,\mathbb C)$
have no common invariant line.  
\end{definition}

By \Cref{lem:barotropic-paths}, any two nonzero vertices
$\boldsymbol k_0,\boldsymbol k_1$ are joined by an admissible path.
If $T$ is its transport matrix, then conjugating a closed path at
$\boldsymbol k_0$ by this path and its reverse gives $\operatorname{Hol}_{\mathcal S}(\boldsymbol k_1)
=[T]\operatorname{Hol}_{\mathcal S}(\boldsymbol k_0)[T]^{-1}.$
Hence irreducibility is independent of the base vertex. The following lemma supplies a characterization  in terms of the geometry of $\mathcal S$. 

\begin{lemma}
\label{lem:lattice-holonomy} 
The group 
$\operatorname{Hol}_{\mathcal S}(\boldsymbol k)$ is irreducible if and
only if $\mathcal S$ occupies at least three unoriented lines. 
\end{lemma}

\begin{proof}
Suppose that the holonomy group has an
invariant line at one vertex. Transporting this line along admissible
paths defines lines $\ell(\boldsymbol k)\subset\mathbb C^2$ at every
nonzero vertex, which is path independent since two such paths differ
projectively by a closed-loop transport. Hence
\begin{align}\label{eq:holonomy-line-field}
 M_{\boldsymbol s}(\boldsymbol k)\ell(\boldsymbol k)
 =\ell(\boldsymbol k+\boldsymbol s),
 \qquad \boldsymbol s\wedge\boldsymbol k\ne0.
\end{align}
We first show that two directions from $\mathcal S$ already force such a line field $\ell$ to a specific form.  For nonparallel $\boldsymbol s,\boldsymbol t\in\mathcal S$, write
$\boldsymbol k=x\boldsymbol s+y\boldsymbol t$ with $x,y\in\mathbb R$ and
$\mathsf a=u\boldsymbol s^\perp+v\boldsymbol t^\perp$  with
$u,v\in\mathbb C$, and set
$d=\boldsymbol s\wedge\boldsymbol t$. In these polarization coordinates, by \eqref{eq:barotropic-gcd-shift}, 
the matrices $M_{\boldsymbol s}(\boldsymbol k)$ and $M_{\boldsymbol t}(\boldsymbol k)$ are
\begin{align*}
 d\begin{pmatrix}y&-1\\0&y\end{pmatrix},
 \qquad
 d\begin{pmatrix}-x&0\\1&-x\end{pmatrix}.
\end{align*}
If $x,y,x+1,y+1$ are nonzero, after removing the common scalar $-d^2$, the two admissible paths from
$\boldsymbol k$ to $\boldsymbol k+\boldsymbol s+\boldsymbol t$ via $\boldsymbol k+\boldsymbol s$ or $\boldsymbol k+\boldsymbol t$ have transport matrices
\begin{align*}
 P(x,y)&=\begin{pmatrix}y(x+1)&-(x+1)\\-y&y(x+1)+1\end{pmatrix},&
 Q(x,y)&=\begin{pmatrix}x(y+1)+1&-x\\-(y+1)&x(y+1)\end{pmatrix}.
\end{align*}
By path independence, both transports must send $\ell(\boldsymbol k)$ to the same line, yielding
\begin{align*}
 0=\det\left(P(x,y)\binom uv,Q(x,y)\binom uv\right)
 =-(yu+xv)^2,
\end{align*}
which forces $(u, v)=c(x,-y)$ for some  $c \in \mathbb{C}$. Therefore
\begin{align}\label{eq:holonomy-forced-line}
 \ell(\boldsymbol k)=\mathbb C(x\boldsymbol s-y\boldsymbol t)^\perp,
 \qquad xy(x+1)(y+1)\ne0.
\end{align}

Suppose now that a third direction
$\boldsymbol r=p\boldsymbol s+q\boldsymbol t\in\mathcal S$ is present.
Then $pq\ne0$, and, after removing the scalar $d$, the transport matrix for $M_{\boldsymbol r}(\boldsymbol k)$ is
\begin{align*}
 T_{\boldsymbol r}(\boldsymbol k)=(py-qx)I+\binom pq\begin{pmatrix}q&-p\end{pmatrix}.
\end{align*}
Using \eqref{eq:holonomy-forced-line}, 
\begin{align}\label{eq:holonomy-third-direction}
 \det\left(T_{\boldsymbol r}(\boldsymbol k)\binom{x}{-y},
 \binom{x+p}{-y-q}\right)
 =-2pq(2xy+py+qx).
\end{align}
Taking $\boldsymbol k=M\boldsymbol s+M^2\boldsymbol t$ with $M$
sufficiently large ensures that \eqref{eq:holonomy-forced-line}
applies at both $\boldsymbol k$ and $\boldsymbol k+\boldsymbol r$, that
the $\boldsymbol r$-shift is admissible, and that the right-hand side
of \eqref{eq:holonomy-third-direction} is nonzero. This contradicts
\eqref{eq:holonomy-line-field}.

Conversely, suppose that $\mathcal S$ occupies only two lines, and let
$\boldsymbol s,\boldsymbol t$ be their primitive lattice generators.
Then
\begin{align}\label{eq:invariant-line}
 \ell(x\boldsymbol s+y\boldsymbol t)
 =\mathbb C(x\boldsymbol s-y\boldsymbol t)^\perp
\end{align}
defines a line field on $\mathbb Z^2\setminus\{0\}$. A step
$p\boldsymbol s$ sends $(x,-y)$ to a nonzero multiple of
$(x+p,-y)$, while a step $q\boldsymbol t$ sends it to a multiple of
$(x,-y-q)$. Hence every admissible edge preserves this line field, so
the holonomy is reducible. The lattice condition $\operatorname{span}_{\mathbb Z}\mathcal S=\mathbb Z^2$ excludes a single
direction.
\end{proof}
In particular, if a nonzero polarization is available at a positive-level
vertex and the transport holonomy is irreducible, some closed admissible
path transports it to a linearly independent polarization; hence the full
polarization fibre is generated at that vertex.

We now use the preceding transport to fill the polarization fibres at
the selected vertical levels $n_r$ from \eqref{eq:select-kna}. Recall the complexified full extension $\mathcal E_{\infty}$ from \eqref{eq:complexified-full-extension}. 

\begin{lemma}\label{lem:selected-level-full-fibres}
For every vertical level $n_r$ selected in \eqref{eq:select-kna}, every polarization $\mathsf a\in\mathbb C^2$ and $\boldsymbol k\ne0$, we have $\phi_{\boldsymbol k,n_r,\mathsf a}\in\mathcal E_\infty$, i.e., the full polarization fibre is
generated at every nonzero horizontal frequency. 

Moreover, for every $\boldsymbol k\ne0$ and $\mathsf a\in\mathfrak P_{(\boldsymbol k,0)}$, we have $\phi_{\boldsymbol k,0,\mathsf a}\in\mathcal E_\infty$, i.e., every barotropic hydrostatic Fourier mode is generated. 
\end{lemma}

\begin{proof}
We first place a nonzero polarization at a nonzero horizontal
frequency. If the witness
$\phi_{\boldsymbol k_r,n_r,\mathsf a_r}$ has
$\boldsymbol k_r=0$, choose $\boldsymbol s\in\mathcal S$ with
$\boldsymbol s\cdot\mathsf a_r\ne0$; such a choice exists since
$\mathcal S$ spans $\mathbb Z^2$. Then \eqref{eq:barotropic-gcd-shift} gives 
\begin{align*}
 \mathsf B(\phi_{0,n_r,\mathsf a_r},
 \phi_{\boldsymbol s,0,\boldsymbol s^\perp})
 =\mathrm{i}(\boldsymbol s\cdot\mathsf a_r)
 \phi_{\boldsymbol s,n_r,\boldsymbol s^\perp}\ne0.
\end{align*}
Thus every selected level $n_r$ contains a nonzero polarization at a nonzero
horizontal vertex, and \Cref{lem:barotropic-paths} transports such a
polarization to every nonzero horizontal frequency.

If $\mathcal S$ occupies at least three unoriented lines, then
\Cref{lem:lattice-holonomy} gives irreducible holonomy. Hence, at each
selected level, some closed admissible path transports any given
nonzero polarization to a linearly independent one. The fibre is
therefore full at one vertex, and invertible admissible transport makes
it full at every nonzero horizontal frequency.

Suppose next that $\mathcal S$ occupies only two unoriented lines. For
nonparallel $\boldsymbol s,\boldsymbol t\in\mathcal S$, the quadratic interaction \eqref{eq:symmetric-quadratic-interaction} yields
\begin{align}\label{eq:barotropic-third-direction}
 \mathsf B(\phi_{\boldsymbol s,0,\boldsymbol s^\perp},
 \phi_{\boldsymbol t,0,\boldsymbol t^\perp})
 =\mathrm{i}\frac{(\boldsymbol s\wedge\boldsymbol t)
 (|\boldsymbol t|^2-|\boldsymbol s|^2)}
 {|\boldsymbol s+\boldsymbol t|^2}
 \phi_{\boldsymbol s+\boldsymbol t,0,
 (\boldsymbol s+\boldsymbol t)^\perp}.
\end{align}
Hence, if $|\boldsymbol s|\ne|\boldsymbol t|$, the interaction
\eqref{eq:barotropic-third-direction} generates a third barotropic
direction. Enlarging $\mathcal S$ symmetrically by this direction,
\Cref{lem:lattice-holonomy} reduces the argument to the irreducible
holonomy case above.

It remains to consider the case in which every pair belonging to the
two different lines has equal length. Then $\mathcal S=\{\pm\boldsymbol e_1,\pm\boldsymbol e_2\}$  since it generates
$\mathbb Z^2$.  We now escape the invariant line of this exceptional geometry, based on the observation that fibres of level $n_r$ could be trapped in the invariant line, but quadratic interaction can produce full fibre at level $2n_r$. For $n\geq 1$, a direct computation through \eqref{eq:fourier-modes} and \eqref{eq:symmetric-quadratic-interaction} gives for $\boldsymbol q=\boldsymbol k+\boldsymbol l\ne0$, 
\begin{align}\label{eq:equal-level-noise-generation}
\begin{split}
 \mathsf B\bigl(
 \phi_{\boldsymbol k,n,\mathsf a},
 \phi_{\boldsymbol l,n,\mathsf b}
 \bigr)
 =
 \phi_{\boldsymbol q,2n,\,
 \frac{\mathrm{i}}2\left\{
 (\mathsf a\cdot(\boldsymbol l-\boldsymbol k))\mathsf b
 -(\mathsf b\cdot(\boldsymbol l-\boldsymbol k))\mathsf a
 \right\}}+
 \phi_{\boldsymbol q,0,\,
 \frac{\mathrm{i}}{2|\boldsymbol q|^2}
 \left\{
 (\mathsf a\cdot\boldsymbol q)(\boldsymbol q\wedge\mathsf b)
 +(\mathsf b\cdot\boldsymbol q)(\boldsymbol q\wedge\mathsf a)
 \right\}\boldsymbol q^\perp}.
\end{split}
\end{align}
If $\boldsymbol k+\boldsymbol l=0$, the barotropic component vanishes
identically.
Fix a selected level $n_r$ and transport a nonzero polarization
$\mathsf a=(a_1,a_2)$ to $(1,1)$ by an admissible path in $\mathscr G$ in \Cref{lem:barotropic-paths}. Besides, the two admissible paths from
$(1,1)$ to $(2,2)$ via $(2,1)$ and $(1,2)$ give, up to a nonzero factor,
\begin{align*}
 \begin{pmatrix}2&1\\0&2\end{pmatrix}
 \begin{pmatrix}1&0\\1&1\end{pmatrix}\mathsf a
 =\begin{pmatrix}3&1\\2&2\end{pmatrix}\mathsf a, \quad \text{ and } \begin{pmatrix}2&0\\1&2\end{pmatrix}
 \begin{pmatrix}1&1\\0&1\end{pmatrix}\mathsf a
 =\begin{pmatrix}2&2\\1&3\end{pmatrix}\mathsf a,
\end{align*}
whose determinant is $-(a_1-a_2)^2$. If $a_1\ne a_2$, level $n_r$
already has a full fibre at $(2,2)$, and hence at every nonzero
horizontal frequency.

Assume therefore that $a_1=a_2$, and rescale
$\mathsf a=(1,1)$. The invariant
line field from \Cref{lem:lattice-holonomy} is $\ell(k_1,k_2)=\mathbb C(k_2,k_1)$. 
Put $R(k_1,k_2)=(k_2,k_1)$, so that
$\ell(\boldsymbol k)=\mathbb C R\boldsymbol k$. From
\eqref{eq:barotropic-gcd-shift}, one has 
\begin{align*}
 \mathsf B(\phi_{\boldsymbol k,n_r,R\boldsymbol k},
 \phi_{\pm\boldsymbol e_1,0,\boldsymbol e_2})
 =\mathrm{i}k_2\phi_{\boldsymbol k\pm\boldsymbol e_1,n_r,
 R(\boldsymbol k\pm\boldsymbol e_1)},\, \text{ and }
 \mathsf B(\phi_{\boldsymbol k,n_r,R\boldsymbol k},
 \phi_{\pm\boldsymbol e_2,0,\boldsymbol e_1})
 =\mathrm{i}k_1\phi_{\boldsymbol k\pm\boldsymbol e_2,n_r,
 R(\boldsymbol k\pm\boldsymbol e_2)}
\end{align*}
showing that the invariant
polarization $R\boldsymbol k$ is available at every
$\boldsymbol k\ne0$. Apply
\eqref{eq:equal-level-noise-generation} with
\begin{align*}
 (\boldsymbol k,\boldsymbol l)
 =((0,1),(1,-1)),\qquad
 (\boldsymbol k,\boldsymbol l)
 =((1,1),(0,-1)),
 \qquad
 \mathsf a=R\boldsymbol k,\quad \mathsf b=R\boldsymbol l.
\end{align*}
The corresponding level-$2n_r$ outputs at horizontal frequency
$\boldsymbol e_1$ have polarizations $(2,1)$ and $(2,-1)$, while their barotropic components
are multiples of the already generated mode
$\phi_{\boldsymbol e_1,0,\boldsymbol e_2}$. Subtracting these
barotropic components therefore yields $\phi_{\boldsymbol e_1,2n_r,(2,\pm 1)}
\in\mathcal E_\infty.$
Since these polarizations are linearly independent, level $2n_r$ has
a full fibre at $\boldsymbol e_1$. By \Cref{lem:barotropic-paths} and
the invertibility of admissible transport, it therefore has a full
fibre at every nonzero horizontal frequency.

We have therefore obtained in every case a positive level $n$ with a
full polarization fibre at every nonzero horizontal frequency. Such
a level generates every barotropic mode. Indeed, for
$\boldsymbol q\ne0$, choose nonzero $\boldsymbol k, \boldsymbol l$ with 
$\boldsymbol k+\boldsymbol l=\boldsymbol q$ and set
$\mathsf a=\boldsymbol q+\boldsymbol q^\perp$. The level-$2n$
component in \eqref{eq:equal-level-noise-generation} vanishes, whereas
\begin{align*}
 \mathsf B(\phi_{\boldsymbol k,n,\mathsf a},
 \phi_{\boldsymbol l,n,\mathsf a})
 =\mathrm{i}|\boldsymbol q|^2
 \phi_{\boldsymbol q,0,\boldsymbol q^\perp}.
\end{align*}
Hence every barotropic hydrostatic Fourier mode is generated. In the
exceptional two-direction case this produces, in particular, a third
barotropic direction. Enlarging $\mathcal S$ and applying \Cref{lem:lattice-holonomy}  makes the holonomy irreducible.
Since every selected level already contains a nonzero polarization at
a nonzero horizontal vertex, the preceding holonomy argument then
gives a full fibre at every selected level and every nonzero horizontal
frequency. 
\end{proof}

It remains to reduce the finitely many selected levels to fixed seed
levels one and two by vertical generation. Let $q=(\boldsymbol k,n)$ and
$r=(\boldsymbol l,m)$ with $n>m\ge1$ and
$\boldsymbol K=\boldsymbol k+\boldsymbol l$. Direct calculation gives a sum over upper channel $(\boldsymbol K,n+m)$ and lower channel $(\boldsymbol K,n-m)$, 
\begin{align}\label{eq:two-channel-interaction}
 \mathsf B(\phi_{\boldsymbol k,n,\mathsf a},
 \phi_{\boldsymbol l,m,\mathsf b})
 =\frac{\mathrm{i}}{2}\phi_{\boldsymbol K,n+m,
 \Gamma^+_{q,r}(\mathsf a,\mathsf b)}
 +\frac{\mathrm{i}}{2}\phi_{\boldsymbol K,n-m,
 \Gamma^-_{q,r}(\mathsf a,\mathsf b)},
\end{align}
where 
\begin{align}\label{eq:two-channel-symbols}
 \Gamma^+_{q,r}(\mathsf a,\mathsf b)
 &=\frac{\mathsf a\cdot(n\boldsymbol l-m\boldsymbol k)}{n}\mathsf b
 -\frac{\mathsf b\cdot(n\boldsymbol l-m\boldsymbol k)}{m}\mathsf a,
 \notag\\
 \Gamma^-_{q,r}(\mathsf a,\mathsf b)
 &=\frac{\mathsf a\cdot(n\boldsymbol l+m\boldsymbol k)}{n}\mathsf b
 +\frac{\mathsf b\cdot(n\boldsymbol l+m\boldsymbol k)}{m}\mathsf a.
\end{align}
In addition, the map 
\begin{align*}
 \Sigma_{q,r}:\mathbb C^2\otimes\mathbb C^2
 \longrightarrow\mathbb C^2\oplus\mathbb C^2,\quad \text{by }
 \mathsf a\otimes\mathsf b
 \longmapsto
 \bigl(\Gamma^+_{q,r}(\mathsf a,\mathsf b),
 \Gamma^-_{q,r}(\mathsf a,\mathsf b)\bigr)
\end{align*}
satisfies 
\begin{align}\label{eq:two-channel-determinant}
 \det\Sigma_{q,r}
 =\frac{4(m^2-n^2)(m^2+n^2)}{m^2n^2}
 (\boldsymbol k\wedge\boldsymbol l)^2.
\end{align}
Hence, whenever $\boldsymbol k\wedge\boldsymbol l\ne0$,
$\Sigma_{q,r}$ is invertible. If the full polarization fibres at
$q$ and $r$ have already been generated, then for arbitrary
$\mathsf c_+,\mathsf c_-\in\mathbb C^2$ a finite linear combination of
the interactions in \eqref{eq:two-channel-interaction} can be chosen
with output polarizations $\mathsf c_+$ and $\mathsf c_-$ in the two
channels, respectively. In particular, either channel can be prescribed
arbitrarily while the other is made to vanish.

We shall also use the following uniform version. For any fixed
linearly independent $\boldsymbol h,\boldsymbol t\in\mathbb Z^2$,
direct inversion of \eqref{eq:two-channel-symbols} gives
\begin{align}\label{eq:uniform-channel-inverse}
 \sup_{n\ge2}
 \left\|
 \Sigma_{(\boldsymbol h,n),(\boldsymbol t,1)}^{-1}
 \right\|
 \le C_{\boldsymbol h,\boldsymbol t}.
\end{align}
Indeed, after the fixed change of coordinates
$P=(\boldsymbol h\ \boldsymbol t)$, this reduces to
$(\boldsymbol h,\boldsymbol t)=(\boldsymbol e_1,\boldsymbol e_2)$,
where the inverse coefficients are rational functions of $n$ with
denominator $n^4-1$ and numerators of degree at most four; hence they
are uniformly bounded for $n\ge2$.

The next lemma provides the initial seeds by vertical generation to start the quantitative generation later. 
\begin{lemma}
\label{lem:finite-initialization}
For every $\boldsymbol k\ne0$, we have 
\begin{align}\label{eq:level-1-2}
 \phi_{\boldsymbol k,n,\mathsf a}\in\mathcal E_\infty
 \qquad
 \text{for }n\in\{1,2\},\ \mathsf a\in\mathbb C^2.
\end{align}
Moreover, for any fixed  nonparallel 
$\boldsymbol h,\boldsymbol t\in\mathcal S$, there exists
$j_0<\infty$, depending only on the forcing family, such that $\phi_{\boldsymbol s,0,\boldsymbol s^\perp}
\in E_{j_0}^{\mathbb C}$ for $\boldsymbol s\in\mathcal S$ and for $i=1,2$, 
\begin{align*}
 &\phi_{\boldsymbol h,1,\boldsymbol e_i},\
 \phi_{\boldsymbol h,2,\boldsymbol e_i},\
 \phi_{\boldsymbol t,1,\boldsymbol e_i}
 \in E_{j_0}^{\mathbb C}. 
\end{align*}
\end{lemma}

\begin{proof}
Let $\mathbb V$ be the set of positive vertical levels whose
polarization fibres are full at every nonzero horizontal frequency.
By \Cref{lem:selected-level-full-fibres}, $n_1,\ldots,n_{\ell_*}\in\mathbb V$ 
and every barotropic hydrostatic Fourier mode is generated.

Pick any nonparallel
$\boldsymbol h,\boldsymbol t\in\mathcal S$. If
$n,m\in\mathbb V$ with $n\ne m$, then the fibres at
$(\boldsymbol h,n)$ and $(\boldsymbol t,m)$ are both equal to
$\mathbb C^2$. Since
$\boldsymbol h\wedge\boldsymbol t\ne0$,
\eqref{eq:two-channel-determinant} shows that
$\Sigma_{(\boldsymbol h,n),(\boldsymbol t,m)}$ is invertible. Hence,
for every $\mathsf c\in\mathbb C^2$, finite linear combinations of
the interactions in \eqref{eq:two-channel-interaction} may be chosen 
with output polarizations $(\mathsf c,0)$ or $(0,\mathsf c)$
in the $(n+m)$ and $|n-m|$ channels respectively. Therefore
\begin{align*}
 \phi_{\boldsymbol h+\boldsymbol t,n+m,\mathsf c},
 \qquad
 \phi_{\boldsymbol h+\boldsymbol t,|n-m|,\mathsf c}
 \in\mathcal E_\infty
 \qquad
 \text{for every }\mathsf c\in\mathbb C^2.
\end{align*}
By \Cref{lem:barotropic-paths} and invertible admissible transport, $n+m\in\mathbb V,\, |n-m|\in\mathbb V$. 
Since $\mathbb V$ is closed under positive differences, repeated
application of the Euclidean algorithm gives
\begin{align*}
 \gcd(n_1,\ldots,n_{\ell_*})\in\mathbb V.
\end{align*}
By \eqref{eq:select-kna}, one has $1\in\mathbb V$, and therefore $\phi_{\boldsymbol k,1,\mathsf a}\in\mathcal E_\infty$ for any $\boldsymbol k\ne 0$ and $\mathsf a\in\mathbb C^2$. To obtain level two, use \eqref{eq:equal-level-noise-generation} at
$n=1$ with $(\boldsymbol k,\boldsymbol l,\mathsf a,\mathsf b)
=((0,1),(1,-1),(1,0),(-1,1))$ and 
\begin{align*}
 (\boldsymbol k,\boldsymbol l,\mathsf a,\mathsf b)
 =((1,1),(0,-1),(1,1),(-1,0)).
\end{align*}
The level-two components have horizontal frequency
$\boldsymbol e_1$ and polarizations $(2,1)$ and $(2,-1)$ 
respectively. Their barotropic components are already generated and
may be subtracted. Thus level two has a full fibre at
$\boldsymbol e_1$, and admissible horizontal transport gives $2\in\mathbb V$, proving \eqref{eq:level-1-2}. 
In particular, for $i=1,2$, $\phi_{\boldsymbol h,1,\boldsymbol e_i},\,
\phi_{\boldsymbol h,2,\boldsymbol e_i},\,
\phi_{\boldsymbol t,1,\boldsymbol e_i}
\in\mathcal E_\infty$. 
Thus these six modes belong to a common
$E_{j_1}^{\mathbb C}$ for some $j_1<\infty$. On the other hand, by
the definition of $\mathcal S$ in \eqref{eq:mathcal-S}   and the one-dimensionality of each
barotropic fibre, $ \phi_{\boldsymbol s,0,\boldsymbol s^\perp}
\in E_{j_*}^{\mathbb C}$ for $\boldsymbol s\in\mathcal S$. 
Taking $j_0=\max\{j_1,j_*\}$ proves the remaining assertion.
\end{proof}

The preceding argument contains all forcing-dependent finite geometry.
From now on, only two operations accumulate with the target frequency:
successive vertical generation and admissible horizontal transport.
The former is controlled by \eqref{eq:uniform-channel-inverse}. For the
latter, an admissible shift satisfies
\begin{align}\label{eq:uniform-shift-inverse}
 M_{\boldsymbol s}(\boldsymbol k)^{-1}
 =
 \frac{I}{\boldsymbol s\wedge\boldsymbol k}
 -\frac{\boldsymbol s^\perp\otimes\boldsymbol s}
 {(\boldsymbol s\wedge\boldsymbol k)^2},
 \qquad
 \|M_{\boldsymbol s}(\boldsymbol k)^{-1}\|
 \le1+|\boldsymbol s|^2,
\end{align}
because $\boldsymbol s\wedge\boldsymbol k$ is a nonzero integer. Since
$\mathcal S$ is fixed and finite, the inverse cost is uniform over all
admissible horizontal shifts.

In what follows  a Lie monomial is any bracket expression (not necessarily right-nested as in \eqref{eq:iteratedLiebracket}) obtained recursively from
$F,b_1,\ldots,b_m$, with arbitrary bracketing, and its parabolic degree is defined as in \eqref{eq:brownian-parabolic-degree} by
\begin{align}\label{eq:Lie-poly-degree}
 \operatorname{deg}(F)=2,\qquad
 \operatorname{deg}(b_j)=1,\qquad
 \operatorname{deg}([\mathcal L_1,\mathcal L_2])
 =\operatorname{deg}(\mathcal L_1)+\operatorname{deg}(\mathcal L_2).
\end{align}

\begin{proposition}
\label{prop:fourier-generation}
Assume the quantitative quadratic
saturation in the sense of \Cref{def:fourier-certificate}. Then there
is $C<\infty$, depending only on the fixed forcing family, such that
for every $(\boldsymbol k,n)\in\mathbf F$ and every
$\mathsf a\in\mathfrak P_{(\boldsymbol k,n)}$ with $|\mathsf a|=1$,
there exist finitely many coefficients $c_\beta\in\mathbb C$ and Lie
monomials $\mathcal L_\beta$ in $F,b_1,\ldots,b_m$ such that $\phi_{\boldsymbol k,n,\mathsf a}
=\sum_\beta c_\beta\mathcal L_\beta$ with 
\begin{align}\label{eq:direct-fourier-generation}
 \max_{c_\beta\ne0}\operatorname{deg}(\mathcal L_\beta)
 \le C(1+|\boldsymbol k|+n),\qquad
 \sum_\beta|c_\beta|
 \le \exp\{C(1+|\boldsymbol k|+n)\}.
\end{align}
\end{proposition}

\begin{proof}
We split the proof  into five steps. 

\emph{Step 1: Finite seeds initialization.} Work in the complexification of the real quadratic algebra and fix
nonparallel $\boldsymbol h,\boldsymbol t\in\mathcal S$. Define the set of finite seeds by 
\begin{align*}
\mathbb S: = \{\phi_{\boldsymbol h,1,\boldsymbol e_i}, \phi_{\boldsymbol h,2,\boldsymbol e_i},\phi_{\boldsymbol t,1,\boldsymbol e_i},\phi_{\boldsymbol s,0,\boldsymbol s^{\perp}}:i=1,2,\boldsymbol s\in \mathcal S \}. 
\end{align*}
By \Cref{lem:finite-initialization}, $\mathbb S\subset E_{j_0}^{\mathbb C}$. 
Then the recursive definition
\eqref{eq:quadratic-saturation} expresses each element in  $\mathbb S$ as a finite
linear combination of terms obtained from $b_1,\ldots,b_m$ by
successive applications of $\mathsf B$. For constant vector fields
$f,g$,
\begin{align}\label{eq:quadratic-lie}
 \mathsf B(f,g)=[f,[F,g]].
\end{align}
Replacing each quadratic operation by this double Lie bracket therefore
gives a finite Lie-polynomial representation for each $\phi\in \mathbb S$ as 
\begin{align}\label{eq:bounds-on-seeds}
\phi = \sum_{\beta}c_{\phi,\beta}\mathcal L_{\phi,\beta}, \, \text{ with } \max_{c_{\phi,\beta}\ne0,\phi\in\mathbb S}\operatorname{deg}(\mathcal L_{\phi,\beta})
\le C_*,\, \text{ and }
\max_{\phi\in\mathbb S}\sum_\beta|c_{\phi,\beta}|
\le C_*
\end{align}
for $C_*=C_*(\mathbb S)$. 

Moreover, if $f$ and
$g$ have Lie-polynomial representations of maximal degrees $L_f,L_g$
and coefficient sums bounded by $A_f,A_g$, then $\mathsf B(f,g)$ has a
representation of maximal degree at most
\begin{align}\label{eq:quadratic-degree-cost}
 L_f+L_g+2
\end{align}
and coefficient sum bounded by $A_fA_g$.

\emph{Step 2: Vertical growth estimate.}  
By induction, assume that the Lie-polynomial representations of
$\phi_{\boldsymbol h,n,\boldsymbol e_i}$, $i=1,2$ are already available with uniform bounds $L_n,A_n$ on the maximal parabolic degree and
coefficient sum. Note that \eqref{eq:bounds-on-seeds} gives such representation and fixed bounds
for the base case $n=1,2$. For $n\ge2$ and $i=1,2$, set
\begin{align*}
 \mathsf c_i
 =M_{-\boldsymbol t}(\boldsymbol h+\boldsymbol t)^{-1}\boldsymbol e_i,
 \qquad
 Z_i
 =\Sigma_{(\boldsymbol h,n),(\boldsymbol t,1)}^{-1}(\mathsf c_i,0).
\end{align*}
By \eqref{eq:barotropic-transport-inverse} and
\eqref{eq:uniform-channel-inverse},
$|\mathsf c_i|+\|Z_i\|\le C$ uniformly in $n$. Expanding $Z_i$ in
the basis $\{\boldsymbol e_\alpha\otimes\boldsymbol e_\beta\}_{\alpha,\beta=1}^2$
and using the available representations of
$\phi_{\boldsymbol h,n,\boldsymbol e_\alpha}$ and
$\phi_{\boldsymbol t,1,\boldsymbol e_\beta}$ in
\eqref{eq:two-channel-interaction} generates
$\phi_{\boldsymbol h+\boldsymbol t,n+1,\mathsf c_i}$ with vanishing
lower channel. Combining it with the admissible horizontal shift $-\boldsymbol t$ then gives the Lie-polynomial representation for 
$\phi_{\boldsymbol h,n+1,\boldsymbol e_i}$. By
\eqref{eq:quadratic-degree-cost} we obtain bounds on the maximal parabolic degree and
coefficient sum as 
\begin{align*}
 L_{n+1}\le L_n+C_0,\qquad
 A_{n+1}\le C_1A_n.
\end{align*}
Consequently, by induction, each $\phi_{\boldsymbol h,n,\boldsymbol e_i}$, $i=1,2, n\ge 1$ admits Lie-polynomial representation with maximal parabolic degree and
coefficient sum bounds 
\begin{align}\label{eq:vertical-generation-bound}
 L_n\le Cn,\qquad A_n\le C^n. 
\end{align}

\emph{Step 3: Horizontal propagation.}
Fix $n\ge1$ and $\boldsymbol k\ne0$. By
\Cref{lem:barotropic-paths}, there is an admissible path $\mathscr P$
from $\boldsymbol h$ to $\boldsymbol k$, through quadratic interaction with $\phi_{\boldsymbol s,0,\boldsymbol s^{\perp}}$ for some $\boldsymbol s\in\mathcal S$ at each step, with at most
$C(1+|\boldsymbol k|)$ edges. To generate
$\phi_{\boldsymbol k,n,\boldsymbol e_i}$, $i=1,2$, we transport the
target polarization $\boldsymbol e_i$ backwards along $\mathscr P$ by
successively applying the inverse maps from
\eqref{eq:barotropic-transport-inverse}. By
\eqref{eq:uniform-shift-inverse}, the resulting polarization
$\mathsf a_i$ at $\boldsymbol h$ satisfies 
\begin{align*}
|\mathsf a_i|
\le \sup_{s\in\mathcal S}\left(1+|\boldsymbol s|^2\right)^{C(1+|\boldsymbol k|)}|\boldsymbol e_i| \le \exp\{C(1+|\boldsymbol k|)\}. 
\end{align*} 
Expanding $\mathsf a_i$ in the basis 
$\boldsymbol e_1,\boldsymbol e_2$ to have $\phi_{\boldsymbol h,n,\mathsf a_i}=a_{i,1}\phi_{\boldsymbol h,n,\boldsymbol e_1}+a_{i,2}\phi_{\boldsymbol h,n,\boldsymbol e_2}$ and using the consequence of Step 2 
gives a Lie-polynomial
representation of $\phi_{\boldsymbol h,n,\mathsf a_i}$ with maximal parabolic degree at
most $L_n$ and coefficient sum at most
$(|a_{i,1}|+|a_{i,2}|)A_n$. Following
$\mathscr P$ forward then gives a Lie-polynomial
representation for 
$\phi_{\boldsymbol k,n,\boldsymbol e_i}$. Since the barotropic
interaction $\phi_{\boldsymbol s,0,\boldsymbol s^{\perp}}$ for $\boldsymbol s\in\mathcal S$ along each edge of $\mathscr P$ has fixed parabolic degree and coefficient cost given  by \eqref{eq:bounds-on-seeds}, combined with \eqref{eq:quadratic-degree-cost} and \eqref{eq:vertical-generation-bound},  the
resulting Lie-polynomial representation for $\phi_{\boldsymbol k,n,\boldsymbol e_i}$ has maximal parabolic degree bounded by 
\begin{align*}
 L_n+C_*(1+|\boldsymbol k|)\le  C(1+|\boldsymbol k|+n)
\end{align*}
and coefficient sum at most
\begin{align*}
 (|a_{i,1}|+|a_{i,2}|)A_n C_*^{C(1+|\boldsymbol k|)}\le C|\mathsf a_i| A_n C_*^{C(1+|\boldsymbol k|)} \le \exp\{C(1+|\boldsymbol k|+n)\}.
\end{align*}
By linearity, the same bounds hold for
$\phi_{\boldsymbol k,n,\mathsf a}$ for every unit
$\mathsf a\in\mathbb C^2$.

\emph{Step 4: Pure vertical modes.} To produce zero horizontal frequency we must pair opposite horizontal
frequencies.  In this collinear configuration the two-channel map
\eqref{eq:two-channel-interaction} is degenerate, so we separate the
upper and lower vertical channels by using two crossed choices of
polarization. 
For $n\ge1$, set $ U_n
=\mathsf B(\phi_{\boldsymbol h,n,\boldsymbol h^\perp},
\phi_{-\boldsymbol h,1,\boldsymbol h})$ and 
\begin{align*}
 V_n
 =\mathsf B(\phi_{\boldsymbol h,n,\boldsymbol h},
 \phi_{-\boldsymbol h,1,\boldsymbol h^\perp}).
\end{align*}
The collinear specialization of \eqref{eq:two-channel-symbols} gives
\begin{align}\label{eq:UnVn}
\begin{split}
 U_n-nV_n
 &=\mathrm{i}(n+1)|\boldsymbol h|^2
 \phi_{0,n+1,\boldsymbol h^\perp},
 \qquad n\ge1,\\
 U_n+nV_n
 &=-\mathrm{i}(n-1)|\boldsymbol h|^2
 \phi_{0,n-1,\boldsymbol h^\perp},
 \qquad n\ge2.
\end{split}
\end{align}
The first identity generates the $\boldsymbol h^\perp$ polarization at
every pure vertical level at least two, while the second identity with
$n=2$ gives level one.  Step 3 gives Lie-polynomial
representations for $\phi_{\boldsymbol h,n,\boldsymbol h^\perp}, \, \phi_{-\boldsymbol h,1,\boldsymbol h}, \phi_{\boldsymbol h,n,\boldsymbol h}$ and $\phi_{-\boldsymbol h,1,\boldsymbol h^\perp}$, which combined with \eqref{eq:quadratic-lie} gives Lie-polynomial
representation for $\phi_{0,n,\boldsymbol h^\perp}$, $n\ge 1$. Repeating the construction with
$\boldsymbol t$ gives the Lie-polynomial
representation for $\phi_{0,n,\boldsymbol t^\perp}$. Since $\boldsymbol h^\perp$ and $\boldsymbol t^\perp$ are fixed and
linearly independent, every unit $\mathsf a\in\mathbb C^2$ can be
written as
$\mathsf a=a_1\boldsymbol h^\perp+a_2\boldsymbol t^\perp$
with $|a_1|+|a_2|\le C$, yielding the Lie-polynomial
representation for $\phi_{0,n,\mathsf a}$. 

After dividing \eqref{eq:UnVn} by the displayed factors, the
coefficients multiplying $U_n$ and $V_n$ remain uniformly bounded. Therefore by Step 3, the Lie-polynomial
representation for $\phi_{0,n,\mathsf a}$ has maximal parabolic degree bounded by $ C(1+n)$ and coefficient sums bounded by $\exp\{C(1+n)\}$.

\emph{Step 5: Barotropic modes.}
Let $\boldsymbol k\ne0$. To isolate the barotropic channel, we use two
level-one inputs with the same polarization, which cancels the
level-two component in \eqref{eq:equal-level-noise-generation}. Choose
$\boldsymbol q\in\{\boldsymbol h,\boldsymbol t\}$ so that
$\boldsymbol l=\boldsymbol k-\boldsymbol q\ne0$, and put
\begin{align*}
 \mathsf a
 =\frac{\boldsymbol k+\boldsymbol k^\perp}{|\boldsymbol k|}.
\end{align*}
Since $\boldsymbol q,\boldsymbol l\ne0$, Step~3 gives quantitative
Lie-polynomial representations of
$\phi_{\boldsymbol q,1,\boldsymbol e_i}$ and
$\phi_{\boldsymbol l,1,\boldsymbol e_i}$, $i=1,2$. By linearity in
the polarization variable, the same therefore holds for
$\phi_{\boldsymbol q,1,\mathsf a}$ and
$\phi_{\boldsymbol l,1,\mathsf a}$. It then follows from \eqref{eq:equal-level-noise-generation} that 
\begin{align*}
 \mathsf B(\phi_{\boldsymbol q,1,\mathsf a},
 \phi_{\boldsymbol l,1,\mathsf a})
 =
 \mathrm{i}|\boldsymbol k|
 \phi_{\boldsymbol k,0,\boldsymbol k^\perp/|\boldsymbol k|},
\end{align*}
yielding the Lie-polynomial representation for $\phi_{\boldsymbol k,0,\boldsymbol k^\perp/|\boldsymbol k|}$ by \eqref{eq:quadratic-lie}.  Since $|\boldsymbol l|\le|\boldsymbol k|+C$ and division by
$|\boldsymbol k|\ge1$ does not increase the coefficient cost, the quantitative estimate in Step~3 for $\phi_{\boldsymbol q,1,\mathsf a}$ and
$\phi_{\boldsymbol l,1,\mathsf a}$, combined with \eqref{eq:quadratic-degree-cost}, shows that 
the maximal parabolic degree is bounded by $C(1+|\boldsymbol k|)$ and
coefficient sum by $\exp\{C(1+|\boldsymbol k|)\}$. 

The positive vertical frequency with nonzero horizontal frequency in Step~3, the pure
vertical mode in Step~4, and the barotropic modes in Step~5  exhaust all hydrostatic frequencies in $\mathbf F$, and the corresponding quantitative estimates prove  \eqref{eq:direct-fourier-generation}.
\end{proof}

The following elementary quantitative consequence of the
Dynkin--Specht--Wever identity straightens the Lie polynomials above into
the right-nested form as in \eqref{eq:iteratedLiebracket} required by the observable hierarchy. In addition to the parabolic degree, we temporarily call the number of letters in a Lie monomial its ordinary length. Since every letter has parabolic degree at least one, its ordinary length is bounded by its parabolic degree. By construction in
\Cref{prop:fourier-generation}, every Lie monomial occurring in the
representations below contains at least one noise symbol.

\begin{lemma}
\label{lem:straightening}
Any Lie polynomial representation obtained in \Cref{prop:fourier-generation} with maximal parabolic degree $L$ and coefficient sum $A$, can be written as a linear combination of right-nested words
\begin{align*}
 \operatorname{ad}_{a_1}\cdots\operatorname{ad}_{a_r}b_j,\qquad a_i\in\{F,b_1,\ldots,b_m\},
\end{align*}
of parabolic degree at most $L$, with coefficient sum at most $2^LA$.
\end{lemma}

\begin{proof}
We first straighten the bracketing. Consider a homogeneous component of ordinary length $r$. Expanding each commutator by $[X,Y]=XY-YX$ in the free associative algebra  where $F,b_1,\ldots,b_m$ are treated as noncommuting formal symbols,  a bracket tree with $r$ leaves and hence $r-1$ brackets produces at most $2^{r-1}$ associative words.

For an associative word $a_1\cdots a_r$, by the Dynkin--Specht--Wever identity
\cite{BonfiglioliFulci2012}, the Dynkin map, which sends
\begin{align*}
 a_1\cdots a_r\longmapsto [a_1,[a_2,\ldots,[a_{r-1},a_r]\ldots]],
\end{align*}
acts as multiplication by
$r$ on homogeneous Lie polynomials of ordinary length $r$. Hence, after expanding the component into associative words, applying the Dynkin map term by term and dividing by $r$ expresses it as a linear combination of right-nested brackets. The coefficient sum increases by at most $2^{r-1}/r\le 2^r$.

This operation only changes the placement of brackets and does not change any of the letters, so the parabolic degree is preserved. Since $r\le L$, summing over the homogeneous components gives a right-nested representation of parabolic degree at most $L$ and coefficient sum at most $2^LA$.

It remains to place a noise symbol at the innermost position. A right-nested word ending in $F$ either has innermost bracket $[F,F]=0$, or has innermost bracket $[b_j,F]=-[F,b_j]$, which moves $b_j$ to the terminal position at the cost of only a sign.  Thus every nonzero term can be written in the required form without changing its parabolic degree or coefficient sum.
\end{proof}

\begin{proof}[Proof of \Cref{prop:geometric-generation}]
Assume quantitative quadratic saturation as in \Cref{def:fourier-certificate}. By
\Cref{lem:finite-initialization} and \Cref{prop:fourier-generation}, every
hydrostatic Fourier mode admits a constant Lie-polynomial
representation with linear weight and exponential coefficient cost.
Applying \Cref{lem:straightening} gives the right-nested representation
and the bounds in \eqref{eq:geometric-quantitative-bounds}, and in
particular proves exact generation.

Conversely, suppose every hydrostatic Fourier mode is exactly
generated. In particular, the barotropic modes at
$\pm\boldsymbol e_1,\pm\boldsymbol e_2$ and a nonzero mode at positive
vertical level one are exactly generated. Since only finitely many
such modes are involved, they belong to one common finite quadratic
extension of $E_0$. The exactly generated barotropic frequencies
therefore span $\mathbb Z^2$, while the exactly generated positive
levels have greatest common divisor one. This is precisely
\Cref{def:fourier-certificate}.

\end{proof}

\subsubsection{Sharp rank counts for Fourier forcing}
\label{sec:minimal-fourier-noise}
In this subsection, we prove the sharp rank counts used in \Cref{sec:settings}. The general
spectral-gap theorem itself only requires the phase-complete and
quantitative quadratic saturation assumptions of \Cref{thm:main}. Recall the phase-complete finite Fourier forcing as in \Cref{def:phase-complete}, where 
\begin{align}\label{eq:phase-complete-E0}
 E_0=\operatorname{span}_{\R}\Bigl(\left\{\boldsymbol a_j(z)\cos(\boldsymbol k_j\cdot\boldsymbol x),\,
 \boldsymbol a_j(z)\sin(\boldsymbol k_j\cdot\boldsymbol x):j\in J\right\}\bigcup\left\{\boldsymbol a_\ell^0(z):\ell\in J_0\right\}\Bigr),
\end{align}
for some finite index sets $J_0,J$. We call a phase-complete finite Fourier range
\emph{single-level} if 
$\boldsymbol a_j(z)=\boldsymbol a_j\cos(n_jz)$ and
$\boldsymbol a_\ell^0(z)=\boldsymbol a_\ell^0\cos(n_\ell^0z)$, where
$n_j\in\mathbb N_0$ and $n_\ell^0\ge1$. At $n_j=0$ we require
$\boldsymbol k_j\cdot\boldsymbol a_j=0$ by the hydrostatic incompressibility. The rank-five example is
\begin{align}\label{eq:five-noise-profiles}
 b_1=\sigma_1\boldsymbol e_2\cos x_1,\quad
 b_2=\sigma_2\boldsymbol e_2\sin x_1,\quad 
 b_3=\sigma_3\boldsymbol e_1\cos x_2,\quad 
 b_4=\sigma_4\boldsymbol e_1\sin x_2,\quad
 b_5=\sigma_5\boldsymbol e_1\cos z,
\end{align}
where every $\sigma_j$ is nonzero.

The following proposition provides a sharp count of the rank  of $Q$ in \eqref{eq:noise-operator} in the case  of single-level class that ensures the quantitative quadratic saturation. 
\begin{proposition}
\label{prop:five-noises}
Among single-level phase-complete finite Fourier forcing ranges, the
smallest rank satisfying quantitative quadratic saturation is five. It
is attained by the profiles in \eqref{eq:five-noise-profiles}. 
\end{proposition}

\begin{proof}
Let $\Lambda_h=\mathrm{span}_{\mathbb Z}\{\boldsymbol k_j:j\in J\}$ be the subgroup of $\mathbb Z^2$ generated by the
nonzero horizontal Fourier frequencies of $E_0$. Then quantitative quadratic saturation requires
$\Lambda_h=\mathbb Z^2$. Therefore, at least two
linearly independent nonzero horizontal frequencies must occur. By phase completeness, $\operatorname{rank}Q\ge4$.

Suppose, for contradiction, that $\operatorname{rank}Q=4$. Then exactly
two nonzero horizontal character pairs occur, say
$\{\pm\boldsymbol k\}$ and $\{\pm\boldsymbol l\}$, each contributing
exactly two real dimensions. In particular, there is no horizontally
constant forcing direction. Since the forcing is single-level, the two
phase pairs have respective vertical levels $n,m\in\mathbb N_0$.
Moreover, $\boldsymbol k,\boldsymbol l$ form a $\mathbb Z$-basis of
$\mathbb Z^2$ as $\Lambda_h=\mathbb Z^2$. Consider the signed Fourier lattice
$\mathbb Z^2\times\mathbb Z$. Using the unique representation
$\boldsymbol q=a\boldsymbol k+b\boldsymbol l$, define
\begin{align*}
 \chi(\boldsymbol q,r)
 =r-an-bm\pmod2.
\end{align*}
Then $\chi:\mathbb Z^2\times\mathbb Z\to\mathbb Z/2\mathbb Z$ is a
group homomorphism. As 
\begin{align*}
 \{(\pm\boldsymbol k,\pm n),(\pm\boldsymbol l,\pm m)\}\subset \ker\chi,
\end{align*}
the full
Fourier support of $E_0$,  $\operatorname{supp}_{\mathrm F}E_0\subset\ker\chi$. It follows from \eqref{eq:two-channel-interaction} and \eqref{eq:equal-level-noise-generation} that $ \operatorname{supp}_{\mathrm F}\mathsf B(f,g)
\subset
\operatorname{supp}_{\mathrm F}f+
\operatorname{supp}_{\mathrm F}g$, induction in $j$
therefore gives
\begin{align}\label{eq:parity-obstruction}
 \operatorname{supp}_{\mathrm F}E_j\subset\ker\chi,
 \qquad j\ge0.
\end{align}
If the horizontal frequency vanishes, then $a=b=0$, and hence $\chi(0,r)=r\pmod2$. 
Thus \eqref{eq:parity-obstruction} excludes every pure vertical mode of
odd level, and hence quantitative quadratic saturation by \Cref{prop:geometric-generation}. Hence
$\operatorname{rank}Q\ge5$.

Note that the first four profiles in
\eqref{eq:five-noise-profiles} give exact barotropic modes at
$\pm\boldsymbol e_1$ and $\pm\boldsymbol e_2$, while $b_5$ gives an
exact pure mode at vertical level one. Thus \Cref{def:fourier-certificate} already holds at depth zero $j_*=0$, proving
quantitative quadratic saturation.
\end{proof}

\begin{corollary}
\label{cor:six-noises}
Among single-level phase-complete finite Fourier forcing ranges with no
horizontally constant profile, the smallest rank satisfying quantitative
quadratic saturation is six. It is attained by
$b_1,\ldots,b_4$ from \eqref{eq:five-noise-profiles}, together with
\begin{align*}
 b_5=\sigma_5\boldsymbol e_1\cos(x_1+x_2)\cos z,\qquad
 b_6=\sigma_6\boldsymbol e_1\sin(x_1+x_2)\cos z,
\end{align*}
where all $\sigma_j$ are nonzero.
\end{corollary}

\begin{proof}
A phase-complete finite Fourier range with no horizontally constant
profile has even real dimension. By \Cref{prop:five-noises}, rank four
is impossible, so the rank is at least six.

For the displayed family, the first four profiles give the barotropic
frequencies $\pm\boldsymbol e_1,\pm\boldsymbol e_2$, while the last two
give a nonzero pure Fourier mode at vertical level one. Hence the two
conditions in \Cref{def:fourier-certificate} hold already at depth zero. 
\end{proof}

We now allow a horizontal phase pair to carry a vertical profile
containing more than one Fourier level. Define
\begin{align}\label{eq:four-noise-profiles}
 f_1&=\boldsymbol e_2\cos x_1(1+\cos z),&
 f_2&=\boldsymbol e_2\sin x_1(1+\cos z),\notag\\
 f_3&=\boldsymbol e_1\cos x_2,&
 f_4&=\boldsymbol e_1\sin x_2.
\end{align}

The following proposition provides a sharp count of the rank of $Q$ in the case when mixed vertical profiles are allowed. 
\begin{proposition}
\label{prop:four-noises}
Among the phase-complete finite Fourier forcing ranges of
\Cref{def:phase-complete}, the smallest rank satisfying quantitative
quadratic saturation is four. It is attained by
$b_j=\sigma_jf_j$, $1\le j\le4$, with arbitrary fixed nonzero
amplitudes.
\end{proposition}

\begin{proof}
The horizontal-lattice argument in the proof of
\Cref{prop:five-noises} gives $\operatorname{rank}Q\ge4$. It remains to verify the four-direction example in \eqref{eq:four-noise-profiles}. For the verification it is convenient to work in the complexification. Initially,
\begin{align*}
 g=f_1+\mathrm{i}f_2=\phi_{\boldsymbol e_1,0,\boldsymbol e_2}
   +\phi_{\boldsymbol e_1,1,\boldsymbol e_2}\in E_0^{\mathbb C},\quad \text{ and } h_\pm=f_3\pm\mathrm{i}f_4=\phi_{\pm\boldsymbol e_2,0,\boldsymbol e_1}
 \in E_0^{\mathbb C}.
\end{align*}
The barotropic contribution in $\mathsf B(g,h_+)$ vanishes, and
\eqref{eq:barotropic-gcd-shift} gives
\begin{align}\label{eq:four-noise-separation}
 \mathsf B(g,h_+)
 =\mathrm{i}\phi_{(1,1),1,(1,1)},\quad \text{ and }
 \mathsf B(\phi_{(1,1),1,(1,1)},h_-)
 =\mathrm{i}\phi_{\boldsymbol e_1,1,\boldsymbol e_2}.
\end{align}
Consequently,
\begin{align*}
 -\mathsf B(\mathsf B(g,h_+),h_-)
 =\phi_{\boldsymbol e_1,1,\boldsymbol e_2},\quad \text{ and }
 g+\mathsf B(\mathsf B(g,h_+),h_-)
 =\phi_{\boldsymbol e_1,0,\boldsymbol e_2}.
\end{align*}
Both modes therefore belong to $E_2^{\mathbb C}$, and their conjugates
give the corresponding negative horizontal frequencies. Together with
$h_\pm$, the exactly generated barotropic frequencies contain
$\pm\boldsymbol e_1,\pm\boldsymbol e_2$, while the exactly generated
positive vertical levels contain level one. Thus
\Cref{def:fourier-certificate} holds at depth two, proving quantitative
quadratic saturation. 
\end{proof}

\begin{remark}
\label{rem:minimum-scope}
The rank-five and rank-six minima concern the single-level subclass.
Allowing one horizontal phase pair to contain several vertical Fourier
levels lowers the sharp minimum to four. These optimality statements
concern quadratic or exact Fourier generation within the stated forcing
classes; they are not claimed to be necessary conditions for every
possible spectral-gap mechanism.
\end{remark}

\subsection{Uniform irreducibility}\label{sec:exponential-mixing}

We verify \Cref{ass:gap-accessibility} by the scaling--saturation
framework of \cite{GHHM}, rooted in \cite{AgrachevSarychev2005,AgrachevSarychev2006}. Unlike the argument using stability of the origin, this approach is stable under the addition of an extra deterministic forcing in \eqref{eq:pe}.  Consider the controlled primitive equations 
\begin{align}\label{eq:pe-controlled}
 \partial_tu+\nu Au+B(u,u)=\sum_{k=1}^m b_k\eta_k(t),
\end{align}
with control $\eta=(\eta_1,\cdots,\eta_m)\in L^2([0,T];\mathbb R^m)$, whose solution map is denoted by $\Phi_{t}^{\eta}$. We set  $\mathsf B(E,E) = \{\mathsf B(f,g): f,g\in E \}$ and the translation operator $\mathcal T_f$ for $f\in H$  by $\mathcal T_fu=u+f$ for $u\in H$.

The next lemma gives the equation-specific scaling step. The sign
obstruction for even quadratic nonlinearities is standard in the
scaling method; see \cite[Remark~2.22 and Section~2.4.2]{GHHM}. For the
primitive equations it is resolved by the exact cancellation between
the two horizontal phases.

\begin{lemma}
\label{lem:phase-pair-control}
Let $s\ge 6$ and $E\subset H^s$ be a finite-dimensional phase-complete Fourier space as in \Cref{def:phase-complete} and set $E^+=\operatorname{span}(E\cup\mathsf B(E,E))$. Then $E^+$ is again phase-complete.  Moreover, for every smooth $g\in H$, every compact $K\subset H^s$, as $\delta\to0$, 
\begin{align}\label{eq:scaling-generation}
 \sup_{u\in K}\left\|\mathcal T_{-\delta^{-1/2}g}
 \circ\Phi_\delta^0
 \circ\mathcal T_{\delta^{-1/2}g}u-\mathcal T_{-B(g,g)}u\right\|_{H^s}\to 0. 
\end{align}
\end{lemma}

\begin{proof}
We first verify that $E^+$ is phase-complete. Write
$(\tau_{\boldsymbol y}u)(\boldsymbol x,z)
=u(\boldsymbol x+\boldsymbol y,z)$ and
$(\mathcal Ru)(\boldsymbol x,z)=u(-\boldsymbol x,z)$. Since
\begin{align*}
 B(\tau_{\boldsymbol y}g,\tau_{\boldsymbol y}h)
 =\tau_{\boldsymbol y}B(g,h),\quad 
 B(\mathcal Rg,\mathcal Rh)=-\mathcal RB(g,h),
\end{align*}
and $E$ is invariant under horizontal translations and reflection,
the same is true for $E^+$ which consists of trigonometric
polynomials, hence each of its
horizontal Fourier components then belongs to $E^+$. Consider a nonzero
frequency $\boldsymbol k$ and a corresponding real component
\begin{align*}
 f(\boldsymbol x,z)
 =\boldsymbol a(z)\cos(\boldsymbol k\cdot\boldsymbol x)
 +\boldsymbol c(z)\sin(\boldsymbol k\cdot\boldsymbol x)
 \in E^+.
\end{align*}
Then reflection separates the two phases, yielding $ \frac12(f+\mathcal Rf)
=\boldsymbol a(z)\cos(\boldsymbol k\cdot\boldsymbol x)\in E^+$ and 
\begin{align*}
 \frac12(f-\mathcal Rf)
 =\boldsymbol c(z)\sin(\boldsymbol k\cdot\boldsymbol x)\in E^+.
\end{align*}
Choose $\boldsymbol y\in\mathbb T^2$ such that
$\boldsymbol k\cdot\boldsymbol y=\pi/2$ modulo $2\pi$. Translation by
$\boldsymbol y$ then shows that
\begin{align*}
 \boldsymbol a(z)\sin(\boldsymbol k\cdot\boldsymbol x),\qquad
 \boldsymbol c(z)\cos(\boldsymbol k\cdot\boldsymbol x)
 \in E^+,
\end{align*}
up to irrelevant signs. Hence every vertical profile occurring at a
nonzero horizontal frequency occurs with both horizontal phases.
Horizontally constant profiles are preserved separately. Therefore
$E^+$ is phase-complete.

We next derive the quadratic scaling limit. Fix a smooth $g$, $u\in K$ and consider the auxiliary composition
\begin{align*}
 v_\delta(t)
 =
 \mathcal T_{-\delta^{-1/2}g}
 \circ\Phi_{\delta t}^0
 \circ\mathcal T_{\delta^{-1/2}g}u,
\end{align*}
which solves 
\begin{align*}
 \partial_tv_\delta
 =-B(g,g)-\delta^{1/2}\{\nu Ag+\mathsf B(g,v_\delta)\}-\delta\{\nu Av_\delta+B(v_\delta,v_\delta)\},
 \qquad v_\delta(0)=u.
\end{align*}
Set $\zeta_{\delta}(t)=e^{-\nu\delta tA}u-tB(g,g)$ and $e_\delta=v_\delta-\zeta_{\delta}$. Compactness of $K$, standard heat semigroup energy identity and  Poincar\'e's inequality  yield
\begin{align*}
 \sup_{u\in K}\sup_{0\le t\le1}\|\zeta_{\delta}(t)\|_{H^s}\le C_K,
 \qquad
 \sup_{u\in K}\delta\int_0^1\|\zeta_{\delta}(t)\|_{H^{s+1}}^2\,\dd t\to0.
\end{align*}
Here and below, $C_K$ may depend on $s,\nu,g$ and $\sup_{u\in K}\|u\|_{H^s}$, but not on $\delta\in(0,1]$.
Subtracting the equation for $\zeta_{\delta}$ from that for $v_\delta$, we obtain
\begin{align*}
 \partial_te_\delta+\nu\delta Ae_\delta
 +\delta^{1/2}\mathsf B(g,e_\delta)
 +\delta\mathsf B(\zeta_{\delta},e_\delta)+\delta B(e_\delta,e_\delta)
 &=\nu\delta tAB(g,g)\\
 &\quad-\delta^{1/2}\{\nu Ag+\mathsf B(g,\zeta_{\delta})\}
 -\delta B(\zeta_{\delta},\zeta_{\delta}),
\end{align*}
with $e_\delta(0)=0$. 
For $s\ge6$, the Sobolev product estimates give 
\begin{align*}
 \|B(f,h)\|_{H^{s-1}}&\le C_s\|f\|_{H^s}\|h\|_{H^s},\\
 \|B(f,h)\|_{H^s}&\le C_s\bigl(\|f\|_{H^{s+1}}\|h\|_{H^s}
 +\|f\|_{H^s}\|h\|_{H^{s+1}}\bigr).
\end{align*}
Using the first estimate for the terms containing $e_\delta$, the second for the remaining terms, and Young's inequality in the $H^s$ energy identity, we find
\begin{align*}
 \frac{d}{dt}\|e_\delta\|_{H^s}^2
 +\nu\delta\|A^{1/2}e_\delta\|_{H^s}^2
 \le C_K\|e_\delta\|_{H^s}^2
 +C_K\delta\|e_\delta\|_{H^s}^4
 +C_K\delta\bigl(1+\|\zeta_{\delta}\|_{H^{s+1}}^2\bigr).
\end{align*}
Up to the first time at which $\|e_\delta\|_{H^s}=1$, Gronwall's inequality implies
\begin{align*}
 \sup_{0\le t\le1}\|e_\delta(t)\|_{H^s}^2
 \le C_K\delta\left(1+\int_0^1\|\zeta_{\delta}(t)\|_{H^{s+1}}^2\,\dd t\right).
\end{align*}
The right-hand side tends to zero uniformly for $u\in K$, so this first-exit time exceeds $1$ for all sufficiently small $\delta$. Consequently,
\begin{align*}
 \sup_{u\in K}\sup_{0\le t\le1}
 \|v_\delta(t)-(u-tB(g,g))\|_{H^s}
 \le \sup_{u\in K}\sup_{0\le t\le1}\|e_\delta(t)\|_{H^s}+\sup_{u\in K}\sup_{0\le t\le1}\|e^{-\nu\delta tA}u-u\|_{H^s}
 \to0,
\end{align*}
which proves \eqref{eq:scaling-generation}.
\end{proof}

The following lemma upgrades the quadratic scaling above to translations in
arbitrary $H^s$ directions under the quantitative saturation condition.
\begin{lemma}\label{lem:short-time-translation}
Assume the forcing conditions of \Cref{thm:main} and $s\geq 6$. Then for every $f\in H^s$, compact $K\subset H^s$, and
$\varepsilon,\tau>0$, there are $0<t<\tau$ and $\eta\in L^2([0,t];\mathbb R^m)$ such that 
\begin{align*}
\sup_{u\in K}\left\|\Phi_{t}^{\eta}(u)-\mathcal T_{f}(u)\right\|_{H^s}<\varepsilon. 
\end{align*} 
\end{lemma}

\begin{proof}
Recall $E_0=\Ran Q$ and the quadratic extensions $E_n$ from
\eqref{eq:quadratic-saturation}. We first prove by induction that, for
every $n\ge0$, $f\in E_n$, and  $\varepsilon,\tau>0$, there are $0<t<\tau$ and
$\eta\in L^2([0,t];\mathbb R^m)$ such that
\begin{align}\label{eq:small-time-En-translation}
 \sup_{u\in K}
 \|\Phi_t^\eta(u)-\mathcal T_f(u)\|_{H^s}<\varepsilon.
\end{align}
For $n=0$, write $f=Q\alpha$. Given $\delta>0$, take the constant
control $\eta_\delta=\delta^{-1}\alpha$ on $[0,\delta]$ and set
$v_\delta(r)=\Phi_{\delta r}^{\eta_\delta}(u)$, $0\le r\le1$. Then
\begin{align*}
 \partial_rv_\delta
 =f-\delta\{\nu Av_\delta+B(v_\delta,v_\delta)\},
 \qquad v_\delta(0)=u.
\end{align*}
The usual short-time estimate gives
$v_\delta(1)\to u+f$ in $H^s$, uniformly for $u\in K$. Taking
$\delta<\tau$ sufficiently small proves
\eqref{eq:small-time-En-translation} for $E_0$.

Assume now that \eqref{eq:small-time-En-translation} holds for $E_n$.
By repeated application of \Cref{lem:phase-pair-control}, $E_n$ is
phase-complete. For $g\in E_n$, the scaling limit
\eqref{eq:scaling-generation} gives
\begin{align*}
 \mathcal T_{-\delta^{-1/2}g}\circ\Phi_\delta^0
 \circ\mathcal T_{\delta^{-1/2}g}
 \longrightarrow\mathcal T_{-B(g,g)}
\end{align*}
uniformly on compact subsets of $H^s$. For each fixed $\delta$, the induction
hypothesis realizes the two translations
$\mathcal T_{\pm\delta^{-1/2}g}$ by genuine controlled endpoints in
arbitrarily small time. The compact-uniform transitivity of the
controlled flows \cite[Section~3]{GHHM} therefore shows that the
translation by $f=-B(g,g)$ satisfies
\eqref{eq:small-time-En-translation}.

Choose a phase-complete spanning family of $E_n$. For every horizontal
phase pair
\begin{align*}
 e^c(\boldsymbol x,z)
 =\boldsymbol a(z)\cos(\boldsymbol k\cdot\boldsymbol x),\qquad
 e^s(\boldsymbol x,z)
 =\boldsymbol a(z)\sin(\boldsymbol k\cdot\boldsymbol x),
\end{align*}
direct computation gives
$B(e^s,e^s)=-B(e^c,e^c)$. Hence both signs of $B(e,e)$ are available
for every member $e$ of the spanning family; for horizontally constant
members one has $B(e,e)=0$. For two members $e,f$, polarization gives
\begin{align*}
 -\mathsf B(e,f)
 &=-B(e+f,e+f)+B(e,e)+B(f,f),\\
 \mathsf B(e,f)
 &=-B(e-f,e-f)+B(e,e)+B(f,f).
\end{align*}
Thus both signs of every direction in $\mathsf B(E_n,E_n)$ satisfy the
same small-time translation property \eqref{eq:small-time-En-translation}. Finite composition and linearity
then give \eqref{eq:small-time-En-translation} for
\begin{align*}
 E_{n+1}=\operatorname{span}(E_n\cup\mathsf B(E_n,E_n)),
\end{align*}
which completes the induction.

By \Cref{prop:geometric-generation}, quantitative quadratic saturation
implies that $\bigcup_{n\ge0}E_n$ is dense in
$H^s$. Given $f\in H^s$, choose $f_N\in E_n$ for some $n$ such that
$\|f-f_N\|_{H^s}<\varepsilon/2$. Applying
\eqref{eq:small-time-En-translation} to $f_N$ with accuracy
$\varepsilon/2$ yields
\begin{align*}
 \sup_{u\in K}
 \|\Phi_t^\eta(u)-\mathcal T_f(u)\|_{H^s}
 \le
 \sup_{u\in K}
 \|\Phi_t^\eta(u)-\mathcal T_{f_N}(u)\|_{H^s}
 +\|f_N-f\|_{H^s}
 <\varepsilon.
\end{align*}
\end{proof}

Recall the  regular phase space $\mathscr E$ in \eqref{eq:regular-phase-space}. We are now ready to prove the uniform irreducibility required by \Cref{ass:gap-accessibility}. 
\begin{proposition}
\label{prop:mixing-accessibility}
Keep the forcing conditions of \Cref{thm:main}. Then, for every
$R<\infty$, $v_*\in\mathscr E$, and every neighborhood
$\mathcal O$ of $v_*$ in $\mathscr E$, there is $p>0$ such that
\begin{align}\label{eq:pe-accessibility}
 \inf_{U(u)\le R}P_2(u,\mathcal O)\ge p.
\end{align}
\end{proposition}

\begin{proof}
Let $s\ge6$. Using the continuous embedding
$H^s\hookrightarrow\mathscr E$, choose a smooth state $\hat v$ so close
to $v_*$ in $\mathscr E$ that
$B_{H^s}(\hat v,3\varepsilon)\subset\mathcal O$ for some
$\varepsilon>0$. Choose a compact set
$\Gamma_0\subset C_0([0,1];\mathbb R^m)$ of positive Wiener measure. Since $U(u)\le R$ bounds
$K(u)$ by \eqref{eq:mixing-lyapunov} and \eqref{eq:K}, \Cref{lem:bounded-driver-regularization} places all prefix
endpoints $\Phi_1(u,\omega_0)$ for $U(u)\le R$ and $\omega_0\in\Gamma_0$
in a bounded analytic set. After decreasing the analytic radius and
taking the closure in $H^s$, we obtain a compact set
$\mathscr D_R\subset H^s$ containing all such endpoints.

We next construct, for each $v\in\mathscr D_R$, a control on the fixed
time interval $[0,1]$ which steers $v$ into
$B_{H^s}(\hat v,\varepsilon)$. Let
$v_r=\Phi_r^0(v)$ and set
$f=\hat v-v_1$. By continuity of the uncontrolled trajectory, there is 
$\tau\in(0,1)$ such that $\|v_r-v_1\|_{H^s}<\varepsilon/2$ for  $1-\tau\le r\le1$. 
The set
\begin{align*}
 K_v=\{v_r:1-\tau\le r\le1\}
\end{align*}
is compact in $H^s$. By
\Cref{lem:short-time-translation}  we obtain
$0<t<\tau$ and a control
$\eta_0\in L^2([0,t];\mathbb R^m)$ such that
\begin{align*}
 \sup_{u\in K_v}
 \|\Phi_{t}^{\eta_0}(u)-\mathcal T_{f}u\|_{H^s}
 <\frac{\varepsilon}{2}.
\end{align*}
Since $v_{1-t}\in K_v$, one has 
\begin{align*}
 \|\Phi_{t}^{\eta_0}(v_{1-t})-\hat v\|_{H^s}
 \le
 \|\Phi_{t}^{\eta_0}(v_{1-t})
      -(v_{1-t}+f)\|_{H^s}
 +\|v_{1-t}-v_1\|_{H^s}<\varepsilon.
\end{align*}
Defining the control $\bar\eta$ on $[0,1]$ by $0$ on $(0,1-t)$ and $\eta_0(r-1+t)$ on $(1-t,1)$, then 
\begin{align}\label{eq:controlled-target}
 \|\Phi_1^{\bar\eta}(v)-\hat v\|_{H^s}<\varepsilon.
\end{align}
Since the Wiener measure has full support and the Markov semigroup is Feller on $H^s$, there is $p_0>0$ such that 
\begin{align*}
\inf_{v\in \mathscr D_R} P_{1}(v,B_{H^s}(\hat v, \varepsilon))\geq p_0.
\end{align*}
By the Chapman--Kolmogorov equation, we have 
\begin{align*}
\inf_{U(u)\le R}P_2(u,\mathcal O) = \inf_{U(u)\le R}\int_{H^s}P_1(v,\mathcal O)P_1(u,dv)\geq p_0 \inf_{U(u)\le R} P_1(u,\mathscr D_R)\geq \gamma_1(\Gamma_0)p_0>0,
\end{align*}
where $\gamma_1$ is the Wiener measure on $C_0\left([0,1];\mathbb R^m\right)$. 
\end{proof}

\subsection{Proof of the main results}\label{sec:assembly}

We now combine the abstract criterion with the established estimates  to complete the proof of \Cref{thm:main} and \Cref{thm:exponential-mixing}.

\subsubsection{Proof of \Cref{thm:main}}
\begin{proof}
It suffices to verify the hypotheses of \Cref{thm:structural-gap}. 
Fix $\eta\in(0,\eta_0]$, with $\eta_0$ from
\Cref{lem:admissibility}, set $\kappa=C_\nu\eta$, and work on the
Polish state space $\mathcal X=\mathscr E$ as in \eqref{eq:regular-phase-space}. By
\Cref{lem:mixing-foster}, $P_t(u,\mathscr E)=1$ for
$u\in H^1$, which together with \Cref{prop:pe-basics} provides an invariant
probability on $H^1$ concentrated on $\mathscr E$. Thus $(P_t)$
restricts to a Markov semigroup on $\mathscr E$ with an invariant
probability. Take
\begin{align*}
 d(u,v)=1\wedge\|u-v\|_H
\end{align*}
as comparison metric and use the block kernel $P=P_2$. Then
for the Lyapunov function 
\begin{align*}
U(u)=\left(1+\alpha\|u\|_8^8+\|\partial_zu\|_4^4+\|\nabla u\|_2^2\right)^\eta\exp\left(C_{\nu}\eta\|u\|_H^2\right)
\end{align*}
as in \eqref{eq:mixing-lyapunov} with $\alpha$ determined by \Cref{lem:cascade-direct}, 
\Cref{ass:gap-drift} follows from \Cref{lem:mixing-foster}, and \Cref{ass:gap-amplification} follows from \Cref{lem:one-sided-growth} and \Cref{lem:mixing-conditional-clock}, particularly \eqref{eq:eht}, with 
\begin{align*}
H_T=C_\nu\int_0^TK(v_s)\,\dd s,\quad K(v) = 1+\alpha\|v\|_8^8+\|\partial_zv\|_4^4+\|\nabla v\|_2^2, 
\end{align*}
while  \Cref{prop:mixing-accessibility} gives
\Cref{ass:gap-accessibility}.

It remains to verify \Cref{ass:gap-local-repair} by combining the regularization property of the primitive equations with \Cref{prop:compact-source-repair}. Fix
\begin{align*}
 R>\frac{4b}{1-\lambda}.
\end{align*}
Choose a compact set
$\Gamma_0\subset C_0([0,1];\R^m)$ of positive Wiener measure. Since $U(u)\le R$ bounds $K(u)$,
\Cref{lem:bounded-driver-regularization} implies that
\begin{align*}
 \Phi_1(u,\omega_0),\qquad
 U(u)\le R,\quad \omega_0\in\Gamma_0,
\end{align*}
belong to a compact analytic set $\mathscr D_R\subset H^s$ for a fixed $s\ge6$, upon decreasing the analytic
radius and taking the closure in $H^s$.  Moreover,
\Cref{lem:one-sided-growth} and \Cref{lem:bounded-driver-regularization} give $L_R<\infty$ such that
\begin{align}\label{eq:pe-prefix-interface}
 \|\Phi_1(v,\omega_0)-\Phi_1(u,\omega_0)\|_H
 \le L_R\|v-u\|_H
\end{align}
whenever $U(u)\le R$, $\omega_0\in\Gamma_0$, and $v$ is sufficiently
close to $u$. We next apply \Cref{prop:polynomial-source} on the independent unit
suffix with initial set $\mathscr D_R$. It gives a compact set
$\Gamma_1\subset C_0([0,1];\R^m)$ of positive Wiener measure and constants
$C,b_*,\delta_0>0$ such that, uniformly on
$\mathscr D_R\times\Gamma_1$,
\begin{align}\label{eq:pe-source-interface}
 JJ^*\le C_\delta^2AA^*+\delta^2I,
 \qquad
 C_\delta\le C\delta^{-b_*},
 \qquad 0<\delta<\delta_0.
\end{align}
The endpoint remainder and continuous-driver estimates in
\Cref{lem:endpoint-remainder,lem:driver-chart} verify the regular
dependence required in \Cref{prop:compact-source-repair} on the same
fixed neighborhood of $\mathscr D_R\times\Gamma_1$. Hence
\Cref{prop:compact-source-repair} applies to the unit suffix.

We now compose this compact-core repair with the prefix regularization.
Let $u,v\in\mathscr E$ satisfy
\begin{align*}
 \min\{U(u),U(v)\}\le R,\qquad
 r=d(u,v)\in(0,r_*],
\end{align*}
and choose the coordinate with smaller $U$ as the reference, say $u_t=\Phi_{t}(u,\omega)$. Use the
same Brownian increment on the first unit interval. On the event
$\{\omega_0\in\Gamma_0\}$, write
\begin{align*}
 u_1=\Phi_1(u,\omega_0),\qquad
 v_1=\Phi_1(v,\omega_0).
\end{align*}
Then $u_1\in \mathscr D_R$ and, by \eqref{eq:pe-prefix-interface},
\begin{align}\label{eq:pe-prepared-separation}
 r_1:=\|u_1-v_1\|_H\le L_Rr.
\end{align}
Conditionally on the prefix, apply the exact suffix coupling supplied
by \Cref{prop:compact-source-repair}. Outside
$\{\omega_0\in\Gamma_0\}$ use the synchronous suffix coupling. Since
both suffix constructions have the correct conditional marginals, their
composition gives a jointly Borel exact coupling $(X',Y')$ of
$P_2(u,\cdot)$ and $P_2(v,\cdot)$. For $b_*$ from \eqref{eq:pe-source-interface}, fix
\begin{align*}
 0<a<(1+2b_*)^{-1},\qquad
 \varkappa=1-a(1+2b_*)>0.
\end{align*}
By \Cref{prop:compact-source-repair} and
\eqref{eq:pe-prepared-separation}, with $p_0=\gamma_1(\Gamma_0)\gamma_1(\Gamma_1)>0$,  where $\gamma_1$ is the Wiener measure on $C_0([0,1];\mathbb R^m)$, one has 
\begin{align*}
 \PP\{d(X',Y')\le C(L_Rr)^{1+a}\}
 \ge p_0-C(L_Rr)^\varkappa. 
\end{align*}
After changing constants and decreasing $r_*$, one has 
\begin{align*}
 C(L_Rr)^{1+a}\le r^{1+a/2},
 \qquad
 C(L_Rr)^\varkappa\le p_0/2,
\end{align*}
and therefore
\begin{align}\label{eq:pe-local-superlinear-repair}
 \PP\{d(X',Y')\le r^{1+a/2}\}\ge p_0/2.
\end{align}

It remains to verify the second estimate \eqref{eq:gap-repair-input-2} in
\Cref{ass:gap-local-repair}. On $\{\omega_0\in\Gamma_0\}$, conditionally
on the prefix, the second estimate in \Cref{prop:compact-source-repair}
gives
\begin{align*}
 \PP\left\{d(X',Y')>
 e^{C+C_\nu\int_1^2K(u_s)\dd s}r_1\,\middle|\,\mathscr F_1\right\}
 \le Cr_1^\varkappa,
\end{align*}
where $u_s$ denotes the reference trajectory. By the synchronous
prefix estimate \eqref{eq:global-difference},
\begin{align*}
 r_1\le e^{C_\nu\int_0^1K(u_s)\dd s}r,
\end{align*}
and hence, with $H_2=C_\nu\int_0^2K(u_s)\dd s$ and
\eqref{eq:pe-prepared-separation},
\begin{align*}
 \PP\left\{d(X',Y')>e^{C+H_2}r\,\middle|\,\mathscr F_1\right\}
 \le Cr_1^\varkappa\le CL_R^\varkappa r^\varkappa.
\end{align*}
On $\{\omega_0\notin\Gamma_0\}$ the suffix coupling is synchronous, so
$d(X',Y')\le e^{H_2}r$. Therefore, after changing constants,
\begin{align*}
 \PP\{d(X',Y')>e^{C_*+H_2}r\}\le C_*r^\varkappa.
\end{align*}
Together with \eqref{eq:pe-local-superlinear-repair}, this verifies
\Cref{ass:gap-local-repair} for $P_2$, with $a_*=a/2$.

Therefore
\Cref{thm:structural-gap} applies. Its explicit cost is exactly
\eqref{eq:main-explicit-cost}, and we obtain
\eqref{eq:main-explicit-gap}, uniqueness on $\mathscr E$, and
$\mu(U)<\infty$. By
\Cref{lem:mixing-foster} and invariance, every invariant  measure on $H^1$ is
concentrated on $\mathscr E$, so uniqueness holds on the full strong phase
space $H^1$. To prove that $\mu$ has full support in $H^1$, let
$\mathcal O\subset H^1$ be nonempty and open. Choose
$v_*\in\mathscr E\cap\mathcal O$ and an $\mathscr E$-neighborhood
$\mathcal O_*\subset\mathcal O$ of $v_*$. Since $\mu(U)<\infty$, there
is $R_0<\infty$ such that
\begin{align*}
 \mu\{U\le R_0\}>0.
\end{align*}
By invariance and \Cref{prop:mixing-accessibility},
\begin{align*}
 \mu(\mathcal O)
 \ge\int_{\{U\le R_0\}}P_2(u,\mathcal O_*)\,\mu(\dd u)
 \ge p\,\mu\{U\le R_0\}>0.
\end{align*}
Thus $\operatorname{supp}\mu=H^1$.
\end{proof}

\subsubsection{Proof of \Cref{thm:exponential-mixing}}

\begin{proof}
Fix $\eta\in(0,\eta_0]$ and set $\vartheta=\eta/(2+\eta)$. Fixing $q\in(0,\vartheta)$, since $q<\eta=\theta$,
\Cref{thm:structural-gap} gives \eqref{eq:main-explicit-gap}. Since $d\le d_1$, one has  $d_q\le d_{q,1}$.

\emph{Step 1: smoothing and fractional moments.}
Let $u,v\in\mathscr E$ and consider the solutions $u_t=\Phi_t(u,W)$ and $v_t=\Phi_t(v,W)$ on $[0,1]$. Set $\zeta_t=u_t-v_t$ and $a_t=(u_t+v_t)/2$. Then
\begin{align*}
 \partial_t\zeta+\nu A\zeta+B(a,\zeta)+B(\zeta,a)=0.
\end{align*}
Write $M_u=\sup_{0\le s\le1}K(u_s)$ and $R_u=\int_0^1D(u_s)/K(u_s)\,\dd s$, and define $M_v,R_v$ analogously. Testing with $\zeta$, using the cancellation of $B(a,\zeta)$ and \Cref{lem:one-sided-growth} gives
\begin{align*}
 \sup_{0\le s\le1}\|\zeta_s\|_H^2+\int_0^1\|\zeta_s\|_{H^1}^2\,\dd s
 \le Ce^{C(M_u+M_v)}\|u-v\|_H^2,
\end{align*}
where we used $K(a_s)\le C\{K(u_s)+K(v_s)\}$. For the $H^1$ estimate, the spatial bilinear estimate
\begin{align*}
 \|B(a,\zeta)+B(\zeta,a)\|_H\le C\|a\|_{H^{3/2}}\|\zeta\|_{H^{3/2}}
\end{align*}
with interpolation  and Young's inequality yield
\begin{align*}
 \bigl|\langle B(a,\zeta)+B(\zeta,a),A\zeta\rangle_H\bigr|
 \le C\|a\|_{H^{3/2}}\|\zeta\|_{H^1}^{1/2}\|\zeta\|_{H^2}^{3/2}\le \frac{\nu}{2}\|A\zeta\|_H^2+C\|a\|_{H^{3/2}}^4\|\zeta\|_{H^1}^2.
\end{align*}
Testing the difference equation with $A\zeta$ therefore gives
\begin{align*}
 \frac{d}{dt}\|\nabla\zeta\|_2^2+\nu\|A\zeta\|_H^2
 \le C\|a\|_{H^{3/2}}^4\|\zeta\|_{H^1}^2.
\end{align*}
By Poincar\'e's and Gronwall's inequalities,
\begin{align*}
 \|\zeta_t\|_{H^1}^2
 \le C\exp\left\{C\int_s^t\|a_r\|_{H^{3/2}}^4\,\dd r\right\}\|\zeta_s\|_{H^1}^2,
 \qquad 0\le s\le t\le1.
\end{align*}
Taking $t=1$, integrating in $s$, and using the preceding $H$ energy estimate bounds $\|\zeta_1\|_{H^1}$ in terms of $\|u-v\|_H$. Taking $s=0$ instead gives a uniform bound in terms of $\|u-v\|_{H^1}$. Moreover, interpolation and the definitions of $K,D$ in \eqref{eq:K} and \eqref{eq:D} imply
\begin{align*}
 \int_0^1\|u_s\|_{H^{3/2}}^4\,\dd s
 \le C\sup_{s\le1}\|u_s\|_{H^1}^2\int_0^1\|u_s\|_{H^2}^2\,\dd s
 \le CM_u^2R_u,
\end{align*}
and likewise for $v$. Since $\|a_s\|_{H^{3/2}}^4\le C(\|u_s\|_{H^{3/2}}^4+\|v_s\|_{H^{3/2}}^4)$, we obtain
\begin{align}
 \|u_1-v_1\|_{H^1}&\le Ce^G\|u-v\|_H,\label{eq:H-H1-one-step}\\
 \sup_{0\le t\le1}\|u_t-v_t\|_{H^1}&\le Ce^G\|u-v\|_{H^1},\label{eq:H1-short-time-amplification}
\end{align}
where $G=C(M_u+M_v+M_u^2R_u+M_v^2R_v)$.

Since $K^\eta\le U$ and $U\ge1$, \eqref{eq:maximal-U} and the $p=1$ case of \eqref{eq:mixing-foster} give
\begin{align*}
 \mathbb E M_u^\eta+\mathbb E R_u
 \le \mathbb E\sup_{s\le1}U(u_s)
 +\mathbb E\int_0^1U(u_s)\frac{D(u_s)}{K(u_s)}\,\dd s
 \le CU(u).
\end{align*}
Weighted Young's inequality yields
\begin{align*}
 (M_u^2R_u)^\vartheta
 =(M_u^\eta)^{2/(2+\eta)}R_u^{\eta/(2+\eta)}
 \le \frac{2}{2+\eta}M_u^\eta+\frac{\eta}{2+\eta}R_u.
\end{align*}
Since $\vartheta<\min\{\eta,1\}$ and $M_u,M_v\ge1$, the corresponding estimates for $v$, together with \eqref{eq:maximal-U}, imply
\begin{align}\label{eq:H1-smoothing-clock}
 \mathbb E G^\vartheta+\mathbb E\sup_{0\le t\le1}V_U(u_t,v_t)\le CV_U(u,v).
\end{align}

\emph{Step 2: lifting the weighted cost and proving the spectral gap.}
Suppose $r=d(u,v)>0$ and put $s=\log(1/r)$. By \eqref{eq:H-H1-one-step}, $d_1(u_1,v_1)\le1\wedge Ce^Gr$. For all sufficiently large $s$, on $\{G\le s/4\}$ we have $d_1(u_1,v_1)\le e^{-s/2}$, and hence
\begin{align*}
 d_{q,1}(u_1,v_1)\le \frac{CV_U(u_1,v_1)}{(1+s)^{q/2}}.
\end{align*}
On the complementary event, use $d_{q,1}(u_1,v_1)\le\sqrt{V_U(u_1,v_1)}$. Splitting the expectation and applying Cauchy--Schwarz, Markov's inequality, and \eqref{eq:H1-smoothing-clock}, we obtain
\begin{align*}
 \mathbb E d_{q,1}(u_1,v_1)
 &\le \frac{C\mathbb E V_U(u_1,v_1)}{(1+s)^{q/2}}
 +\bigl(\mathbb E V_U(u_1,v_1)\bigr)^{1/2}\mathbb P(G>s/4)^{1/2}\\
 &\le CV_U(u,v)\bigl((1+s)^{-q/2}+s^{-\vartheta/2}\bigr)
 \le \frac{CV_U(u,v)}{(1+s)^{q/2}},
\end{align*}
where the last inequality uses $q<\vartheta$. On the other hand, Jensen's inequality and \eqref{eq:H1-smoothing-clock} give
\begin{align*}
 \mathbb E d_{q,1}(u_1,v_1)
 \le \bigl(\mathbb E V_U(u_1,v_1)\bigr)^{1/2}
 \le C\sqrt{V_U(u,v)}.
\end{align*}
Taking the minimum of these bounds proves $\mathbb E d_{q,1}(u_1,v_1)\le Cd_q(u,v)$ for sufficiently large $s$. For bounded $s$, the definition gives $d_q(u,v)\ge c\sqrt{V_U(u,v)}$, so the same conclusion follows from the last estimate. If $r=0$, the synchronous solutions coincide. Thus
\begin{align*}
 \mathbb E d_{q,1}(u_1,v_1)\le Cd_q(u,v),\qquad u,v\in\mathscr E.
\end{align*}
To control short times, use \eqref{eq:H1-short-time-amplification} and repeat the same argument with $r=d_1(u,v)$ and $u_t,v_t$ in place of $u_1,v_1$. The moment bounds in \eqref{eq:H1-smoothing-clock} are uniform for $0\le t\le1$, so
\begin{align*}
 \sup_{0\le t\le1}\mathbb E d_{q,1}(u_t,v_t)\le Cd_{q,1}(u,v).
\end{align*}
Integrating these synchronous couplings against any coupling of $\mu_1,\mu_2\in\mathcal P(\mathscr E)$ and taking the infimum gives
\begin{align}
 \mathcal W_{d_{q,1}}(\mu_1P_1,\mu_2P_1)
 &\le C\mathcal W_{d_q}(\mu_1,\mu_2),\label{eq:H1-measure-lift}\\
 \sup_{0\le t\le1}\mathcal W_{d_{q,1}}(\mu_1P_t,\mu_2P_t)
 &\le C\mathcal W_{d_{q,1}}(\mu_1,\mu_2).\label{eq:H1-short-time-bound}
\end{align}

For $t\ge1$, \eqref{eq:H1-measure-lift}, \eqref{eq:main-explicit-gap}, and $d_q\le d_{q,1}$ yield
\begin{align*}
 \mathcal W_{d_{q,1}}(\mu_1P_t,\mu_2P_t)
 &\le C\mathcal W_{d_q}(\mu_1P_{t-1},\mu_2P_{t-1})\\
 &\le Ce^{-\gamma(t-1)}\mathcal W_{d_q}(\mu_1,\mu_2)
 \le Ce^{-\gamma(t-1)}\mathcal W_{d_{q,1}}(\mu_1,\mu_2).
\end{align*}
Together with \eqref{eq:H1-short-time-bound}, this proves \eqref{eq:main-explicit-ga-H1} after adjusting the constants.

\emph{Step 3: exponential mixing in $H^1$.}
Since $0<q<1$ and $r[1+\log(1/r)]^{q/2}\le1$ for $0<r\le1$, one has  $d_1\le d_{q,1}$ on $\mathscr E$. Fix $u\in H^1$. By \eqref{eq:mixing-coming-down}, $P_1(u,\mathscr E)=1$ and $P_1U(u)\le Ce^{\mathsf k\|u\|_H^2}$, while $\mu(\mathscr E)=1$ and $\mu(U)<\infty$. Since $d_{q,1}(v,w)\le\sqrt{V_U(v,w)}$, the product coupling and Jensen's inequality give
\begin{align*}
 \mathcal W_{d_{q,1}}(P_1(u,\cdot),\mu)
 \le \bigl(1+P_1U(u)+\mu(U)\bigr)^{1/2}
 \le C\exp\left(\kappa \|u\|_H^2\right)
\end{align*}
with $\kappa =\mathsf k/2$. 
Consequently, invariance of $\mu$ and \eqref{eq:main-explicit-ga-H1} imply, for $t\ge1$,
\begin{align*}
 \mathcal W_{d_1}(P_t(u,\cdot),\mu)
 \le \mathcal W_{d_{q,1}}\bigl(P_1(u,\cdot)P_{t-1},\mu P_{t-1}\bigr)\le \widetilde C e^{-\widetilde\gamma t}
 \mathcal W_{d_{q,1}}(P_1(u,\cdot),\mu)
 \le C\exp\{\kappa\|u\|_H^2-ct\}.
\end{align*}
Since $d_1\le1$, increasing $C$ covers $0\le t\le1$ and proves \eqref{eq:exponential-mixing-wp}.
\end{proof}

\appendix

\section{Estimates for the primitive equation}\label{sec:appendix}

This appendix collects the supporting analytic estimates used in \Cref{sec:pe-gap}. We first recall the global well-posedness and the existence of invariant measures. Recall that $\Omega_T=C_0([0,T];\mathbb R^m)$. 

\begin{proposition}\label{prop:pe-basics}
For every smooth finite-rank noise operator $Q$ as in \eqref{eq:noise-operator} and every initial state in $H^1$, system \eqref{eq:pe} has a unique global strong solution and defines a Feller Markov semigroup on $H^1$ possessing at least one invariant measure.  For each $t\ge0$, the solution map $(u,\omega)\mapsto\Phi_t(u,\omega)$ from $H^1\times \Omega_T$ to $H^1$ is Borel measurable.
\end{proposition}

The global well-posedness, Feller property, and existence of invariant measures follow from \cite[Theorems 1.3, 1.5, and 1.6]{GHKVZ}.  The Borel realization follows from the shifted deterministic formulation obtained after subtracting the additive noise path.

\subsection{Analyticity and uniform estimates}\label{app:analytic-core}

Recall the analytic Sobolev norm  $\|v\|_{\rho,s}=\|e^{\rho\Lambda}v\|_{H^s}$ with $\Lambda=(-\Delta)^{1/2}$ as in \eqref{eq:analytic-space}. The following spatial regularization fact will be used both in the observability estimate \Cref{prop:polynomial-source} and in the preparation of its reference set when combined with \Cref{prop:compact-source-repair} in verifying \Cref{ass:gap-local-repair} in \Cref{sec:assembly}. Recall the functions $K,D$   defined in \eqref{eq:K} and \eqref{eq:D}. For  $v_t=\Phi_t(u,\omega)$ we consider the shifted process
\begin{align}\label{eq:vflat}
v^{\flat}=v-\zeta, \quad \text{ where } \zeta=Q\omega.
\end{align}

\begin{lemma}\label{lem:bounded-driver-regularization}
Assume that the noise profiles $\{b_j\}$ in \eqref{eq:pe} are analytic. Given $R,M,T,s<\infty$ and $0<\tau<T$, there are $C,\rho>0$ such that, for every $u\in H$ with $K(u)\le R$ and every $\omega\in \Omega_T$ with $\|\omega\|_{\Omega_T}\le M$,  one has
\begin{align*}
 \sup_{t\le T}K(v^{\flat}_t)+\int_0^TD(v^{\flat}_t)\,\dd t\le C,
 \qquad \text{and }
 \sup_{\tau\le t\le T}\|v_t\|_{\rho,s}\le C. 
\end{align*}
\end{lemma}
\begin{proof}
Every prescribed spatial norm of $\zeta$ is bounded in terms of $M$ and $\{b_j\}$. Note $v^\flat$ satisfies
\begin{align}\label{eq:vflatPDE}
 \partial_t v^\flat+\nu A v^\flat+B(v^\flat+\zeta,v^\flat+\zeta)
 =-\nu A\zeta,\qquad v^\flat(0)=u.
\end{align}
The proof is given in two steps.

\emph{Step 1: Energy and cascade bounds.}
The cancellations in $B(v^\flat,v^\flat)$ and $B(\zeta,v^\flat)$ give
\begin{align*}
 \frac{d}{dt}\|v^\flat\|_H^2+\nu\|v^\flat\|_{H^1}^2
 \le C_M(1+\|v^\flat\|_H^2).
\end{align*}
For the diagnostic cross term use $\|w(v^\flat)\|_2\le C\|v^\flat\|_{H^1}$ and $\sup_{t\le T}\|\partial_z\zeta(t)\|_\infty\le C_M$. Hence
\begin{align}\label{eq:energy-vflat}
 \sup_{t\le T}\|v_t^\flat\|_H^2+\int_0^T\|v_t^\flat\|_{H^1}^2\,\dd t
 \le C_{R,M,T}.
\end{align}
The same three tests as in \Cref{lem:cascade-direct} give
\begin{align}\label{eq:Kvsharp}
 \frac{d}{dt}K(v^\flat)+cD(v^\flat)
 \le C_M(1+\|v^\flat\|_{H^1}^2)K(v^\flat).
\end{align}
We estimate the additional terms introduced by $\zeta$. The transport $B(\zeta,v^\flat)$ still cancels in the $L^8$ test. Using the zero-mean vertical primitive $\mathscr I_z$, integration by parts gives
\begin{align*}
 \left|\int w(v^\flat)\partial_z\zeta\cdot|v^\flat|^6v^\flat\,\dd \boldsymbol x\,\dd z\right|
 &\le C_M\left\{\|v^\flat\|_8^8
 +\|v^\flat\|_8^4
 \left(\int|v^\flat|^6|\nabla v^\flat|^2\,\dd \boldsymbol x\,\dd z\right)^{1/2}\right\}\\
 &\le\varepsilon\int|v^\flat|^6|\nabla v^\flat|^2\,\dd \boldsymbol x\,\dd z
 +C_{M,\varepsilon}\|v^\flat\|_8^8.
\end{align*}
Indeed $w(v^\flat)=-\operatorname{div}_h\mathscr I_zv^\flat$ and $\|\mathscr I_zv^\flat\|_8\le C\|v^\flat\|_8$. Writing $p=p^\flat+p^{\rm rem}$ according to
$v=v^\flat+\zeta$, the pressure generated by $v^\flat$ is estimated
exactly as in \eqref{eq:pressure-L8}. The remaining pressure satisfies
\begin{align*}
 -2\pi\Delta_hp^{\rm rem}
 =
 \nabla_h\cdot\nabla_h\cdot
 \int_{\mathbb T}
 \bigl(v^\flat\otimes\zeta+\zeta\otimes v^\flat
       +\zeta\otimes\zeta\bigr)\,\dd z.
\end{align*}
Since the prescribed spatial norms of $\zeta$ are bounded by $C_M$,
the same Calder\'on--Zygmund argument as in
\eqref{eq:pressure-L8} gives
\begin{align*}
 \|\nabla_hp^{\rm rem}\|_{L_h^{8/5}}
 \le C_M(1+\|v^\flat\|_{H^1}).
\end{align*}
Pairing this with
$\int_{\mathbb T}|v^\flat|^6v^\flat\,\dd z$ and using
\eqref{eq:pressure-L8-2} gives
\begin{align*}
 \left|\int \nabla_hp^{\rm rem}\cdot |v^\flat|^6v^\flat\dd \boldsymbol x\,\dd z\right|\le \varepsilon\int|v^\flat|^6|\nabla v^\flat|^2\,\dd \boldsymbol x\,\dd z
 +C_{M,\varepsilon}(1+\|v^\flat\|_{H^1}^2)
 (1+\|v^\flat\|_8^8).
\end{align*}

For the vertical $L^4$ test, the new local terms involving $\zeta$ and its spatial derivatives are bounded by constants times integrals of $|\partial_zv^\flat|^4$, $|v^\flat||\partial_zv^\flat|^3$, or $|\nabla v^\flat||\partial_zv^\flat|^3$. For the last one,
\begin{align*}
 \int|\nabla v^\flat||\partial_zv^\flat|^3\,\dd \boldsymbol x\,\dd z
 &\le\|v^\flat\|_{H^1}\|\partial_zv^\flat\|_6^3\\
 &\le C\|v^\flat\|_{H^1}\|\partial_zv^\flat\|_4^{3/2}
 \left(\|\partial_zv^\flat\|_4^4
 +\int|\partial_zv^\flat|^2|\nabla\partial_zv^\flat|^2\,\dd \boldsymbol x\,\dd z\right)^{3/8}\\
 &\le\varepsilon\int|\partial_zv^\flat|^2|\nabla\partial_zv^\flat|^2\,\dd \boldsymbol x\,\dd z
 +C_\varepsilon(1+\|v^\flat\|_{H^1}^2)(1+\|\partial_zv^\flat\|_4^4).
\end{align*}
The remaining nonlocal term arises from the factor
$w(v^\flat)\partial_z^2\zeta$, which is controlled by
\begin{align*}
 \|v^\flat\|_{H^1}
 \left\|\int|\partial_zv^\flat|^3\,\dd z\right\|_{L_h^2}
 \le C\|v^\flat\|_{H^1}\left\{
 \|\partial_zv^\flat\|_4^3
 +\|\partial_zv^\flat\|_4
 \left(\int|\partial_zv^\flat|^2|\nabla\partial_zv^\flat|^2\,\dd \boldsymbol x\,\dd z\right)^{1/2}\right\},
\end{align*}
and has the same upper bound. The horizontal Sobolev estimate follows by differentiating $|\partial_zv^\flat|^3$ and applying the two-dimensional embedding to its vertical integral. The pure $v^\flat$ terms are exactly those in the original vertical cascade.

In the $H^1$ test, the additional terms $B(v^\flat,\zeta)$, $B(\zeta,v^\flat)$ and $\nu A\zeta+B(\zeta,\zeta)$ paired with $Av^\flat$ are bounded by
\begin{align*}
 \varepsilon\|\Delta v^\flat\|_2^2
 +C_{M,\varepsilon}(1+\|v^\flat\|_{H^1}^2).
\end{align*}
Choosing the weights and the small absorption constants as in \Cref{lem:cascade-direct} proves \eqref{eq:Kvsharp}. This combined with \eqref{eq:energy-vflat} and Gronwall's inequality gives the desired estimate for $K,D$.

\emph{Step 2: Uniform spatial analyticity.}
By Step~1 and the smoothness of $\zeta$,
\begin{align*}
 \sup_{t\le T}\|v_t\|_{H^1}^2
 +\int_0^T\|v_t\|_{H^2}^2\,\dd t
 \le C_{R,M,T},
\end{align*}
and hence
\begin{align}\label{eq:bounded-driver-L4H32}
 \int_0^T\|v_t\|_{H^{3/2}}^4\,\dd t
 \le
 \sup_{t\le T}\|v_t\|_{H^1}^2
 \int_0^T\|v_t\|_{H^2}^2\,\dd t
 \le C_{R,M,T}.
\end{align}
Moreover, the shifted equation for $v^\flat$ and
\eqref{eq:bounded-driver-L4H32} give
$v^\flat\in H^1(0,T;H)\cap L^2(0,T;H^2)$.

The argument of
\cite[Lemma~6.4]{GGHHK}, specialized to the present $L^2$ setting,
applies using \eqref{eq:bounded-driver-L4H32}: partition $[0,T]$ into
finitely many intervals on which $\|v\|_{L^4H^{3/2}}$ is sufficiently
small and use
\begin{align*}
 \|B(v,h)+B(h,v)\|_{L^2H}
 \le C\|v\|_{L^4H^{3/2}}\|h\|_{L^4H^{3/2}}.
\end{align*}
Maximal $L^2$ regularity for $\partial_t+\nu A$, followed by iteration
over this uniformly bounded number of intervals, shows that
\begin{align*}
 \mathbb A : h\longmapsto\partial_th+\nu Ah+B(v,h)+B(h,v)
\end{align*}
is an isomorphism from
$\{h\in H^1(0,T;H)\cap L^2(0,T;H^2):h(0)=0\}$ onto $L^2(0,T;H)$, with uniformly
bounded inverse. We then use the spatial part of the parameter argument in
\cite[proof of Theorem~3.3(c)]{GGHHK}. Temporarily work on the full periodic hydrostatic space, extending the
diagnostic operator by the zero-mean vertical primitive so that it
commutes with spatial translations.
For $\boldsymbol y\in\mathbb R^3$, set
\begin{align*}
 \zeta_{\boldsymbol y}(t,\boldsymbol x,z)
 =\zeta(t,\boldsymbol x+t\boldsymbol y_h,z+t y_3).
\end{align*}
Since the forcing profiles are analytic and
$\|\omega\|_{\Omega_T}\le M$, the map
$\boldsymbol y\mapsto\zeta_{\boldsymbol y}$ is analytic into $C([0,T];H^k)$ on a common
complex neighborhood, uniformly over the driver family. Consider
\begin{align*}
 \mathcal F(\boldsymbol y,\widehat v)
 =
 \partial_t\widehat v+\nu A\widehat v
 +B(\widehat v+\zeta_{\boldsymbol y},\widehat v+\zeta_{\boldsymbol y})
 +\nu A\zeta_{\boldsymbol y}-\boldsymbol y\cdot\nabla\widehat v.
\end{align*}

At $(0,v^\flat)$ one has $\mathcal F(0,v^\flat)=0$, and by regarding $\mathcal F(\boldsymbol y,\cdot)$ on the affine space $v^\flat+
\{h\in H^1(0,T;H)\cap L^2(0,T;H^2):h(0)=0\}$, the derivative
with respect to $\widehat v$ is the uniformly invertible operator $\mathbb A$ above.
The analytic implicit function theorem therefore gives a common
parameter radius. For real $\boldsymbol y$, uniqueness identifies the resulting
solution with
\begin{align*}
 \widehat v_{\boldsymbol y}(t,\boldsymbol x,z)
 =v^\flat(t,\boldsymbol x+t\boldsymbol y_h,z+t y_3).
\end{align*}
Unlike the full space--time parameter argument of \cite{GGHHK}, no
time-rescaling parameter is introduced, so no time derivative of
$\omega$ is required. Cauchy's estimates and
\begin{align*}
 \left.\partial_{\boldsymbol y}^\beta\widehat v_{\boldsymbol y}
 (t,\boldsymbol x,z)\right|_{\boldsymbol y=0}
 =t^{|\beta|}\partial_{(\boldsymbol x,z)}^\beta
 v^\flat(t,\boldsymbol x,z)
\end{align*}
give, for every multi-index $\beta$,
\begin{align*}
 \sup_{\tau\le t\le T}\|\partial^\beta v_t^\flat\|_{H^1}
 \le C\,\beta!\,(C_\tau)^{|\beta|}.
\end{align*}
Thus, by analyticity of the noise profiles, after decreasing the analytic radius once, for every fixed
$s<\infty$ there are $\rho>0$ and $C<\infty$ such that
\begin{align*}
 \sup_{\tau\le t\le T}\|v_t\|_{\rho,s}\le C.
\end{align*}

\end{proof}

For $\rho>0$, $s\ge0$, and $R,T>0$, define the analytic reference class
\begin{align}\label{eq:analytic-reference-core}
 \mathcal C_{\rho,s}(R,T)
 =\Bigl\{(u,\omega)\in H^s\times \Omega_T:
 \sup_{0\le t\le T}\|\Phi_t(u,\omega)\|_{\rho,s}
 +\sum_{k=1}^m\|b_k\|_{\rho,s}\le R\Bigr\}.
\end{align}
The next lemma shows that a compact family of uniformly analytic
initial states, driven by a uniformly bounded family of continuous
paths, is contained in a single analytic reference class after
decreasing the analytic radius. We keep the notation $\zeta$ and $v^\flat$ from \eqref{eq:vflat}. 

\begin{lemma}\label{lem:analytic-core}
Let $s\ge6$ and suppose that the forcing profiles $\{b_j\}$ are analytic. Let $\mathscr D\subset H^s$ be compact and suppose that, for some $\rho_{\mathscr D}>0$ and $R_{\mathscr D}<\infty$,
\begin{align}\label{eq:analytic-initial-core}
 \sup_{u\in\mathscr D}\|u\|_{\rho_{\mathscr D},s}\le R_{\mathscr D}.
\end{align}
Then, for every $M,T<\infty$, there are $\rho\in(0,\rho_{\mathscr D})$ and $R<\infty$ such that
\begin{align}\label{eq:uniform-analytic-reference-core}
 \mathscr D\times\left\{\omega\in \Omega_T:
 \|\omega\|_{\Omega_T}\le M\right\}
 \subset\mathcal C_{\rho,s}(R,T).
\end{align}
\end{lemma}

\begin{proof}
Choose $\rho_Q>0$ such that
$\sum_{j=1}^m\|b_j\|_{\rho_Q,s+1}<\infty$, and decrease $\rho_{\mathscr D}$ if
necessary so that $\rho_{\mathscr D}<\rho_Q$. 
Fix $r_0=\rho_{\mathscr D}/2$. Since the profiles are analytic,
\begin{align*}
 \sup_{\|\omega\|_{\Omega_T}\le M}\sup_{t\le T}
 \|\zeta(t)\|_{r_0,s+1}<\infty.
\end{align*}
Using $s-1>3/2$ and the analytic algebra estimate $\|B(f,g)\|_{r,s-1}\le C_s\|f\|_{r,s}\|g\|_{r,s}$, 
the standard analytic energy estimate for
\eqref{eq:vflatPDE} gives
\begin{align*}
 \frac{d}{dt}\|v^\flat\|_{r_0,s}^2
 +\nu\|v^\flat\|_{r_0,s+1}^2
 \le C\left(1+\|v^\flat\|_{r_0,s}^4\right),
\end{align*}
with $C$ uniform for $\|\omega\|_{\Omega_T}\le M$. Since
\eqref{eq:analytic-initial-core} is uniform on $\mathscr D$, a first-exit
argument yields $t_*>0$ and $R_*<\infty$ such that
\begin{align}\label{eq:short-analytic-core}
 \sup_{u\in\mathscr D}\sup_{\|\omega\|_{\Omega_T}\le M}
 \sup_{0\le t\le t_*\wedge T}
 \|\Phi_t(u,\omega)\|_{r_0,s}\le R_*.
\end{align}
Since $\mathscr D$ is compact in $H^s$ and $s\ge6$, $K$ is bounded by some constant $R_K$ uniformly on $\mathscr D$. 
If $T>t_*/2$, \Cref{lem:bounded-driver-regularization}, applied with
initial bound $R_K$, driver bound $M$, and positive time $t_*/2$,
gives $\rho_1>0$ and
\begin{align*}
 \sup_{u\in\mathscr D}\sup_{\|\omega\|_{\Omega_T}\le M}
 \sup_{t_*/2\le t\le T}
 \|\Phi_t(u,\omega)\|_{\rho_1,s}<\infty.
\end{align*}
Together with \eqref{eq:short-analytic-core}, this gives
\eqref{eq:uniform-analytic-reference-core} after choosing
$0<\rho<\min\{r_0,\rho_1,\rho_Q\}$ 
and increasing $R$ to include
$\sum_{j=1}^m\|b_j\|_{\rho,s}$. If $T\le t_*/2$, the same conclusion
follows directly from \eqref{eq:short-analytic-core}.
\end{proof}

\subsection{Bounds on  linearizations and remainders}\label{app:endpoint-charts}
In this subsection, we derive estimates on linearizations
and remainders to justify the regular dependence required in \Cref{prop:compact-source-repair}.  

\begin{lemma}\label{lem:tangent-adjoint-core}
Fix $T<\infty$ and $s\ge6$. Let $v$ satisfy $\sup_{t\in[0,T]}\|v_t\|_{H^s}\le R$. Then there is a constant $C>0$ such that for every $\xi\in H$, the tangent flow $J_{r,t}$ satisfies
\begin{align}\label{eq:tangent-core-bound}
 \sup_{0\le r\le T}
 \left(
 \|J_{r,\cdot}\xi\|_{L^\infty(r,T;H)}
 +\|J_{r,\cdot}\xi\|_{L^2(r,T;H^1)}
 \right)
 \le C\|\xi\|_H.
\end{align}
The terminal adjoint $p_t=J_{t,T}^*\xi$ belongs to
$C([0,T];H)\cap L^2(0,T;H^1)\cap W^{1,\infty}(0,T;H^{-2})$, and
\begin{align}\label{eq:adjoint-core-bounds}
 \|p\|_{L^\infty([0,T];H)}
 +\|\partial_t p\|_{L^\infty([0,T];H^{-2})}
 \le C\|\xi\|_H.
\end{align}
Moreover, for almost every $t\in(0,T)$,
\begin{align}\label{eq:adjoint-core-energy}
 \frac{d}{dt}\|p_t\|_H^2
 \ge\nu\|\nabla p_t\|_2^2-C\|p_t\|_H^2.
\end{align}
\end{lemma}
\begin{proof}
For $h\in H$, the tangent flow satisfies
\begin{align*}
 \partial_t(J_{r,t}h)+\nu A(J_{r,t}h)
 +B(v_t,J_{r,t}h)+B(J_{r,t}h,v_t)=0,
 \qquad J_{r,r}h=h.
\end{align*}
Pairing with $J_{r,t}h$, the transport term
$B(v_t,J_{r,t}h)$ cancels, while \eqref{eq:one-sided-K} gives
\begin{align*}
 \bigl|\langle B(J_{r,t}h,v_t),J_{r,t}h\rangle\bigr|
 \le\frac{\nu}{2}\|\nabla J_{r,t}h\|_2^2
 +CK(v_t)\|J_{r,t}h\|_2^2.
\end{align*}
Since the regular reference bound controls $\sup_{t\le T}K(v_t)$,
Gronwall's inequality gives \eqref{eq:tangent-core-bound}.

By duality,
\begin{align*}
 \|p_t\|_H=\|J_{t,T}^*\xi\|_H\le C\|\xi\|_H.
\end{align*}
The adjoint satisfies
\begin{align}\label{eq:backward-adjoint-equation}
 \partial_t p_t=-DF(v_t)^*p_t.
\end{align}
As $\|v_t\|_{H^s}\le R$ uniformly on $[0,T]$, $DF(v_t):H^2\to H$ is uniformly
bounded. Hence
\begin{align*}
 \|\partial_t p_t\|_{H^{-2}}
 \le C\|p_t\|_H
 \le C\|\xi\|_H,
\end{align*}
which proves \eqref{eq:adjoint-core-bounds}. Reversing time in
\eqref{eq:backward-adjoint-equation} gives a forward parabolic equation,
so the standard energy approximation also yields
$p\in L^2(0,T;H^1)$ and continuity in $H$.

For smooth terminal data, using
$DF(v)h=-\nu Ah-B(v,h)-B(h,v)$ and the hydrostatic transport
cancellation gives
\begin{align*}
 \frac{d}{dt}\|p_t\|_H^2
 =-2\langle p_t,DF(v_t)p_t\rangle_H=2\nu\|\nabla p_t\|_2^2
 +2\langle B(p_t,v_t),p_t\rangle_H.
\end{align*}
Applying \eqref{eq:one-sided-K} once more,
\begin{align*}
 |\langle B(p_t,v_t),p_t\rangle_H|
 \le\frac{\nu}{2}\|\nabla p_t\|_2^2
 +CK(v_t)\|p_t\|_H^2,
\end{align*}
and \eqref{eq:adjoint-core-energy} follows. Approximation of the
terminal datum in $H$ yields the same inequality for arbitrary
$\xi\in H$. 
\end{proof}

Recall the Malliavin derivative $\mathcal A_T$ in \eqref{eq:control-operator}, the Cameron--Martin space $\mathsf H_T=L^2(0,T;\mathbb R^m)$, and  the Cameron--Martin embedding $\iota$ from \eqref{eq:iota}. The following lemma gives the remainder estimates for the first order approximation $J_{0,T}$ and $\mathcal A_T$ in the state and driver directions. 
\begin{lemma}
\label{lem:endpoint-remainder}
Fix $T>0$, $s\ge6$, and $R<\infty$. There exist
$\varepsilon_*>0$ and $C<\infty$, depending only on
$T,R,\nu$, and $Q$, such that the following holds.
Let $(u,W)$ be a reference pair whose trajectory
$v_t=\Phi_t(u,W)$ satisfies $\sup_{0\le t\le T}\|v_t\|_{H^s}\le R$. 
For every $\widehat u\in H^1$ and every
$h\in \mathsf H_T$, put $\varepsilon
=\|\widehat u-u\|_H+\|h\|_{\mathsf H_T}$. 
If $\varepsilon\le\varepsilon_*$, then
\begin{align}\label{eq:remainder}
 \left\|
 \Phi_T(\widehat u,W+\iota h)-\Phi_T(u,W)
 -J_{0,T}(\widehat u-u)-\mathcal A_Th
 \right\|_H
 \le C\varepsilon^2.
\end{align}
\end{lemma}

\begin{proof}
Write $\xi_t=\Phi_t(\widehat u,W+\iota h)-\Phi_t(u,W)$. 
Then
\begin{align*}
 \partial_t\xi+\nu A\xi+B(v,\xi)+B(\xi,v)+B(\xi,\xi)=Qh,
 \qquad
 \xi_0=\widehat u-u.
\end{align*}
Note that 
\begin{align*}
 \eta_t = J_{0,t}(\widehat u-u)+\int_0^tJ_{r,t}Qh(r)\,\dd r
\end{align*}
solves 
\begin{align*}
\partial_t \eta+\nu A \eta+B(v, \eta)+B(\eta, v)=Q h, \quad \eta_0=\widehat u-u. 
\end{align*}
Hence $\varrho=\xi-\eta$ satisfies
\begin{align*}
 \partial_t\varrho+\nu A\varrho+B(v,\varrho)+B(\varrho,v)
 =-B(\xi,\xi),\qquad \varrho_0=0.
\end{align*}
Variation of constants then gives at time $T$
\begin{align}\label{eq:first-order-remainder}
 \Phi_T(\widehat u,W+\iota h)-\Phi_T(u,W)
 -J_{0,T}(\widehat u-u)-\mathcal A_Th
 =
 -\int_0^TJ_{t,T}B(\xi_t,\xi_t)\,\dd t.
\end{align}
It therefore remains to estimate the integral on the right-hand side. 

Since
$\langle B(v,\xi),\xi\rangle=\langle B(\xi,\xi),\xi\rangle=0$,
\Cref{lem:one-sided-growth}, Young's inequality, and the uniform
$H^s$ bound on the reference trajectory give
\begin{align}\label{eq:remainder-energy}
 \sup_{0\le t\le T}\|\xi_t\|_H^2
 +\int_0^T\|\xi_t\|_{H^1}^2\,\dd t
 \le C\varepsilon^2.
\end{align}
Hence there exists $t_*\in[T/4,T/3]$ such that $\|\xi_{t_*}\|_{H^1}\le C\varepsilon$. 
For an admissible field $a$,
\begin{align}\label{eq:strong-bilinear}
 \|B(a,a)\|_H\le C\|a\|_{H^1}\|a\|_{H^2}.
\end{align}
Indeed,
\begin{align*}
 \|w(a)\partial_za\|_2
 \le C\|\nabla_ha\|_{L_h^4L_z^2}
       \|\partial_za\|_{L_h^4L_z^2}
 \le C\|a\|_{H^1}\|a\|_{H^2},
\end{align*}
while the horizontal term follows from
$\|a\|_6\|\nabla_ha\|_3
 \le C\|a\|_{H^1}^{3/2}\|a\|_{H^2}^{1/2}$.
The regular reference bound also gives
\begin{align*}
 \|B(v,a)+B(a,v)\|_H\le C\|a\|_{H^1}.
\end{align*}
Testing the difference equation with $A\xi$ therefore yields
\begin{align*}
 \frac{d}{dt}\|\xi\|_{H^1}^2+\nu\|\xi\|_{H^2}^2
 \le C\|\xi\|_{H^1}\|\xi\|_{H^2}^2
 +C\|\xi\|_{H^1}^2+C_Q|h(t)|^2.
\end{align*}
As long as $\|\xi\|_{H^1}\le\nu/(2C)$, the first term on the
right is absorbed by the dissipation. As $\|h\|_{\mathsf H_T}\leq \varepsilon$, starting from $t_*$ and using
Gronwall's inequality, followed by a first-exit argument and a decrease
of $\varepsilon_*$, gives
\begin{align}\label{eq:remainder-strong}
 \sup_{t_*\le t\le T}\|\xi_t\|_{H^1}^2
 +\int_{t_*}^T\|\xi_t\|_{H^2}^2\,\dd t
 \le C\varepsilon^2.
\end{align}
We also use the weak bilinear bound
\begin{align}\label{eq:weak-bilinear}
 \|B(a,a)\|_{H^{-3/2}}\le C\|a\|_{H^1}^2.
\end{align}
Indeed, for the vertical term, pairing against $\varphi\in H^{3/2}$ and
integrating by parts gives
\begin{align*}
 \left|\int w(a)\partial_za\cdot\varphi\,\dd x\right|
 \le
 C\|w(a)\|_{L_h^2L_z^\infty}
 \|a\|_{L_h^4L_z^2}
 \|\partial_z\varphi\|_{L_h^4L_z^2}
 \le C\|a\|_{H^1}^2\|\varphi\|_{H^{3/2}},
\end{align*}
while the horizontal term is estimated in divergence form in the same
way. 
On $[0,T/2]$, the positive terminal gap and the uniform regularity of
the reference trajectory give
\begin{align*}
 \sup_{0\le t\le T/2}
 \|J_{t,T}\|_{\mathcal L(H^{-3/2},H)}=\sup_{0\le t\le T/2}
 \|J_{t,T}^*\|_{\mathcal L(H,H^{3/2})}\le C.
\end{align*}
Hence \eqref{eq:weak-bilinear} and \eqref{eq:remainder-energy} imply
\begin{align*}
 \left\|
 \int_0^{T/2}J_{t,T}B(\xi_t,\xi_t)\,\dd t
 \right\|_H
 \le C\int_0^{T/2}\|\xi_t\|_{H^1}^2\,\dd t
 \le C\varepsilon^2.
\end{align*}
On $[T/2,T]$, \Cref{lem:tangent-adjoint-core},
\eqref{eq:strong-bilinear}, and \eqref{eq:remainder-strong} give
\begin{align*}
 \left\|
 \int_{T/2}^TJ_{t,T}B(\xi_t,\xi_t)\,\dd t
 \right\|_H
 &\le
 C\int_{T/2}^T\|\xi_t\|_{H^1}\|\xi_t\|_{H^2}\,\dd t\\
 &\le
 C\left(\int_{t_*}^T\|\xi_t\|_{H^1}^2\,\dd t\right)^{1/2}
 \left(\int_{t_*}^T\|\xi_t\|_{H^2}^2\,\dd t\right)^{1/2}\le C\varepsilon^2.
\end{align*}
Combining the two time intervals with \eqref{eq:first-order-remainder} proves \eqref{eq:remainder}. 
\end{proof}

We also need the following differentiability of the solution map and the linearization on the driver variable, which follows from considering the shifted formulation $v^\flat=v-Q\omega$ as in \eqref{eq:vflat}. 

\begin{lemma}\label{lem:driver-chart}
Let $s\ge6$ and assume that the noise profiles of $Q$ belong to
$H^{s+2}$. Fix $T,R,M<\infty$. There exist $\delta>0$ and $C<\infty$,
depending only on $T,R,M,\nu$, and $Q$, such that the following holds. Let $(u,\omega)\in H^s\times\Omega_T$ satisfy
\begin{align}\label{eq:state-driver-bounds}
 \|\omega\|_{\Omega_T}\le M,
 \qquad
 \sup_{0\le t\le T}\|\Phi_t(u,\omega)\|_{H^s}\le R.
\end{align}
Then on the neighborhood $\mathcal O_{\delta}:=\left\{(\widehat u,\widehat\omega):
\|\widehat u-u\|_{H^s}
+\|\widehat\omega-\omega\|_{\Omega_T}<\delta\right\},$ 
the driver map
$\widehat\omega\mapsto\Phi_T(\widehat u,\widehat\omega)$ is twice
continuously differentiable and
$\widehat\omega\mapsto J_{0,T}(\widehat u,\widehat\omega)$ is
continuously differentiable, with
\begin{align*}
 \|D_\omega\Phi_T\|_{\mathcal L(\Omega_T,H^s)}
 +\|D_\omega^2\Phi_T\|_{\mathcal L^2(\Omega_T,H^s)}
 +\|D_\omega J_{0,T}\|_{\mathcal L(\Omega_T,\mathcal L(H,H))}
 \le C.
\end{align*}
\end{lemma}

\begin{proof}
For $s\ge6$, the hydrostatic nonlinearity is a bounded bilinear map
$H^s\times H^s\to H^{s-1}$ since $H^{s-1}$ is an algebra. Set
$R_0=R+C_QM$, 
so that $\sup_{t\le T}\|v_t^\flat\|_{H^s}\le R_0$. Fix an interval
$I=[a,a+\tau]\subset[0,T]$  and define $\mathcal M=\mathcal M_{\mathsf v_a,\widehat\omega}$ by 
\begin{align*}
 \mathcal M(\mathsf v)(t)
 :=
 e^{-\nu A(t-a)}\mathsf v_a
 -\int_a^t e^{-\nu A(t-r)}
 \Bigl\{\nu AQ\widehat\omega(r)
 +B\bigl(\mathsf v(r)+Q\widehat\omega(r),
 \mathsf v(r)+Q\widehat\omega(r)\bigr)\Bigr\}\,\dd r .
\end{align*}
The heat estimate $\|e^{-\nu At}\|_{H^{s-1}\to H^s}\le Ct^{-1/2}$ shows that  
the map $\mathcal M$ acts on
$X_I=C(I;H^s)$. Consider the reference tube $\mathbb B_I
=
\left\{\mathsf v\in X_I:
\|\mathsf v-v^\flat\|_{X_I}\le1\right\}.$
If $\|\widehat\omega-\omega\|_{C(I)}\le1$, all arguments of $B$
in this tube are bounded in $H^s$ by a constant depending only on
$R,M$, and $Q$. Hence
\begin{align}\label{eq:driver-chart-contraction}
 \|\mathcal M(\mathsf u)
       -\mathcal M(\mathsf v)\|_{X_I}
 \le C\tau^{1/2}\|\mathsf u - \mathsf v\|_{X_I},
\end{align}
and choosing $\tau=\tau(R,M,\nu,Q)>0$ so that  $C\tau^{1/2}\leq 1/4$ makes $\mathcal M$ a contraction. Moreover, viewed as a map of $(\mathsf v,\mathsf v_a,\widehat\omega)$,
$\mathcal M$ is twice continuously
differentiable as its dependence on $\mathsf v_a$ is linear and on $(\mathsf v,\widehat\omega)$ is quadratic.  The same heat estimate and
bilinear bound show that these derivatives are uniformly bounded on
$\mathbb B_I$. At  $(v_a^{\flat},\omega)$ one has
$\mathcal M (v^\flat)=v^\flat$, while by \eqref{eq:driver-chart-contraction} and the same estimates,
\begin{align*}
 \|\mathcal M_{\mathsf v_a,\widehat\omega}(\mathsf v)-v^\flat\|_{X_I}
 \le \frac14\|\mathsf v-v^\flat\|_{X_I}
 +C\left(
 \|\mathsf v_a-v_a^\flat\|_{H^s}
 +\|\widehat\omega-\omega\|_{C(I)}
 \right).
\end{align*}
Therefore, after choosing
$\delta_0>0$ sufficiently small, uniformly for $(v_a^{\flat},\omega)$ satisfying \eqref{eq:state-driver-bounds}, 
\begin{align*}
 \|\mathsf v_a-v_a^{\flat}\|_{H^s}
 +\|\widehat\omega-\omega\|_{C(I)}
 <\delta_0
\end{align*}
implies that $\mathcal M$ is a contraction mapping $\mathbb B_I$
into itself. 
Thus $I-D_{\mathsf v}\mathcal M$
is uniformly invertible, and the Banach-space implicit function theorem shows that the
local solution map $(\mathsf v_a,\widehat\omega)\mapsto \mathsf v$ is $C^2$, with
uniform first and second derivative bounds.

Cover $[0,T]$ by finitely many intervals of the above length. Composing
the corresponding local solution maps and decreasing the parameter
radius finitely many times gives constants $\delta>0$ and $C<\infty$,
depending only on $T,R,M,\nu$, and $Q$, such that the asserted $C^2$
driver estimate holds on $\mathcal O_\delta$. In particular, for
$\vartheta\in\Omega_T$ and
$Z=D_\omega\Phi_\cdot(u,\omega)[\vartheta]$,
\begin{align}\label{eq:driver-variation}
 \|Z\|_{C([0,T];H^s)}
 &\le C\|\vartheta\|_{\Omega_T},\notag\\
 \|\Phi_\cdot(u,\omega+\vartheta)-\Phi_\cdot(u,\omega)-Z\|_{C([0,T];H^s)}
 &\le C\|\vartheta\|_{\Omega_T}^2.
\end{align}

It remains to differentiate the tangent flow with respect to the
driver while retaining the incoming tangent space $H$. For a smooth
field $a$, set $\mathcal T_a\xi=B(a,\xi)+B(\xi,a)$. 
Then
\begin{align}\label{eq:smooth-coefficients}
 \|\mathcal T_a\xi\|_{H^{-1}}
 \le C\|a\|_{H^s}\|\xi\|_H.
\end{align}
Indeed, for the vertical part of $B(\xi,a)$, let $\mathscr I_z$ be the
zero-mean vertical primitive. Since
\begin{align*}
 w(\xi)=-\operatorname{div}_h\mathscr I_z\xi,
 \qquad
 \|\mathscr I_z\xi\|_2\le\|\xi\|_2,
\end{align*}
integration by parts in the horizontal variables gives, for
$\varphi\in H^1$,
\begin{align*}
 \left|\int w(\xi)\partial_za\cdot\varphi\,\dd x\right|
 \le C\|\xi\|_2
 \left(
 \|\nabla_h\partial_za\|_\infty\|\varphi\|_2
 +\|\partial_za\|_\infty\|\nabla_h\varphi\|_2
 \right).
\end{align*}
The remaining terms are estimated in divergence form in the same way,
and Sobolev embedding proves \eqref{eq:smooth-coefficients}.
For $Z=D_\omega\Phi_\cdot(u,\omega)[\vartheta]$, and $J_t=J_{0,t}$, let $Y$ be the solution to 
\begin{align*}
 \partial_tY+\nu AY+\mathcal T_vY=-\mathcal T_ZJ_{t}\xi,
 \qquad Y_0=0.
\end{align*}
Then by the inhomogeneous tangent energy estimate, \eqref{eq:smooth-coefficients}, \eqref{eq:driver-variation} and
\Cref{lem:tangent-adjoint-core}, we have  
\begin{align}\label{eq:driver-J-linear}
 \|Y\|_{L^\infty(0,T;H)}
 +\|Y\|_{L^2(0,T;H^1)}
 \le C\|\mathcal T_ZJ\xi\|_{L^2(0,T;H^{-1})}\le  C\|\vartheta\|_{\Omega_T}\|\xi\|_H.
\end{align}
Set $\mathcal Z
=\Phi_\cdot(u,\omega+\vartheta)-\Phi_\cdot(u,\omega),$
and let $\mathsf J_t=J_{0,t}(u,\omega+\vartheta)\xi$.
Subtracting the two tangent equations for $J_t$ and $\mathsf J_t$ gives
\begin{align*}
 (\partial_t+\nu A+\mathcal T_v)(\mathsf J-J)
 =-\mathcal T_{\mathcal Z}\mathsf J,
\end{align*}
with zero initial data. 
The same inhomogeneous energy estimate together with \eqref{eq:smooth-coefficients}, \eqref{eq:driver-variation} and
\Cref{lem:tangent-adjoint-core} yields
\begin{align*}
 \|\mathsf J-J\|_{L^\infty(0,T;H)} + \|\mathsf J-J\|_{L^2(0,T;H^1)}
 \le C\|\vartheta\|_{\Omega_T}\|\xi\|_H.
\end{align*}
Subtracting also the equation for $Y$ yields
\begin{align*}
 (\partial_t+\nu A+\mathcal T_v)(\mathsf J-J-Y)
 =
 -\mathcal T_{\mathcal Z}(\mathsf J-J)
 -\mathcal T_{\mathcal Z-Z}J.
\end{align*}
Using \eqref{eq:driver-variation} once more,
\begin{align*}
 \|\mathsf J-J-Y\|_{L^\infty(0,T;H)}+ \|\mathsf J-J-Y\|_{L^2(0,T;H^1)}
 \le C\|\vartheta\|_{\Omega_T}^2\|\xi\|_H.
\end{align*}
Taking the endpoint and then the supremum over $\|\xi\|_H\le1$ proves
that $\omega\mapsto J_{0,T}(u,\omega)$ is differentiable in
$\mathcal L(H,H)$, with
\begin{align*}
 D_\omega J_{0,T}(u,\omega)[\vartheta]\xi=Y_T,
 \qquad
 \|D_\omega J_{0,T}(u,\omega)\|_
 {\mathcal L(\Omega_T,\mathcal L(H,H))}\le C.
\end{align*}
The same estimates for two nearby reference pairs give continuity of
this derivative and complete the proof.
\end{proof}

\section{Quantitative nondegeneracy of Brownian signatures}\label{app:probability}

This appendix establishes the quantitative nondegeneracy of Brownian signatures needed in \Cref{sec:quantitative-Malliavin-observability}. The underlying mechanism belongs to the pathwise H\"ormander theory for degenerate SPDEs. In the additive-noise setting, \cite{MattinglyPardoux,HM06,HM} replace the classical Norris lemma by quantitative nondegeneracy estimates for Wiener polynomials with possibly non-adapted coefficients, while  \cite{GerasimovicsHairer} use rough-path methods to obtain a genuinely pathwise H\"ormander argument for semilinear SPDEs.  Our terminal adjoint produces the anticipative quantities $\langle J_{t,T}^*\xi,g(v_t)\rangle$, so we follow the same pathwise principle.

The probabilistic input required here is, however, different from a fixed-order Norris estimate. The Lie-word depth grows linearly with the Fourier frequency resolution, and we therefore need a nondegeneracy estimate uniform in the parabolic degree cutoff. Iterated It\^o and Stratonovich integrals naturally form the Brownian signature, as in the stochastic Taylor and cubature constructions of \cite{BenArous1989,Kusuoka2001,LyonsVictoir,BayerTeichmann}. We exploit this algebraic structure together with It\^o orthogonality, independent Brownian increments and Gaussian hypercontractivity to prove a quantitative frame estimate with explicit dependence on the word parabolic degree. The resulting bound holds simultaneously for all coefficient vectors, including coefficients selected after observing the Brownian path, and can be imposed on exponentially many shrinking time intervals. Combined with the controlled Taylor remainder, this yields the $e^{CN^2}$ low-mode observation cost and, ultimately, the polynomial accuracy--cost bound $C_\delta\lesssim\delta^{-b}$ in \Cref{prop:polynomial-source}.

For a word $\mathtt w=(i_1,\cdots,i_k)$, denote the
corresponding iterated It\^o integral on $[s,t]$ as 
\begin{align}\label{eq:ito-iteration}
I_{\mathtt w}(s,t)=\int_{s<r_1<\cdots<r_k<t} 1 d Z_{r_1}^{i_1} \cdots d Z_{r_k}^{i_k}, \, \quad  I_{\varnothing}(s,t)=1, 
\end{align}
where 
$dZ^0=dt$, \, $dZ^i=dW^i$ for $1\le i\le m$ and $\varnothing$ denotes the empty word. The corresponding iterated Stratonovich integral is denoted as 
\begin{align}\label{eq:stratonovich-iteration}
 S_{\mathtt w}(s,t)
 =
 \int_{s<r_1<\cdots<r_k<t}
 1\,
 \circ dZ_{r_1}^{i_1}\cdots\circ dZ_{r_k}^{i_k}, \, \quad S_{\varnothing}(s,t)=1.
\end{align}

The following Gramian type estimate is a quantitative growing-order refinement of the classical nondegeneracy furnished by Wiener-It\^o orthogonality. Its proof combines the chaos decomposition with a uniform coefficient estimate for polynomials on a simplex.
\begin{lemma}
\label{lem:mean-frame}
There is
$C=C(m)$ such that, for every integer $K\ge1$ and every family of real
coefficients $c_{\mathtt w}\in \mathbb R$,
\begin{align}\label{eq:mean-frame}
 \mathbb E\int_0^1
 \left|
 \sum_{|\mathtt w|_{\rm deg}\le K}
 c_{\mathtt w}I_{\mathtt w}(0,t)
 \right|^2dt
 \ge
 \exp\{-CK\log(K+2)\}
 \sum_{|\mathtt w|_{\rm deg}\le K}|c_{\mathtt w}|^2.
\end{align}
\end{lemma}

\begin{proof}
Every word $\mathtt w$ has a unique Brownian-letter subsequence
$i=(i_1,\ldots,i_\ell)$ and a tuple
$a=(a_0,\ldots,a_\ell)$ recording the numbers of zero letters before
$i_1$, between consecutive Brownian letters, and after $i_\ell$.
Its parabolic degree is $\ell+2|a|$ with $|a|=a_0+\cdots+a_\ell$. 
Stochastic Fubini gives
\begin{align}\label{eq:ito-gap-kernel}
 I_{\mathtt w}(0,t)
 =
 \int_{0<s_1<\cdots<s_\ell<t}
 \prod_{j=0}^{\ell}
 \frac{(s_{j+1}-s_j)^{a_j}}{a_j!}\,
 dW^{i_1}_{s_1}\cdots dW^{i_\ell}_{s_\ell},
 \qquad
 s_0=0,\quad s_{\ell+1}=t.
\end{align}
For $\ell=0$, this reduces to $I_{\mathtt w}(0,t)=t^{a_0}/a_0!$. Different values of $\ell$ belong to different Wiener chaoses and are
orthogonal. For fixed $\ell$, iterated It\^o isometry also makes the
terms corresponding to different ordered Brownian subsequences
$i\in\{1,\ldots,m\}^\ell$ orthogonal. Hence, after integrating in
$t$ and introducing the gap variables
\begin{align*}
 r_0=s_1,\qquad
 r_j=s_{j+1}-s_j,\quad 1\le j\le\ell-1,\qquad
 r_\ell=t-s_\ell,
\end{align*}
and denoting $\Delta_{\ell} =\{r_j\ge 0:r_0+\cdots+r_{\ell}<1\}$, 
the left-hand side of \eqref{eq:mean-frame} equals
\begin{align}\label{eq:mean-frame-blocks}
 \sum_{\ell=0}^K
 \sum_{i\in\{1,\ldots,m\}^{\ell}}
 \int_{\Delta_{\ell}}
 \left|
 \sum_{|a|\le (K-\ell)/2}
 c_{\mathtt w(i,a)}
 \prod_{j=0}^{\ell}\frac{r_j^{a_j}}{a_j!}
 \right|^2dr.
\end{align}
Fix $\ell$ and $i$, and put
\begin{align}\label{eq:hdp}
 h=2(\ell+1),\qquad
 D=\left\lfloor\frac{K-\ell}{2}\right\rfloor,\qquad
 P_i(r)
 =
 \sum_{|a|\le D}
 c_{\mathtt w(i,a)}\prod_{j=0}^{\ell}\frac{r_j^{a_j}}{a_j!}.
\end{align}
The cube $[0,h^{-1}]^{\ell+1}$ is contained in the simplex
$\Delta_{\ell}$. Let $\phi_j$ be the $L^2(0,1)$-normalized shifted Legendre polynomial
of degree $j$, so that
\begin{align}\label{eq:Legendre-polynomial}
 \phi_j(x)=\sum_{q=0}^j b_{j,q}x^q,
 \qquad
 \sum_{q=0}^j|b_{j,q}|
 \le C^j\sqrt{2j+1}.
\end{align}
The corresponding normalized basis on $[0,h^{-1}]$ is
\begin{align*}
 \phi_{j,h}(x)
 =
 h^{1/2}\phi_j(hx)
 =
 \sum_{q=0}^j b_{j,q}h^{q+1/2}x^q. 
\end{align*}
For a tensor basis polynomial
$\Phi_a(r)=\prod_{q=0}^{\ell}\phi_{a_q,h}(r_q)$ with $a=(a_0,\ldots,a_\ell)$  and 
$|a|\le D$, \eqref{eq:Legendre-polynomial} implies that the coefficient vector satisfies
\begin{align*}
 \|\operatorname{coeff}(\Phi_a)\|_{\ell^2}
 &\le
 \|\operatorname{coeff}(\Phi_a)\|_{\ell^1}\le
 h^{|a|+(\ell+1)/2}
 C^{|a|}
 \prod_{q=0}^{\ell}\sqrt{2a_q+1}\\
 &\le
 h^{D+(\ell+1)/2}
 \exp\{C(D+\ell+1)\log(D+\ell+3)\}.
\end{align*}
There are at most 
$2^{D+\ell+1}$ such basis polynomials. Since
$D+\ell+1\le K+1$ by \eqref{eq:hdp}, Cauchy--Schwarz and orthonormality therefore give
\begin{align}\label{eq:monomial-coefficient-bound}
 \sum_{|a|\le D}
 \left|
 \frac{c_{\mathtt w(i,a)}}{\prod_{j=0}^{\ell}a_j!}
 \right|^2
 \le
 \exp\{CK\log(K+2)\}
 \int_{[0,h^{-1}]^{\ell+1}}|P_i(r)|^2\,dr.
\end{align}
Moreover,
\begin{align*}
 \prod_{j=0}^{\ell}a_j!
 \le |a|!
 \le D!
 \le \exp\{CK\log(K+2)\}.
\end{align*}
Thus, after enlarging $C$,
\begin{align}\label{eq:simplex-coefficient-bound}
 \sum_{|a|\le D}|c_{\mathtt w(i,a)}|^2
 \le
 \exp\{CK\log(K+2)\}
 \int_{\Delta_{\ell}}
 |P_i(r)|^2\,dr.
\end{align}
Summing \eqref{eq:simplex-coefficient-bound} over $\ell$ and $i$ and using \eqref{eq:mean-frame-blocks} 
proves \eqref{eq:mean-frame}.
\end{proof}

The following lemma gives standard signature conventions and It\^o--Stratonovich conversion that will be used later. Recall the iterated integrals $I_{\mathtt w}(s,t)$ and $S_{\mathtt w}(s,t)$ in \eqref{eq:ito-iteration} and \eqref{eq:stratonovich-iteration}. The It\^o signature is denoted as  $\mathbf I(s,t) :=  (I_{\mathtt w}(s,t))_{\mathtt w}$. Let $\mathcal M(\mathtt w)$ denote the collections of pairwise
disjoint adjacent equal nonzero pairs in $\mathtt w$, including the empty collection. For
$M\in\mathcal M(\mathtt w)$, let $\mathtt w/M$ denote the word
obtained by contracting each selected pair $ii$ to the time letter
$0$.
\begin{lemma}
\label{lem:signature-conventions}
Let $s\le a\le t$. 
\begin{enumerate}
\item One has 
\begin{align}\label{eq:ito-chen}
 I_{\mathtt w}(s,t)
 =
 \sum_{\mathtt w=\mathtt u\mathtt v}
 I_{\mathtt u}(s,a)I_{\mathtt v}(a,t),
\end{align}
where the sum runs over all decompositions of $\mathtt w$ into a
prefix $\mathtt u$ and a suffix $\mathtt v$, including the empty word. Particularly, in every finite tensor truncation,
\begin{align*}
 \mathbf I(s,t)=\mathbf I(s,a)\mathbf I(a,t),
\end{align*}
where the product is the tensor-algebra concatenation product.
\item The following conversion rules hold: 
\begin{align*}
 S_{\mathtt w}(s,t)
 =
 \sum_{M\in\mathcal M(\mathtt w)}
 2^{-|M|}I_{\mathtt w/M}(s,t),\quad 
 I_{\mathtt w}(s,t)
=
 \sum_{M\in\mathcal M(\mathtt w)}
 (-2)^{-|M|}S_{\mathtt w/M}(s,t).
\end{align*}
Both transformations preserve parabolic degree. Restricted to the
coordinates of parabolic degree at most $K$, their Euclidean operator
norms are bounded by $\exp(C_mK)$.
\end{enumerate}
\end{lemma}

\begin{proof}
For $\mathtt w=(i_1,\ldots,i_r)$, decompose the simplex $\{s<r_1<\cdots<r_r<t\}$ into subregions 
according to the unique $q\in\{0,\ldots,r\}$ for which
\begin{align*}
 s<r_1<\cdots<r_q<a<r_{q+1}<\cdots<r_r<t.
\end{align*}
The corresponding integral factors as $ I_{(i_1,\ldots,i_q)}(s,a)
I_{(i_{q+1},\ldots,i_r)}(a,t)$, 
which gives \eqref{eq:ito-chen}. For the change of coordinates, using 
\begin{align}\label{eq:basic-ito-strat}
 S_{\mathtt w i}(s,t)
 =
 \int_s^t S_{\mathtt w}(s,\tau)\circ dW_\tau^i =
 \begin{cases}
 \displaystyle
 \int_s^t S_{\mathtt v i}(s,\tau)\,dW_\tau^i
 +\frac12 S_{\mathtt v0}(s,t),
 & \mathtt w=\mathtt v i,\\[2mm]
 \displaystyle
 \int_s^t S_{\mathtt v j}(s,\tau)\,dW_\tau^i,
 & \mathtt w=\mathtt v j,\quad j\ne i
 \end{cases}
\end{align}
inductively shows the first sum. Reversing the sign of the correction gives the second sum. 

Each contraction replaces two Brownian letters of total parabolic
degree two by one time letter of degree two, so both transformations
preserve parabolic degree. Note that $\#\mathcal M(\mathtt w)\le 2^K$  for $|\mathtt w|_{\rm deg}\le K$
and the 
number $N_K$ of words of parabolic degree at most $K$ satisfies
\begin{align*}
 N_K\le (m+2)^{K+1}.
\end{align*}
Hence the Frobenius norms of the two truncated change-of-coordinate
matrices satisfy
\begin{align}\label{eq:ito-Stratonovich-transfer}
 \|T_K\|_{\mathrm F}^2+\|T_K^{-1}\|_{\mathrm F}^2
 \le
 2N_K2^K
 \le
 \exp(C_mK),
\end{align}
after enlarging $C_m$. Since the Euclidean operator norm is bounded by
the Frobenius norm, the claimed estimate follows.
\end{proof}

Next we pass from the averaged estimate of \Cref{lem:mean-frame} to a
pathwise nondegeneracy estimate on a subset of Wiener space of positive
measure.

\begin{proposition}
\label{prop:pathwise-frame}
There are a compact set $G$ of positive Wiener probability and a
deterministic constant $C<\infty$ such that, simultaneously for every
$K\ge1$ and every family of real coefficients
$(c_{\mathtt w})_{|\mathtt w|_{\rm deg}\le K}$,
\begin{align}\label{eq:pathwise-frame}
 \int_0^1
 \left|
 \sum_{|\mathtt w|_{\rm deg}\le K}
 c_{\mathtt w}I_{\mathtt w}(0,t)
 \right|^2dt
 \ge
 e^{-CK^2}
 \sum_{|\mathtt w|_{\rm deg}\le K}|c_{\mathtt w}|^2,
 \qquad \omega\in G.
\end{align}
In particular, the coefficients may be chosen after the Brownian path
has been observed. 
\end{proposition}

\begin{proof}
We divide the proof into three steps. 

\emph{Step 1: bound for the signature and its inverse.}
Let $\mathcal T_1$ be the completed tensor algebra over
$\{0,1,\ldots,m\}$ with norm given by the sum of the absolute values of
all word coefficients and multiplication given by concatenation. We first claim the existence of $C_m>0$ such that 
\begin{align}\label{eq:signature-integrability}
 \mathbb E\sup_{0\le t\le1}
 \left(
 \|\mathbf I(0,t)\|_{\mathcal T_1}
 +
 \|\mathbf I(0,t)^{-1}\|_{\mathcal T_1}
 \right)\leq C_m.
\end{align}
Indeed, by induction and the martingale maximal inequality or Cauchy--Schwarz, for every word $\mathtt w$ of ordinary length $n$,
\begin{align}\label{eq:ito-factorial-bound}
 \mathbb E\sup_{0\le u\le t}
 |I_{\mathtt w}(0,u)|^2
 \le
 \frac{4^nt^n}{n!},
 \qquad 0\le t\le1,
\end{align}
which implies
\begin{align*}
 \mathbb E\sup_{t\le1}\|\mathbf I(0,t)\|_{\mathcal T_1}
 = \mathbb E\sup_{t\le1} \sum_{\mathtt w}|I_{\mathtt w}(0,t)|
 \le\sum_{n\ge0}
 \frac{(2m+2)^n}{\sqrt{n!}}
 <\infty.
\end{align*}
For the inverse, work first in each finite tensor truncation. Since
\begin{align*}
 d\mathbf I
 =
 \mathbf I e_0\,dt
 +
 \sum_{i=1}^m\mathbf I e_i\,dW^i,
\end{align*}
the It\^o product rule shows that $\mathbf V=\mathbf I^{-1}$ solves
\begin{align}\label{eq:inverse-signature-equation}
 d\mathbf V
 =
 -\sum_{i=1}^m e_i\mathbf V\,dW^i
 +
 \left(
 \sum_{i=1}^m e_ie_i-e_0
 \right)\mathbf V\,dt,
 \qquad
 \mathbf V(0)=1.
\end{align}
The Picard expansion of \eqref{eq:inverse-signature-equation}, together
with \eqref{eq:ito-factorial-bound}, is bounded in
$L^1(\Omega;C([0,1];\mathcal T_1))$ by a convergent series of the form
$\sum_{n\ge0}C_m^n/\sqrt{n!}$, uniformly in the tensor truncation.
Passing to the limit gives the algebra inverse and proves
\eqref{eq:signature-integrability}.

\emph{Step 2: bound for fixed coefficient and word cutoff.}
Choose $R<\infty$ such that
\begin{align}\label{eq:prefix-event}
 G_R
 =
 \left\{
 \sup_{t\le1}
 \left(
 \|\mathbf I(0,t)\|_{\mathcal T_1}
 +
 \|\mathbf I(0,t)^{-1}\|_{\mathcal T_1}
 \right)
 \le R
 \right\},
 \qquad
 \mathbb P(G_R)\ge\frac34.
\end{align}
For $K\ge1$, let $D_K$ be the number of words of parabolic degree at most $K$ and set $\lambda_K=e^{-C_mK\log(K+2)}$. Then $D_K\le(m+2)^{K+1}$ 
and \Cref{lem:mean-frame} gives the mean lower bound with constant
$\lambda_K$. Fix a deterministic unit vector $c\in\mathbb R^{D_K}$, and an
integer $M\ge1$. Put $h=M^{-1}$ and $t_j=jh$. On the $j$-th interval
define
\begin{align*}
 \widetilde W_s^{(j)}
 =
 h^{-1/2}\bigl(W_{t_j+hs}-W_{t_j}\bigr),
 \qquad 0\le s\le1,
\end{align*}
and let $I_{\mathtt w}^{(j)}(0,s)$ be the corresponding iterated
It\^o integrals. Brownian scaling gives
\begin{align*}
 I_{\mathtt w}(t_j,t_j+hs)
 =
 h^{|\mathtt w|_{\rm deg}/2}
 I_{\mathtt w}^{(j)}(0,s).
\end{align*}
For
\begin{align*}
 F_c(t)
 =
 \sum_{|\mathtt w|_{\rm deg}\le K}
 c_{\mathtt w}I_{\mathtt w}(0,t),
\end{align*}
\Cref{lem:signature-conventions} therefore yields
\begin{align}\label{eq:cell-chen-expansion}
 F_c(t_j+hs)
 =\sum_{\mathtt v}\left(\sum_{\mathtt u}I_{\mathtt u}(0,t_j)c_{\mathtt{uv}}\right)I_{\mathtt v}(t_j,t_j+hs)
 =\sum_{|\mathtt w|_{\rm deg}\le K}
 c_{\mathtt w}^{(j)}
 h^{|\mathtt w|_{\rm deg}/2}
 I_{\mathtt w}^{(j)}(0,s),
\end{align}
by relabeling the suffix $\mathtt v$ as $\mathtt w$, 
where $c^{(j)}=L_{\mathbf I(0,t_j)}^*c$ with 
\[(L_{\mathbf I(0,t_j)}^*c)_{\mathtt w} = \sum_{\mathtt u, \, |\mathtt u\mathtt w|_{\rm deg}\leq K}I_{\mathtt u}(0,t_j)c_{\mathtt{uw}}. \]
The convolution inequality gives $\|L_{\mathbf I}\|_{\ell^2\to\ell^2}\le \|\mathbf I\|_{\mathcal T_1}$. 
Hence, 
\begin{align}\label{eq:prefix-coefficient-lower}
 1=\|c\|_{\ell^2}
 \le
 \|L_{\mathbf I(0,t_j)^{-1}}^*\|\,\|c^{(j)}\|_{\ell^2}
 \le
 R\|c^{(j)}\|_{\ell^2}
\end{align}
on $A_j:= \{\|\mathbf I(0,t_j)^{-1}\|_{\mathcal T_1}\leq R\}\in \mathcal F_{t_j}$. Set
\begin{align*}
 Z_j
 =
 \int_{t_j}^{t_j+h}|F_c(t)|^2dt.
\end{align*}
Conditioned on $\mathcal F_{t_j}$, the process
$\widetilde W^{(j)}$ is a fresh Brownian motion. Thus
\eqref{eq:cell-chen-expansion}, \Cref{lem:mean-frame}, and
\eqref{eq:prefix-coefficient-lower} give
\begin{align*}
 \mathbb E[Z_j\mid\mathcal F_{t_j}]=
 h\,
 \mathbb E\int_0^1
 \left|
 \sum_{|\mathtt w|_{\rm deg}\le K}
 c_{\mathtt w}^{(j)}
 h^{|\mathtt w|_{\rm deg}/2}
 I_{\mathtt w}^{(j)}(0,s)
 \right|^2ds\ge
 \lambda_K h
 \sum_{|\mathtt w|_{\rm deg}\le K}
 |c_{\mathtt w}^{(j)}|^2h^{|\mathtt w|_{\rm deg}}\ge
 \frac{\lambda_Kh^{K+1}}{R^{2}}
\end{align*}
over $A_j$. On the other hand,  Gaussian hypercontractivity \cite{Bogachev1998}
gives
\begin{align*}
 \mathbb E[|F_c(t_j+hs)|^4\mid\mathcal F_{t_j}]
 \le
 3^{2K}
 \left(\mathbb E[|F_c(t_j+hs)|^2\mid\mathcal F_{t_j}]\right)^2;
\end{align*}
Minkowski's inequality then yields
\begin{align*}
 \left(\mathbb E[|Z_j|^2\mid\mathcal F_{t_j}]\right)^{1/2}
 &\le
 h\int_0^1
 \mathbb E[|F_c(t_j+hs)|^4\mid\mathcal F_{t_j}]^{1/2}ds\le
 3^K\,
 \mathbb E[Z_j\mid\mathcal F_{t_j}],
\end{align*}
and consequently
\begin{align}\label{eq:cell-second-moment}
 \mathbb E[Z_j^2\mid\mathcal F_{t_j}]
 \le
 3^{2K}
 \bigl(\mathbb E[Z_j\mid\mathcal F_{t_j}]\bigr)^2.
\end{align}
Denoting $Y_j=\mathbf{1}_{A_j^c} + \mathbf{1}_{A_j}\mathbf{1}_{\{Z_j\ge
\frac12\lambda_Kh^{K+1}R^{-2}\}}$, then by Paley--Zygmund, one has 
\begin{align}\label{eq:cell-success}
\begin{split}
 \mathbb E[Y_j\mid\mathcal F_{t_j}] &= \mathbf{1}_{A_j^c} + \mathbf{1}_{A_j}\mathbb P
 \left\{
 Z_j\ge
 \frac12\lambda_Kh^{K+1}R^{-2}
 \,\middle|\,
 \mathcal F_{t_j}
 \right\}\\
&\ge \mathbf{1}_{A_j^c} + \mathbf{1}_{A_j}\frac14\frac{\left(\mathbb E[Z_j\mid\mathcal F_{t_j}]\right)^2}{\mathbb E[Z_j^2\mid\mathcal F_{t_j}]}
 \ge
 p_K:=3^{-2(K+1)}.
\end{split}
\end{align}
The conditional Chernoff bound then gives
\begin{align}\label{eq:cell-chernoff}
 \mathbb P\left\{\sum_{j=0}^{M-1}Y_j<\frac12Mp_K\right\}
 \le \exp\{-Mp_K/8\}.
\end{align}
Since $G_R\subset A_j$ for every $j$, no artificial value $Y_j=1$
occurs on $G_R$.
Hence, on $G_R\setminus E_{c,K}$, with $\mathbb P(E_{c,K})\leq \exp\{-Mp_K/8\}$, at least $Mp_K/2$ cells satisfy
\begin{align*}
 Z_j\ge \frac12\lambda_Kh^{K+1}R^{-2}.
\end{align*}
Since $h=M^{-1}$ and $\|F_c\|_{L^2(0,1)}^2=\sum_{j=0}^{M-1}Z_j$, it follows that
\begin{align}\label{eq:fixed-vector-frame}
 G_R\setminus E_{c,K}\quad\Longrightarrow\quad
 \|F_c\|_{L^2(0,1)}^2\ge b_K^2,
 \qquad
 b_K^2:=\frac{p_K\lambda_KM^{-K}}{4R^2}.
\end{align}

\emph{Step 3: uniform bound for all coefficients and word cutoffs.}
We next make the estimate uniform over the coefficient vector. On $G_R$,
\begin{align*}
 \|F_c\|_{L^2(0,1)}\le R\|c\|_{\ell^2},
 \qquad
 \|F_c-F_{\tilde c}\|_{L^2(0,1)}\le R\|c-\tilde c\|_{\ell^2}.
\end{align*}
Let $\mathcal N_K=\{c_i\}$ be a deterministic $b_K/(2R)$-net of the Euclidean unit sphere in $\mathbb R^{D_K}$ with
\begin{align*}
 \#\mathcal N_K\le \left(1+\frac{4R}{b_K}\right)^{D_K}.
\end{align*}
Applying \eqref{eq:fixed-vector-frame} to every $c_i\in\mathcal N_K$, the probability that some net vector fails is 
\begin{align}\label{eq:frame-net-failure}
 \mathbb P\left(\bigcup_i E_{c_i,K} \right)\leq \exp\left\{
 D_K\log\left(1+\frac{4R}{b_K}\right)-\frac18Mp_K
 \right\}.
\end{align}
Outside this event, $\|F_{c_i}\|_{L^2}\ge b_K$ for every $c_i\in\mathcal N_K$. Hence, for any unit vector $c$, choosing $c_i\in\mathcal N_K$ with $\|c-c_i\|_{\ell^2}\le b_K/(2R)$ gives
\begin{align*}
 \|F_c\|_{L^2}\ge \|F_{c_i}\|_{L^2}-\|F_c-F_{c_i}\|_{L^2}
 \ge b_K-R\|c-c_i\|_{\ell^2}\ge \frac12b_K.
\end{align*}
Choose $M=M_K=\left\lceil e^{a(K+1)}\right\rceil$ for $a>0$ that will be determined below. 
Since $p_K=3^{-2(K+1)}$ and $\lambda_K=e^{-C_mK\log(K+2)}$ as in the beginning of Step 2, the definition of $b_K$ in \eqref{eq:fixed-vector-frame} gives
\begin{align}\label{eq:bK-log-bound}
 \log b_K^{-1}
 \le C_{m,R}K\log(K+2)+\frac12K\log M_K
 \le C_{m,R,a}K^2.
\end{align}
Together with $D_K\le(m+2)^{K+1}$, this yields
\begin{align*}
 D_K\log\left(1+\frac{4R}{b_K}\right)
 \le C_{m,R,a}K^2(m+2)^{K+1},
\end{align*}
whereas $M_Kp_K\ge C\,\exp\{(a-2\log3)K\}$. 
Choose $a>2\log3+\log(m+2)$. Then the negative term in \eqref{eq:frame-net-failure} dominates the positive one as $K\to\infty$. Increasing $a$ if necessary to cover the finitely many remaining values of $K$, we obtain from \eqref{eq:frame-net-failure} that 
\begin{align*}
 \mathbb P(E_K)\le 2^{-K-4},\quad \text{ where } E_K= \bigcup_i E_{c_i,K} , \,  K\ge1.
\end{align*}
Since $\mathbb P(G_R)\ge3/4$ and $\sum_{K\ge1}2^{-K-4}<1/4$, one has
\begin{align*}
\mathbb P(G_*)>1/2, \quad \text{where }  G_*=G_R\setminus\bigcup_{K\ge1}E_K.
\end{align*}
On $G_*$, for every $K\ge1$ and every coefficient vector $c$, one has $\|F_c\|_{L^2(0,1)}\ge \frac12b_K\|c\|_{\ell^2}$. 
By \eqref{eq:bK-log-bound}, we obtain 
\begin{align*}
 \int_0^1\left|
 \sum_{|\mathtt w|_{\rm deg}\le K}c_{\mathtt w}I_{\mathtt w}(0,t)
 \right|^2dt
 \ge e^{-C_{m,R}K^2}\sum_{|\mathtt w|_{\rm deg}\le K}|c_{\mathtt w}|^2
\end{align*}
simultaneously for every $K\ge1$ and every coefficient vector $c$ on $G_*$. Finally, inner regularity of Wiener measure provides the desired compact subset
$G\subset G_*$ with $\mathbb P(G)>1/4$. 
\end{proof}

From the proof we see that the failure probability in \eqref{eq:frame-net-failure} is actually double exponentially small with respect to the word cutoff $K$. This allows us to obtain the following nondegeneracy estimates simultaneously on exponentially small time intervals, which is crucial in deriving the polynomial cost $C_{\delta}\lesssim \delta^{-b}$  in quantitative Malliavin estimates. 

\begin{proposition}
\label{prop:interval-frames}
Fix $n_0\geq 1$, a time interval $I=[0,T_0]$ and integers $L_N$ satisfying $N\le L_N\le n_0N$. 
For $a>0$, partition $I$ into $n_N=\lceil T_0e^{aL_N}\rceil$
equal subintervals $I_{N,j}=[s_{N,j},s_{N,j}+\tau_N]$. 
Then, for every $\varepsilon>0$, there exist $N_0\ge1$ and an event
$\Omega_{0}$ with $\mathbb P(\Omega_{0})\ge1-\varepsilon$ 
such that, simultaneously for every $N\ge N_0$, every interval
$I_{N,j}$, and every family of real coefficients
$(c_{\mathtt w})_{|\mathtt w|_{\rm deg}\le4L_N}$,
\begin{align}\label{eq:cluster-frame}
 \int_0^1
 \left|
 \sum_{|\mathtt w|_{\rm deg}\le4L_N}
 c_{\mathtt w}S_{\mathtt w}^{N,j}(t)
 \right|^2dt
 \ge
 \exp\{-C_*L_N^2-C_aL_N\}
 \sum_{|\mathtt w|_{\rm deg}\le4L_N}|c_{\mathtt w}|^2.
\end{align}
Here $S^{N,j}$ denotes the Stratonovich signature of the Brownian
increment on $I_{N,j}$ after rescaling the interval to unit time and
the Brownian increment to unit variance. The constant $C_*$ depends
only on $m$ and $n_0$, whereas $C_a$ and $N_0$ may depend on
$a,T_0,\varepsilon$. 
\end{proposition}

\begin{proof}
For each interval $I_{N,j}$, rescale its Brownian increment to $[0,1]$ with unit variance, and denote the corresponding It\^o signature by $\mathbf I^{N,j}$. By Brownian scaling, $\mathbf I^{N,j}$ has the same law as the unit time signature in \Cref{prop:pathwise-frame}. Set $R_N=R_*n_N(N+1)^2$ and define
\begin{align*}
G_{N,j}=\left\{\sup_{t\le1}\left(\|\mathbf I^{N,j}(0,t)\|_{\mathcal T_1}+\|\mathbf I^{N,j}(0,t)^{-1}\|_{\mathcal T_1}\right)\le R_N\right\}.
\end{align*}
By \eqref{eq:signature-integrability} and Markov's inequality, $\mathbb P(G_{N,j}^c)\le\frac{C_m}{R_N}$ and therefore 
\begin{align}\label{eq:GNjc}
\sum_{N\ge1}\sum_{j=0}^{n_N-1}\mathbb P(G_{N,j}^c)
\le\frac{C_m}{R_*}\sum_{N\ge1}(N+1)^{-2}<\frac{\varepsilon}{2}
\end{align}
by choosing $R_*$ large. Put $K=4L_N$ and  $M_N=\lceil e^{bL_N}\rceil$, where $b=b(m)$ is fixed so that 
\begin{align*}
b>8\log3+4\log(m+2)+2. 
\end{align*}
For such $K$, setting the corresponding $p_K=3^{-2(K+1)}$ and $\lambda_K=e^{-C_mK\log(K+2)}$, 
and repeating the same argument as in Step 2--Step 3 of the proof of \Cref{prop:pathwise-frame}, gives an event $E_{N,j}$ such that, on $G_{N,j}\setminus E_{N,j}$,
\begin{align}\label{eq:interval-ito-frame}
\int_0^1\left|\sum_{|\mathtt w|_{\rm deg}\le K}c_{\mathtt w}I_{\mathtt w}^{N,j}(0,t)\right|^2dt
\ge\frac{b_K^2}{4}
\sum_{|\mathtt w|_{\rm deg}\le K}|c_{\mathtt w}|^2, \qquad
b_K^2=\frac{p_K\lambda_KM_N^{-K}}{4R_N^2}.
\end{align}
for every coefficient vector $c$, while
\begin{align*}
\mathbb P(E_{N,j})
\le\exp\left\{D_K\log\left(1+\frac{4R_N}{b_K}\right)-\frac18M_Np_K\right\}.
\end{align*}
It follows from \eqref{eq:interval-ito-frame} and the dependence of the parameters on $L_N$ that 
\begin{align*}
\log\left(4 b_K^{-2}\right) 
\le C_*L_N^2+C_aL_N,
\end{align*}
with $C_*$ independent of $a$,  whereas 
\begin{align*}
D_K\log\left(1+\frac{4R_N}{b_K}\right)
\le
C_{a,T_0,R_*}L_N^2e^{4\log(m+2)L_N}, \text{ and } M_Np_K\ge C\exp\{(b-8\log3)L_N\}. 
\end{align*}
Hence, for some $C,C'>0$ and all large enough $N$,
\begin{align*}
\mathbb P(E_{N,j})\le\exp\{-Ce^{C'L_N}\}.
\end{align*}
Using $N\le L_N\le n_0N$ and $n_N\le C_{T_0}e^{aL_N}$, and increasing $N_0$ yields 
\begin{align}\label{eq:ENj}
\sum_{N\ge N_0}\sum_{j=0}^{n_N-1}\mathbb P(E_{N,j})
\le C_{T_0}\sum_{N\ge N_0}\exp\{an_0N-ce^{c'N}\}\leq \varepsilon/2.
\end{align}
Finally, let $T_K$ be the It\^o--Stratonovich coordinate map from the proof of \Cref{lem:signature-conventions}. Then
\begin{align*}
\mathbf S_K^{N,j}=T_K\mathbf I_K^{N,j},
\qquad
\|T_K\|+\|T_K^{-1}\|\le e^{C_mK}.
\end{align*}
Thus, for every coefficient vector $c$,
\begin{align*}
\sum_{\mathtt w}c_{\mathtt w}S_{\mathtt w}^{N,j}
=\sum_{\mathtt w}(T_K^*c)_{\mathtt w}I_{\mathtt w}^{N,j},
\qquad
\|T_K^*c\|_{\ell^2}\ge e^{-C_mK}\|c\|_{\ell^2}.
\end{align*}
Applying \eqref{eq:interval-ito-frame} to $T_K^*c$ and enlarging $C_a$ gives \eqref{eq:cluster-frame}. Setting 
\begin{align*}
\Omega_0 = \bigcap_{N\ge N_0,j}\left(G_{N,j}\setminus E_{N,j}\right)
\end{align*}
completes the proof by 
\eqref{eq:GNjc} and \eqref{eq:ENj}. 
\end{proof}

\medskip

\textbf{Competing interests}.
The authors declare that they have no competing interests.

\medskip

\textbf{Use of generative artificial intelligence}.
OpenAI ChatGPT and Anthropic Claude were used to assist with proof auditing and manuscript preparation. The authors take full responsibility for the mathematical content.

\end{document}